\documentclass[11pt]{article}

\usepackage[dvipdf]{graphicx}
\usepackage[latin1]{inputenc}
\usepackage{latexsym}
\usepackage{amsmath,amssymb,amsfonts,amsthm}
\usepackage{xcolor}
\usepackage{mathrsfs}
\usepackage[colorlinks=true,citecolor=red,linkcolor=blue,urlcolor=blue,pdfstartview=FitH]{hyperref}

\newcommand{\nrn}{\rightarrow+\infty}
\newcommand{\xrn}{\xrightarrow}

\newcommand{\ER}{\mathbb {R}}
\newcommand{\R}{\mathbb {R}}
\newcommand{\EN}{\mathbb {N}}

\newcommand{\PE}{\mathbb {P}}
\newcommand{\ES}{\mathbb{E}}

\newcommand{\psg}{( }
\newcommand{\psd}{) }

\newcommand{\cfk}{\mathbf{(C(f,k))}}
\newcommand{\cqfd}{\Box}

\newtheorem{theorem}{\textnormal{\bf{T\scriptsize{HEOREM}}}}
\newtheorem{prop}{\textnormal{\bf{P\scriptsize{ROPOSITION}}}}
\newtheorem{Corollaire}{\textnormal{\bf{C\scriptsize{OROLLARY}}}}
\newtheorem{lemme}{\textnormal{\bf{L\scriptsize{EMMA}}}}

\theoremstyle{definition}
\newtheorem{definition}{\textnormal{\bf{D}\scriptsize{EFINITION}}}
\theoremstyle{remark}

\newtheorem{Remarque}{\textnormal{\bf{R\scriptsize{EMARK}}}}

\author{Vincent Lemaire\footnote{Laboratoire de Probabilit\'es et Mod\`eles Al\'eatoires, UMR 7599, UPMC, Case 188, 4 pl. Jussieu,
F-75252 Paris Cedex 5, France, E-mail: \texttt{vincent.lemaire@upmc.fr}}, Gilles Pag\`es\footnote{ Laboratoire de Probabilit\'es et Mod\`eles Al\'eatoires, UMR 7599, UPMC, Case 188, 4 pl. Jussieu,
F-75252 Paris Cedex 5, France, E-mail: \texttt{gilles.pages@upmc.fr}},
Fabien Panloup\footnote{Institut de Math\'ematiques de Toulouse, Universit\'e Paul Sabatier \& INSA Toulouse, 135, av. de Rangueil, F-31077 Toulouse Cedex 4, France, E-mail: \texttt{fabien.panloup@math.univ-toulouse.fr}}
}

\title{Invariant distribution of duplicated diffusions and application to Richardson-Romberg extrapolation}
\date{November 25, 2013}
\begin{document}
\maketitle

\begin{abstract}
With a view to numerical applications we address the following question: given an ergodic Brownian diffusion with a unique invariant distribution, what are the invariant distributions of the duplicated system consisting of two trajectories? We mainly focus on the interesting case where the two trajectories are driven by the same Brownian path. Under this assumption, we first show that uniqueness of the invariant distribution (weak confluence) of the duplicated system  is essentially always true in the one-dimensional case. In the multidimensional case, we begin by exhibiting explicit counter-examples. Then, we provide  a series of  weak confluence criterions (of integral type) and also of \textit{a.s. pathwise confluence}, depending on the drift and diffusion coefficients through a  \textit{non-infinitesimal Lyapunov exponent}. {As examples, we apply our criterions to some non-trivially confluent settings such as classes of gradient systems with non-convex potentials or diffusions where the confluence is generated by the diffusive component. {We finally establish that the weak confluence property is connected  with an optimal transport problem}.}

As a main  application, we apply our results to the optimization of the Richardson-Romberg extrapolation for the numerical approximation of the invariant measure of the initial ergodic Brownian diffusion.
\end{abstract}


\medskip
\noindent \textit{AMS classification (2000)}: 60G10, 60J60, 65C05, 60F05.
\
\section{Introduction and motivations}
When one discretizes a {stochastic (or not) differential equation (SDE)} by an Euler scheme with step $h$, 
a classical method to reduce the discretization error is the so-called Richardson-Romberg ($RR$) extrapolation {introduced in~\cite{TATU} for diffusion processes}. Roughly speaking, the idea of this method is to introduce a second Euler scheme with step $h/2$ and to choose an appropriate linear  combination of the two  schemes   to cancel the first-order discretization error. Such an idea can be adapted to the long-time setting. More precisely,
when one tries to approximate the invariant distribution of a diffusion by empirical measures based on an Euler scheme (with decreasing step) of the diffusion,
it is also possible to implement the same strategy by introducing a second Euler scheme with half-step {(see~\cite{lemairethese})}. In fact,  tackling    the rate of convergence of such a  procedure   involving a couple of Euler schemes of the same {SDE}  leads to studying the long run behaviour of  the underlying   couple of continuous processes that we will call  \textit{duplicated diffusion}{. When the two solutions only differ by the starting value and are driven by the same Brownian motion, the resulting coupled process is }also known as {\em $2$-point motion} {(terminology coming from the more general theory of stochastic flows, see~\cite{BAX,CAR,KUN})}.
Before being more specific as concerns this motivation, let us now define {precisely what we call} a duplicated diffusion.
 
\noindent Consider the following Brownian diffusion solution to the stochastic differential equation
\begin{equation}\label{EDS}
	(SDE) \quad \equiv \quad dX_t=b(X_t)dt+\sigma(X_t)dW_t,\;X_0=x\!\in \ER^d,
\end{equation}
where $b:\ER^d\to \ER^d$ and $\sigma:\ER^d\to {\cal M}(d,q,\ER)$ ($d\times q$ matrices with real valued entries) are locally Lipschitz continuous with linear growth and $W$ is a  standard $q$-dimensional Brownian motion defined on a filtered probability space $(\Omega,{\cal A}, \PE,({\cal F}_t)_{t\ge0})$ (satisfying the usual conditions). This stochastic differential equation ($SDE$)
has a unique strong solution denoted $X^{x}=(X^{x}_t)_{t\ge 0}$. Let  $\rho\!\in {\cal M}(q,q,\ER)$ be a square matrix with transpose $\rho^*$ such that $I_q-\rho^*\rho$ is non-negative as a symmetric matrix. We consider a filtered probability space, still denoted $(\Omega,{\cal A}, \PE,({\cal F}_t)_{t\ge0})$ on which is defined a $2q$-dimensional standard $({\cal F}_t)$-Brownian motion denoted $(W,\widetilde W)$ so that $W$ and $\widetilde W$ are two {\em independent} $q$-dimensional standard $({\cal F}_t)$-Brownian motions. Then we define   $W^{(\rho)}$ a third    standard $q$-dimensional $({\cal F}_t)$-Brownian motions  by
\begin{equation}\label{eq:wwrho}
W^{(\rho)}= \rho^*W + \sqrt{I_q- \rho^*\rho} \,\widetilde W,
\end{equation}
which clearly satisfies
\[
\langle W^i,W^{(\rho),j}\rangle_t =\rho_{ij}\,t,\; t\ge 0
\]
(the square root should be understood in the set of symmetric non-negative matrices). The duplicated diffusion or ``duplicated stochastic differential system" ($DSDS$) is then defined by
\begin{equation}\label{eds_duplicated}
	(DSDS) \quad \equiv \quad \begin{cases}dX_t=b(X_t)dt+\sigma(X_t)\,dW_t,&X_0=x_1\!\in \ER^d,\\
dX^{(\rho)}_t =b(X^{(\rho)}_t)dt+\sigma(X^{(\rho)}_t)\,dW^{(\rho)}_t,&X^{(\rho)}_0=x_2\!\in \ER^d.
\end{cases}
\end{equation} 
Under the previous assumptions on $b$ and $\sigma$, \eqref{eds_duplicated} has a unique (strong) solution. 
Then both $(X^{x}_t)_{t\ge 0}$ and $(X^{x_1}_t, X^{(\rho),x_2}_t)_{t\ge 0}$ are homogeneous Markov processes with transition (Feller) semi-groups, denoted $(P_t(x,dy))_{t\ge 0}$ and 
$\big(Q_t^{(\rho)}((x_1,x_2),(dy_1,dy_2))\big)_{t\ge 0}$ respectively, and defined on test Borel functions $f:\ER^d \to \ER$ and $g:\ER^{d}\times\ER^{d} \to \R$, by
\[
P_t(f)(x)=\ES f(X^{x}_t)\quad \mbox{ and }\quad Q_t^{(\rho)}(g)(x_1,x_2) = \ES\, g(X^{x_1}_t,X^{(\rho),x_2}_t).
\]

  We will   assume throughout the paper  that the original diffusion $X^{x}$ has an unique invariant distribution denoted $\nu$ $i.e.$ satisfying $\nu P_t=\nu$ for every $t\!\in \ER_+$. 
The first part of the   paper is devoted  to determining what are the  invariant measures of  $(Q_t^{(\rho)})_{t\ge0}$ (if any) depending on the correlation matrix $\rho$. Thus, if $\rho=0$, it is clear that   $\nu\otimes \nu$ is invariant for $Q^{(0)}$ and if $\rho = I_q$ so is  $\nu_{\Delta} = \nu \circ (x\mapsto (x,x))^{-1}$, but are they the only ones?  To be more precise, we want to establish  easily verifiable criterions on $b$ and $\sigma$ which ensure that $\nu_{\Delta}$ is the {\em unique} invariant distribution of $(DSDS)$.  In the sequel, we will denote by $\mu$ a generic invariant measure of $Q^{(\rho)}$. Now, we present the problem in more details (including references to the literature).

\smallskip
\noindent $\rhd$ {\em Existence of an invariant distribution for $(Q^{(\rho)}_t)_{t\ge 0}$.} First, the family of probability measures $(\mu_t^{(\rho)})_{t>0}$ defined on $(\ER^d\times\ER^d, {\cal B}or(\ER^d)^{\otimes2})$ by
\begin{equation}\label{loiind1}
	\mu_t^{(\rho)} = \frac 1t \int_0^t \nu^{\otimes 2} (dx_1,dx_2) Q_s^{(\rho)}((x_1,x_2),(dy_1,dy_2)) ds 
\end{equation}
is tight since both  its marginals on $\ER^d$ are equal to $\nu$. Furthermore, the semi-group $(Q^{(\rho)}_t)_{t\ge 0}$ being  Feller, one easily shows that any of its limiting distributions $\mu^{(\rho)}$ as $t\to \infty$ is an invariant distribution for $(Q^{(\rho)}_t)_{t\ge 0}$ such that  $\mu^{(\rho)}(dx\times \R^d) = \mu^{(\rho)}(\R^d\times dx)= \nu(dx)$.  Also note that, if uniqueness fails and $(P_t)_{t\ge 0}$ has two distinct invariant distributions $\nu$ and $\nu'$, a straightforward adaptation of the above (sketch of) proof shows that $(Q^{(\rho)}_t)_{t\ge 0}$ has (at least) an invariant distribution with marginals $(\nu,\nu')$ and another with $(\nu',\nu)$ as marginals.

\smallskip
\noindent $\rhd$ {\em Uniqueness of the invariant distribution of $(Q^{(\rho)}_t)_{t\ge 0}$.} It is clear that in full generality the couple $(X,X^{(\rho)})$ may admit several invariant distributions even if $X$ has only one such distribution. So is the case when  $\sigma\equiv 0$ if the flow $\Phi(x,t)$ of the $ODE \equiv \dot x =b(x)$ has $0$ as a unique repulsive equilibrium and a unique  invariant distribution $\nu$ on $\R^d\setminus \{0\}$. Then both distributions $\nu^{\otimes 2}$ and $\nu_{\Delta}$ (defined as above)  on $(\ER^d\setminus\{0\})^2$ are invariant and if $\nu$ is not reduced to a Dirac mass (think $e.g.$ to a $2$-dimensional ODE with a limit cycle around $0$) $(DSDS)$ has at least two invariant distributions.

\smallskip In the case ($\sigma\not \equiv 0$) the situation is more involved and depends on the correlation structure $\rho$ between the two Brownian motions $W$ and $W^{(\rho)}$. 
The diffusion matrix $\Sigma(X_t^{x_1},X_t^{(\rho), x_2})$ of the couple $(X^{x_1},X^{(\rho), x_2})$ at time $t>0$  is  given by any continuous 
 solution  to the equation 
\[
\Sigma(\xi_1,\xi_2)\Sigma(\xi_1,\xi_2)^*= \left [\begin{array}{cc}\sigma\sigma^*(\xi_1 )&  \sigma(\xi_1 )\rho  \sigma^*(\xi_2) \\
\sigma(\xi_2) \rho^*\sigma^*(\xi_1)   &     \sigma \sigma^*(\xi_2)    \end{array}\right]
\]
($e.g.$ the square root in the symmetric non-negative matrices or the Choleski transform\dots).

\smallskip
First, note that if  $I_q-\rho^*\rho$ is {\em positive definite} as a symmetric matrix, it is straightforward that ellipticity or uniform ellipticity of $\sigma\sigma^*$(when $q\ge d$)  for $X^{x}$  is transferred to $\Sigma(X_t^{x_1},X_t^{(\rho), x_2})\Sigma(X_t^{x_1},X_t^{(\rho), x_2})^*$ for the couple $(X^{x_1},X^{(\rho), x_2})$. Now, uniform ellipticity, combined with  standard regularity and growth/boundedness  assumption on the coefficients $b$, $\sigma$ and their partial derivatives, classically implies the existence  for every $t>0$  of a (strictly) positive probability density  $p_t(x, y)$ for $X_t^x$. These additional conditions are automatically satisfied by  the ``duplicated coefficients" of $(DSDS)$. At this stage, it is classical background that  any homogeneous Markov process whose transition has a (strictly) positive density for every $t>0$  has at most one invariant distribution (if any).  Consequently, under these standard assumptions on $b$ and $\sigma$ which ensure uniqueness of the invariant distribution  $\nu$ for $X$, we get uniqueness for  the ``duplicated" diffusion process $(X,X^{(\rho)})$  as well. 

The {\em hypo-elliptic}  case also implies the existence of a density for $X^x_t$ and the uniqueness of the invariant distribution under {\em controllability assumptions} on a companion differential system of  the  SDE. {This property} can  also be transferred to $(DSDS)$, although the proof becomes significantly less straightforward {than above} (see Appendix~B for a precise statement and  a detailed proof). 

We now consider one of the main problems of this paper: the degenerate case $\rho=I_q$. This corresponds to $W^{(\rho)}=W$ so that $X^{(\rho),x_2}= X^{x_2}$,  $i.e.$ $(DSDS)$ is the equation of the $2$-point motion in the sense of \cite{KUN} section 4.2 and \cite{Harris}. This $2$-point motion has been extensively investigated (see~\cite{CAR}) from an ergodic viewpoint, especially when the underlying diffusion, or more generally the stochastic flow $\Phi(\omega, x,t)$ lives on a (smooth) compact Riemannian manifold $M$. When this flow is smooth enough in $x$, the long run behaviour of such a flow (under its steady regime) can be classified  owing to its Lyapunov spectrum. For what we are concerned with, this classification is based on the top Lyapunov exponent defined by
\[
\lambda_{1} := \limsup_{t\to +\infty} \frac 1t \log\|D_x\Phi(x,t)\|
\] 
{where $\|D_x\Phi(x,t)\|$ denotes the  operator norm of the differential (tangent) of the flow. In this compact setting and when the top Lyapunov exponent is positive,  the long run behaviour of the two-point on $M^2\backslash\Delta$ has been deeply investigated in \cite{BAX} (see also \cite{dolgopyat} for further results in this direction). Such an assumption implies that $\Delta$ is somewhat repulsive.

	Here, we are in fact concerned with the opposite case. Our aim is to identify natural assumptions under which the invariant distribution of the $2$-point motion is unique (hence equal to $\nu_{\Delta}$). It seems clear that these conditions  should    in some 
sense imply that the paths  cluster asymptotically {either in a pathwise or in a statistical sense}. 
When  $\lambda_{1}<0$, a local form of such a clustering  has been obtained  
in~\cite{CAR} (see Proposition~2.3.3) : it is shown that  at a given point
``asymptotic clustering'' holds  with an arbitrarily high probability, 
provided  the starting points are close enough. However, this result seems to 
be not sufficient to imply uniqueness of the invariant distribution for the 
two-point motion and is still in a compact setting.

\smallskip 
In Sections~\ref{sec: d=1} and~\ref{sec:dge2}, we provide precise answers under verifiable conditions
on the coefficients $b$ and $\sigma$ of the original $\ER^d$-valued diffusion, not assumed to be smooth. More precisely, we show in Section~\ref{sec: d=1}
that in the one-dimensional case, uniqueness of $\nu_\Delta$ is almost always true (as soon as $(SDE)$ has a unique invariant distribution) and that under some slightly more constraining conditions, the diffusion is \textit{pathwise confluent} ($i.e.$  pathwise asymptotic clustering holds). This second result slightly extends by a different method a result by Has'minskii in~\cite{Has'm}.\\
 Section~\ref{sec:dge2}  is devoted to the multidimensional framework. We first provide a simple counter-example where uniqueness of $\nu_\Delta$ does not hold. Then, we obtain some sharp criterions for uniqueness. We begin by a general uniqueness result (for $\nu_\Delta$)
 (Theorem \ref{thm:echellemultidim}) involving in an Euclidean framework (induced by a positive definite matrix $S$ and its norm $|\,.\,|_{_S}$) a \textit{pseudo-scale} function $f_\theta$ designed from  a non-negative continuous function $\theta:\ER_+\rightarrow\ER$. Basically both uniqueness and pathwise confluence  follow from 

 conditions involving the coefficients of the diffusion $b$ and $\sigma$,  $S$ and $\theta$, combined with  a requested  behavior of the pseudo-scale function  at $0^+$. The  main ingredient of the proof is Birkhoff's ergodic Theorem applied to the one-dimensional It\^o process $f_\theta(|X_t^{x_1}-X_t^{x_2}|^2_{_S})$. Using additional martingale arguments, we also establish that the asymptotic pathwise confluence holds under  slightly more stringent conditions.\\ 
Then, in Subsection~\ref{subsec:NILS}, we draw a series of corollaries of Theorem~\ref{thm:echellemultidim} (illustrated on few examples) which  highlight easily verifiable conditions.
To this end we introduce  a function $\Lambda_{_S}:\ER^{d}\times\ER^{d} \setminus\Delta_{\ER^{d}\times\ER^{d} }\rightarrow\ER$ called  {\em Non-Infinitesimal $S$-Lyapunov} (NILS) exponent  
 defined  for every  $x,\, y\!\in \R^d, \; x\neq y$ by
\begin{equation}\label{defNILS}
	\hskip -0.5 cm\Lambda_{_S}(x, y) = \frac{(b(x)-b(y)|x-y)_{_S}}{|x-y|^2_{_S}}	+  \frac{1}{2}\frac{\|\sigma(x)-\sigma(y)\|^2_{_S}}{|x-y|^2_{_S}} - \Biggl(
	\Big| \frac{(\sigma^*(x)-\sigma^*(y))S(x-y)}{|x-y|^2_{_S}}\Big|^2\Biggr).
\end{equation} 

In particular we show (see Corollary~\ref{Cor:lyap}) that if, {for every probability measure $m$ on $^c \Delta_{\ER^{d}\times\ER^{d} }$ such that $m(dx\times\ER^d)=m(\ER^d\times dy)=\nu$,  
$$
\int_{\R^d\times\R^d} \Lambda_{_S}(x,y) m(dx,dy)<0
$$ 
then $\nu_{\Delta}$ is unique and if furthermore $\Lambda_{_S}\le -c_0<0$ on a  uniform stripe around the diagonal $\Delta_{\ER^{d}\times\ER^{d} }$, then pathwise confluence holds true.}
Moreover, under a  \textit{directional ellipticity} condition on $\sigma$, we show that the negativity  of $\Lambda_{_S}$ ({at least in an integrated sense}) can be localized near the diagonal (see Subsection~\ref{subsec:NILS}   for details).   A differential version of the criterion is established when $b$ and $\sigma$ are smooth (see Corollary~\ref{Cor:lypadiff}).

Note that these criterions obtained in the case $\rho = I_q$ can be  extended to 
the (last) case $\rho^* \rho = I_q$  using that $W^{(\rho)}= \rho W$ is still a standard B.M. (think to $\rho=-1$ when $d=1$). For the sake of simplicity (and since it is of little interest for the practical implementation of 
the Richardson-Romberg extrapolation), we will not consider this case in the paper.

{Then, we give some examples and provide an application to gradient systems ($b=-\nabla U$ and a constant $\sigma$ function)}. In particular, we obtain that our criterions can be applied to some situations where the potential is not convex. More precisely, we prove that for a large class of non-convex potentials,  super-quadratic at infinity , the $2$-point motion is weakly confluent if the diffusive component $\sigma$ is sufficiently large. {Furthermore, in the particular case $U(x)=(|x|^2-1)^2$, we prove that the result is true for every $\sigma>0$.}

{We end the first part of the paper by a connection with optimal transport. More precisely, we show that, up to a slight strengthening of the condition on the \textit{Integrated} NILS, the weak confluence property can be connected with an optimal transport problem.}

\smallskip {The second part of the paper (Section 4) is devoted to a first attempt in a long run  ergodic setting to combine the Richardson-Romberg extrapolation with a control of the variance of this procedure (see~\cite{PAG} in a finite horizon framework). To this end we consider two Euler schemes with decreasing steps $\gamma_n$ and $\tilde \gamma_n$ satisfying $\tilde \gamma_{2n-1}=\tilde\gamma_{2n}= \gamma_n/2$ and $\rho$-correlated Brownian motion increments. We show that the optimal efficiency of the Richardson-Romberg extrapolation in this framework  is obtained when $\rho=I_q$, at least when  the above uniqueness problem for $\nu_{\Delta}$ is satisfied. To support this claim we establish a Central Limit Theorem whose variance is  analyzed as a function of $\rho$.}

\medskip
\noindent{\sc Notations.} $\bullet$ $|x|=\sqrt{xx^*}$ denotes the canonical Euclidean norm of $x\!\in\ER^d$ ($x^*$ transpose of the column vector $x$).

\noindent $\bullet$   $\|A\|= \sqrt{{\rm Tr}(AA^*)}$  if $A\!\in {\cal M}(d,q,\ER)$ and $A^*$ is the transpose of $A$ (which is but   the canonical Euclidean norm on $\ER^{d^2}$).

\noindent $\bullet$  $\Delta_{\ER^{d}\times\ER^{d} }= \{(x,x),\; x\!\in \ER^d\}$ denotes the diagonal of $\ER^{d}\times\ER^{d} $.

\noindent $\bullet$ ${\cal S}(d,\R)= \{S\!\in  {\cal M}(d,d,\ER),\; S^*\!=\!S\}$, ${\cal S}^+(d,\R)$ the subset of ${\cal S}(d,\R)$ of non-negative matrices, ${\cal S}^{++}(d,\R)$ denotes  the subset of positive definite such matrices and  $\sqrt{S}$ denotes the unique square root of $S\!\in {\cal S}^+(d,\R)$ in ${\cal S}^{++}(d,\ER)$ (which commutes with $S$). $x\otimes y= x y^*=[x_iy_j]\!\in {\cal M}(d,d,\ER)$, $x,\,y\!\in \ER^d$.

\noindent $\bullet$  If $S\in{\cal S}^{++}(d,\ER)$, we denote by $(\,.\,|\,.\,)_{_S}$ and by $|\,.\,|_{_S}$, the induced inner product and norm on $\ER^d$, defined   by $(x|y)_{_S}=(x|Sy)$ and $|x|_{_S}^2=(x|x)_{_S}$ respectively.  Finally, for $A\!\in {\cal M}(d,d,\ER)$, we set
$\| A\|_{_S}^2={\rm Tr}(A^*SA)$.

\noindent $\bullet$  $\displaystyle \mu_n \stackrel{(\ER^d)}{\Longrightarrow}\mu$ denotes the weak convergence of the sequence $(\mu_n)_{n\ge 1}$  of probability measures defined on $(\ER^d, {\cal B}or(\ER^d))$ toward the probability measure $\mu$. ${\cal P}(X,{\cal A})$ denotes the set of probability distributions on $(X,{\cal A})$.

\noindent $\bullet$ For every function $f:\ER^d\to \ER$, define the Lipschitz coefficient of $f$ by  $[f]_{\rm Lip}=\sup_{x\neq y}\frac{|f(x)-f(y)|}{|x-y|}\le +\infty$.


\section{The one-dimensional case}\label{sec: d=1}
We first show that, in the one-dimensional case $d=q=1$,  uniqueness of $\nu$ implies that $\nu_\Delta$, as defined in the introduction, is the unique invariant distribution  of the duplicated diffusion.  The main theorem of this section is Theorem \ref{theoremhasm} which consists of two claims. The first one establishes   this uniqueness claim  using some ergodic-type arguments. Note that we do not require that $\sigma$ never vanishes. The second claim is an asymptotic pathwise confluence property  {for the diffusion} in its own scale, established under some slightly more stringent assumptions involving the scale function $p$, see below. This second result, under slightly less general assumptions,  is originally due to  Has'minskii (see~\cite{Has'm}, Appendix to the English edition, Theorem 2.2, p.308). It is revisited here by different techniques, mainly comparison results for one dimensional diffusions and ergodic arguments. Note that uniqueness of $\nu_\Delta$ can always be retrieved from asymptotic confluence (see Remark \ref{remhasm}).

Before stating the result, let us recall some definitions. We denote by  $M$ the {\em speed measure} of the diffusion classically defined by $\displaystyle M(d\xi)={(\sigma^2p')^{-1}(\xi)}d\xi$, where $p$ is the {\em scale function} defined (up to a constant) by 
\[
p(x) = \int_{x_0}^x d\xi e^{-\int_{x_0}^\xi \frac{2b}{\sigma^2}(u)du}, \; x\!\in \ER.
\]
Obviously, we will consider $p$ only when it makes sense {\em as a finite function} (so is the case if $b/ \sigma^{2}$ is locally integrable on the real line). 
We are now in position  to state the result.
\begin{theorem}\label{theoremhasm} Assume that  $b$ and $\sigma$ are continuous functions on $\ER$ being such that 
strong existence, pathwise uniqueness and the Feller Markov property hold for $(SDE)$ from any $x\!\in \ER$.
Assume furthermore that  there exists $\lambda:\ER_+\rightarrow\ER_+$, strictly increasing,  with $\lambda(0)=0$ and $\int_{0^+}\lambda(u)^{-2} du=+\infty$ such that for all $x,y\in\ER$,
$|\sigma(y)-\sigma(x)|\le \lambda(|x-y|)$.
 Then, the following claims hold true.

\medskip
\noindent  $(a)$ If $(X_t)_{t\ge0}$ admits a unique invariant distribution $\nu$, the distribution $\nu_{\Delta}= \nu\circ (\xi\mapsto (\xi,\xi))^{-1}$ is the unique invariant measure of the duplicated diffusion $(X^{x_1}_t,X^{x_2}_t)_{t\ge 0}$.

\medskip
\noindent $(b)$ (Has'minskii) Assume that { the scale function $p$ is well-defined as a real function on the real line} and that,
$$
\lim_{x\to \pm \infty}p(x)=\pm \infty\quad\mbox{ and }\quad M\mbox{ is finite}.
$$
Then,   $\nu=M/M(\R)$ is the unique invariant distribution of $(X_t)_{t\ge 0}$ and  $(p(X_t))_{t\ge0}$ is pathwise confluent: $\PE$-$a.s.$,  for every $x_1, x_2\! \in \ER$, $p(X^{x_1}_t) - p(X^{x_2}_t)$ tends to 0 when $t\to+\infty$.
\end{theorem}

\smallskip \begin{Remarque}
\label{remhasm}
$\rhd$ The general assumptions  on $b$ and $\sigma$ are obviously  fulfilled whenever  these functions are locally Lipschitz with linear growth.

\smallskip
\noindent $\rhd$  The proofs of both claims are  based on (typically one-dimensional)  comparison arguments. 
This also explains the assumption on $\sigma$ which is a classical sufficient assumption to ensure comparison of solutions, namely, if
$x_1\le x_2$, then $X_t^{x_1}\le X_t^{x_2}$ for every $t\ge 0$ $a.s.$ (see \cite{ikeda-watanabe}).

\smallskip
\noindent $\rhd$  The additional assumptions made in $(b)$ imply the uniqueness of $\nu$ (see the proof below).  The uniqueness of the invariant distribution $\nu_{\Delta}$ for the duplicated diffusion follows by $(a)$. However,  it can also be viewed as a direct consequence of the asymptotic pathwise confluence of $p(X^{x_i}_t)$, $i=1,2$ as $t\rightarrow+\infty$. Actually, if for all $x_1, x_2 \in \ER^d$, $p(X^{x_1}_t) - p(X^{x_2}_t)\xrn{t\nrn}0$ $a.s$, we deduce that for any  invariant distribution $\mu$ of $(X^{x_1},X^{x_2})$ and every $K > 0$ 
$$
\int_{\ER} \Big(|p(x_1)-p(x_2)|\wedge K\Big)\mu(dx_1,dx_2)\le \limsup_{t\to +\infty}\frac{1}{t}\int_0^t \ES_{\mu}\Big(|p(X_s^{x_1})-p(X_s^{x_2})|\wedge K\Big)ds=0.$$
As a consequence, $p(x_1)=p(x_2)$ $\mu(dx_1,dx_2)$-$a.s.$ Since $p$ is an increasing function, it follows that $\mu(\{(x,x),x\!\in\ER\})=1$ and thus that $\mu=\nu_{\Delta}$. 

\smallskip
\noindent $\rhd$  {As mentioned before, $(b)$  slightly extends a result by Has'minskii obtained in~\cite{Has'm}} with different methods and under the additional assumption that $\sigma$ never vanishes (whereas we only need the scale function $p$ to be finite which allows e.g. for the existence of integrable singularities of $\frac{b}{\sigma^2}$). Note however   that the case of an infinite speed measure $M$ (which corresponds to null recurrent diffusions) is also investigated  in~\cite{Has'm}, requiring extra non-periodicity assumptions on $\sigma$.
\end{Remarque}

\noindent{\em Proof.} $(a)$ Throughout the proof we denote by $(X^{x_1}_t,X^{x_2}_t)$  the duplicated diffusion at time $t\ge 0$ and by $(Q_t((x_1,x_2), dy_1,dy_2))_{t\ge 0}$ its Feller Markov semi-group. The set ${\cal I}_{DSDS}$ of  invariant distributions of $(Q_t)_{t\ge0}$ is clearly nonempty, convex and weakly closed. Since any such distribution $\mu$ has $\nu$ as marginals (in the sense $\mu(dx_1\times \ER)= \mu(\ER \times dx_2)=\nu$), the set ${\cal I}_{DSDS}$ is tight and consequently weakly compact in the the topological vector space of signed measures on $(\ER^2, {\cal B}or(\ER^2))$ equipped with the weak topology. As a consequence of the Krein-Millman Theorem, ${\cal I}_{DSDS}$ admits extremal distributions and is the  convex hull of these extremal distributions.

\smallskip Let $\mu$ be such an extremal distribution and consider the following   three subsets of $\ER^2$:
\[
A^+=\{(x_1,x_2), \; x_2 >  x_1\},\; A^-=\{(x_1,x_2), \; x_1> x_2\} \mbox{ and } A_0 =  \{(x,x),\; x\!\in \ER\}= \Delta_{\ER^2}.
\] 
We first want to show that if $\mu(A^+)>0$ then the conditional distribution $\mu^{A^+}$ defined by $\mu^{A^+} = \frac{\mu(\,.\,\cap A^+)}{\mu(A^+)}$ is also an invariant distribution for $(Q_t)_{t\ge0}$.
\smallskip
Under the above assumptions on $b$ and $\sigma$, one derives  from classical comparison  theorems and strong pathwise uniqueness arguments for the solutions of $(SDE)$ (see $e.g.$ \cite{ikeda-watanabe})
that
$$
\forall\,(x_1,x_2)\in {}^c\!A^{+} = \ER^2 \setminus \!A^{+},\quad 
Q_t((x_1,x_2),^c\!A^{+})=1.
$$

We deduce that for every $(x_1,x_2)\in\ER^2$ and $t\ge0$, 
\[
Q_t((x_1,x_2),A^{+})= \PE\big((X^{x_1}_t,X^{x_2}_t)\!\in A^{+}\big) = \mbox{\bf 1}_{A^{+}}(x_1,x_2)\PE(\tau_{x_1,x_2}>t)
\]
where $\tau_{x_1,x_2}=\inf\{t\ge 0, X_{t}^{x_2}\le X_t^{x_1}\}$. The second equality follows from the pathwise uniqueness since no bifurcation can occur. Now, let $\mu\!\in {\cal I}_{DSDS}$. Integrating the above equality and letting $t$ go to infinity implies
\[
\mu(A^+) = \int_{A^+}\mu(dx_1,dx_2)\PE(\tau_{x_1,x_2}=+\infty).
\]
If $\mu(A^+)>0$, then $\mu(dx_1,x_2)$-$a.s.$ $\PE(\tau_{x_1,x_2}=+\infty)=1$ on $A^+$ $i.e.$ $X^{x_2}_t>X^{x_1}_t$ for every $t\ge 0$ $a.s.$.  As a consequence,
 $\mu(dx_1,dx_2)$-$a.s.$, for every $B\in{\cal B}(\ER^{d}\times\ER^{d} )$,
$$
{\bf 1}_{(x_1,x_2)\in A^+} Q_t((x_1,x_2),B)={\bf 1}_{(x_1,x_2)\in A^+}
Q_t((x_1,x_2),B\cap A^+)=Q_t((x_1,x_2),B\cap A^+)
$$ 
where we used again that $Q_t((x_1,x_2), A^+)=0$ if $x_2\le x_1$. 
Then, since $\mu$ is invariant, we deduce from an integration of the above equality that
$$\mu(B\cap A^{+})= \int_{\ER^2} Q_t((x_1,x_2),B){\bf 1}_{(x_1,x_2)\in A^+}\mu(dx_1,dx_2).$$
It follows that if $\mu(A^+)>0$, $\mu^{A^+}$ is invariant.\\
If $\mu(A^+)<1$, one shows likewise that  $\mu^{^cA^+}$ an invariant distribution   for $(Q_t)_{t\ge0}$ as well. 
Then,  if $\mu(A^+)\!\in (0,1)$, then $\mu$ is a convex combination of elements of ${\cal I}_{DSDS}$
\[
\mu = \mu(A^+) \mu^{A^+}+ \mu(^cA^+) \mu^{^cA^+} 
\]
so that $\mu$ cannot be extremal. Finally  $\mu(A^+)=0$ or $1$.

\smallskip
Assume $\mu(A^+)=1$ so that $\mu= \mu(.\cap A^+)$. This implies that $X^1_0>X^2_0$ $\PE_{\mu}$-$a.s.$. But $\mu$ being invariant, both its marginals are $\nu$ $i.e.$ $X^1_0$ and $X^2_0$ are $\nu$-distributed. This yields a contradiction. Indeed, let $\varphi$ be a bounded increasing positive function. For instance, set $\varphi(u):= 1+\frac{u}{\sqrt{u^2+1}}$, $u\!\in \ER$. Then, $\ES[\varphi(X_0^1)-\varphi(X_0^2)]>0$ since $X_0^1>X_0^2$ $\PE_{\mu}$-$a.s.$ but we also have
$\ES[\varphi(X_0^1)-\varphi(X_0^2)]=0$ since $X_0^1$ and $X_0^2$ have the same distribution. This contradiction implies 
that $\mu(A_+)=0$.
\smallskip

One shows likewise that $\mu(A^-)=0$ if $\mu$ is an  extremal measure. Finally any extremal distribution of ${\cal I}_{DSDS}$ is supported by $A_0=\Delta_{\ER^2}$. Given the fact that the marginals of $\mu$ are $\nu$ this implies that $\mu = \nu_{\Delta}= \nu \circ (x\mapsto (x,x))^{-1}$ which in turn implies that ${\cal I}_{DSDS} = \{\nu_{\Delta}\}$.

\medskip
\noindent $(b)$ Since the speed measure $M$  is finite and $\sigma$ never vanishes, the distribution $\nu(d\xi)= M(d\xi)/M(\ER)$  is the unique invariant measure of the diffusion. Thus, by $(a)$, we also have the uniqueness of the invariant distribution for the duplicated diffusion.Let $x_1$, $x_2\!\in \ER$.
If $x_1>x_2$ then $X^{x_1}_t\ge X^{x_2}_t$, still by a comparison argument,  and $p(X^{x_1}_t)\ge p(X^{x_2}_t)$  since $p$ is increasing.
Consequently $M^{x_1,x_2}_t= p(X^{x_1}_t)-p(X^{x_2}_t)$, $t\ge 0$, is a non-negative continuous local martingale, hence $\PE$-$a.s.$ converging toward  a finite random limit $\ell^{x_1,x_2}_{\infty} \ge 0$. One proceeds likewise when $x_1<x_2$ (with $\ell^{x_1,x_2}_{\infty}\le 0$). When $x_1=x_2$, $M_t= \ell^{x_1,x_2}_{\infty}\equiv 0$. 
The aim is now to show that $\ell^{x_1,x_2}_{\infty}= 0$ $a.s.$ To this end, we introduce
$$\mu_t(dy_1,dy_2):=\frac{1}{t}\int_0^t Q_s((x_1,x_2),dy_1,dy_2) ds,\qquad (x_1,x_2)\in\ER^{d}\times\ER^{d} $$ 
and we want to check that for every $(x_1,x_2)\in\ER^{d}\times\ER^{d} $, $(\mu_t(dy_1,dy_2))_{t\ge1}$ converges weakly to $\nu_{\Delta}$.
Owing to the uniqueness of $\nu_\Delta$ established in $(a)$ and to the fact that any weak limiting distribution  of $(\mu_t(dy_1,dy_2))_{t\ge1}$ is always  invariant (by construction), it is enough to prove that  $(\mu_t(dy_1,dy_2))_{t\ge1}$ is tight. Since the tightness of a sequence of probability measures  defined on a product space is clearly equivalent to that of its first and second marginals, it is here enough to prove the tightness of  $(t^{-1}\int_0^{t} P_s(x_0,dy) ds)_{t\ge1}$ for any $x_0\in\ER$.\\
Let $x_0\in\ER$. Owing to the comparison theorems, we have for all $t\ge0$ and $M\in\ER$, $P_t(x_0,[M,+\infty))  \le  P_t(x,[M,+\infty))$ if $x\ge x_0$ and $   P_t(x_0,(-\infty,M])  \le P_t(x,(-\infty,M])$ if $x_0\ge x$.  Since  $\nu$ is invariant and equivalent to the Lebesgue measure, we deduce that
$$
P_t(x_0,[M,+\infty))  \le \frac{\nu([M,+\infty))}{\nu([x_0,+\infty))}\quad\textnormal{and}\quad  P_t(x_0,(-\infty,M))  \le \frac{\nu((-\infty,M))}{\nu((-\infty,x_0])}.
$$
The tightness of $(P_t(x_0,dy))_{t\ge1}$ follows (from that of $\nu$) and we derive from what preceeds that
$$
\forall \,(x_1,x_2)\in\ER^{d}\times\ER^{d} ,\quad \frac{1}{t}\int_0^t Q_s((x_1,x_2),dy_1,dy_2) ds\stackrel{(\ER^d)}{\Longrightarrow}\nu_{\Delta}(dy_1,dy_2).
$$
Now, note that  for every $L\!\in \EN$, the function $g_{_L}: (y_1,y_2)\mapsto |p(y_1)-p(y_2)|\wedge L$ is  continuous and bounded. Hence  by C\'esaro's Theorem, we have that
\[
\frac 1t \int_0^t Q_s(g_{_L})(x_1,x_2)ds = \frac 1t \int_0^t \ES\, g_{_L}(X^{x_1}_s,X^{x_2})ds \longrightarrow \ES \,( |\ell^{x_1,x_2}_\infty|\wedge L )
\]
whereas, by the above weak convergence of $(\mu_t(dy_1,dy_2))_{t\ge1}$, we get 
\[
\frac 1t \int_0^t Q_s(g_{_L})(x_1,x_2)ds \longrightarrow \int_{\ER^d} g_{_L}(y_1,y_2)\nu_{\Delta}(dy_1,dy_2) = 0\quad \mbox{ as }\quad t\to +\infty
\]
since $g_{_L}$ is identically $0$ on $\Delta_{\ER^{d}\times\ER^{d} }$. It follows, by letting $L$ go to infinity, that
\[
\ES \,|\ell^{x_1,x_2}_\infty| = 0.
\]
This implies $\ell^{x_1,x_2}_{\infty} = 0$ $\PE$-$a.s.$ which in turn implies that
\[
\PE \mbox{-}a.s. \quad p(X^{x_1}_t)-p(X^{x_2}_t)\longrightarrow 0\quad \mbox{ as }\quad t\to +\infty.\qquad
\]
{Finally, it remains to prove that we can exchange the quantifiers, $i.e.$ that  $\PE \mbox{-}a.s.$,
 $p(X^{x_1}_t)-p(X^{x_2}_t)\longrightarrow 0$ for every $x_1$, $x_2$. Assume that $x_1\ge x_2$. Again by the comparison theorem and the fact that $p$ increases, we have $0\le p(X_t^{x_1})-p(X_t^{x_2})\le p(X_t^{\lfloor x_1\rfloor+1})-p(X_t^{\lfloor x_2\rfloor})$. This means that we can come down to a countable set of starting points. $\hfill \cqfd$}
 
  \medskip
\noindent In the continuity of the second part of Theorem \ref{theoremhasm}$(b)$, it is natural to wonder whether a one-dimensional diffusion is asymptotically confluent, $i.e.$ when for all $x_1,x_2\in\ER$, $X_t^{x_1}-X_t^{x_2}$ tends to $0$ $a.s$ as $t\rightarrow+\infty$.
In the following corollary, we show that such property holds in a quite general setting.
\begin{Corollaire}\label{corhasm} $(a)$ Assume  the hypothesis of Theorem \ref{theoremhasm}$(b)$ hold. If furthermore,
\[
\sigma \mbox{ never vanishes $\quad$ and }\quad \limsup_{|x|\to +\infty} \int_0^x\frac{b}{\sigma^2}(\xi)d\xi<+\infty
\]
then, {$\PE \mbox{-}a.s.$, for every $x_1,\,x_2\!\in \R$},
$$
 \quad X^{x_1}_t-X^{x_2}_t\longrightarrow 0\quad \mbox{ as }\quad t\to +\infty.
$$
\noindent $(b)$ The above condition is in particular satisfied if there exists $M>0$ such that
$$
|x|>M\Longrightarrow {\rm sign}(x) b(x)\le0.
$$
\end{Corollaire}

\begin{proof} $(a)$ Under the assumptions of the theorem, $p$ is continuously differentiable on $\ER$ and 
\[
p'(x) =  e^{-\int_{x_0}^x \frac{2b}{\sigma^2}(u)du}, \; x\!\in \ER.
\]
Then it is clear that $\displaystyle p'_{\inf} = \inf_{x\!\in \ER}p'(x)>0$ iff $\displaystyle \limsup_{|x|\to +\infty}\int_{x_0}^x \frac{2b}{\sigma^2}(\xi)d\xi<+\infty$.
By the fundamental theorem of calculus, we know that,
$$
|X_t^{x_1}-X_t^{x_2}|\le \frac{1}{p'_{\min}}|p(X_t^{x_1})-p(X_t^{x_2})|
$$
and the result follows from Theorem \ref{theoremhasm}$(b)$.

\smallskip
\noindent $(b)$ Since $\sigma$ never vanishes, $p''$ is well-defined and for every $x\!\in\ER$, $p''(x)=-\frac{2b(x) p'(x)}{\sigma^2(x)}$. Using that $p'$ is positive, we deduce from the assumptions that
\begin{equation*}
\exists M>0\quad \textnormal{such that}\quad \begin{cases} p''(x)\ge 0,&  x\ge M\\
p''(x)\le 0,& x\le -M.
\end{cases}
\end{equation*}
Now, $p'$ being continuous, it follows that $p'$ attains  a positive minimum   $p'_{\min}>0$. 
\end{proof}

\noindent {\sc \bf Examples.} {\bf 1.} Let $U$ be a positive a twice differentiable function such that $\displaystyle \lim_{|x|\rightarrow+\infty} U(x)=+\infty$  and consider the one-dimensional Kolmogorov equation $dX_t=-U'(X_t)dt+\sigma dW_t$ 
with $\sigma>0$. Then, 
$$
\liminf_{|x|\rightarrow+\infty} x U'(x)>\frac{\sigma^2}{2}\quad \Longrightarrow\;\quad \; X_t^x-X_t^y\xrn{t\nrn}0\; a.s.
$$
Note that in particular, this result holds true even if   $U$ has several local minimas. 

\smallskip
\noindent 
{\bf 2.} Let $\sigma:\R\to (0,+\infty)$ be a   locally Lipschitz continuous  function with linear growth so that the SDE  
$$
dX_t = \sigma(X_t)dW_t 
$$ 
defines a (Markov)  flow $(X^x_t)_{t\ge 0}$ of local martingales. If $\frac{1}{\sigma}\!\in L^2(\R, {\cal B}or(\R),\lambda)$ then there exists a unique invariant measure $\nu(d\xi)=c_{\sigma}\frac{d\xi}{\sigma^2(\xi)}$ and $(X^{x_i}_t)_{t\ge 0}$, $i=1,2$ is pathwise confluent (in the sense of Theorem~\ref{theoremhasm}$(b)$) since $p(x)=x$. Note that  the linear growth assumption cannot be significantly relaxed since a stationary process cannot be  a true martingale which in turn implies that $\nu$ has no (finite) first moment.

\section{The multidimensional case}\label{sec:dge2}
In this section, we begin by an example of a multidimensional Brownian diffusion $(X^{x_1},X^{x_2})$ for which $\nu_\Delta$ (image of $\nu$ on the diagonal) is not the only one invariant distribution. Thus, Theorem \ref{theoremhasm} is specific to the case $d = 1$ and we can not hope to get a similar result for the general case $d \ge 2$. It is of course closely related to the classification of two-point motion on smooth compact Riemannian manifolds since the unit circle will turn out to be a uniform attractor of the diffusion.

\subsection{Counterexample in $2$-dimension}
Roughly speaking, saying that $\nu_\Delta$ is the only one invariant distribution means in a sense that $X_t^x-X_t^y$ has a tendency to converge towards $0$ when $t\rightarrow+\infty$. Thus, the idea in the counterexample below is to build a ``turning'' two-dimensional ergodic process where the angular difference between the two coordinates does not depend on $t$. Such a construction leads to a model where the distance between the two coordinates can not tend to $0$ (Note that some proofs are deferred to Appendix B).  
\smallskip
    \noindent We consider the $2$-dimensional $SDE$ with Lipschitz continuous coefficients defined $\forall x \in \ER^2$ by
\begin{eqnarray*}
b(x)&=&\Big(x\mbox{\bf 1}_{\{0\le |x|\le 1\}} -\frac{x}{|x|} \mbox{\bf 1}_{\{ |x|\ge 1\}}      \Big)(1-|x|)\\
\sigma(x) &=&\vartheta {\rm Diag}(b(x)) + \left[\begin{array}{cc}0&-cx^2\\cx^1&0,
\end{array}\right]
\end{eqnarray*}
where $\vartheta,\, c\!\in (0,+\infty)$ are fixed parameters. 

Switching to polar coordinates $X_t = (r_t\cos \varphi_t, r_t\sin \varphi_t)$, $t\!\in \ER_+$, we obtain that this $SDE$ also reads
\begin{eqnarray}
dr_t& =&\min(r_t,1)(1-r_t)(dt +\vartheta dW^1_t),\quad r_0\!\in \ER_+ \label{eq:countex1}\\
d\varphi_t &=&c\,dW^2_t,\quad \varphi_0\!\in [0,2\pi), \label{eq:countex2}
\end{eqnarray}
where $x_0=r_0(\cos\varphi_0, \sin\varphi_0)$ and $W=(W^1,W^2)$ is a standard $2$-dimensional Brownian motion. 

\noindent Standard considerations about Feller classification (see Appendix B for details) show that, if $x_0\neq 0$ ($i.e.$ $r_0>0$) and $\vartheta\!\in (0,\sqrt{2})$  then 
\begin{equation}\label{rayontend}
r_t\longrightarrow 1\;\mbox{ as }\;t\to +\infty,
\end{equation}
while it is classical background that
\[
\PE \mbox{-}a.s.\quad \forall\varphi_0\!\in \ER_+,\quad \frac 1t \int_0^t \delta_{e^{i(\varphi_0+cW^2_s)}}ds\Longrightarrow \lambda_{\mathcal{S}_1} \quad \mbox{ as } \quad t\to + \infty
\] 
where $\mathcal{S}_1$ denotes the unit circle of $\ER^2$. Combining these two results straightforwardly yields 
\[
	\forall\, x\!\in \ER^2\setminus\{(0,0)\},\; \PE\mbox{-}a.s.\; \quad \frac 1t \int_0^t \delta_{X^{x}_s}ds \stackrel{(\ER^2)}{\Longrightarrow} \lambda_{\mathcal{S}_1} \quad \mbox{ as }\quad t\to+ \infty.
\]

On the other hand, given the form of  $\varphi_t$, it is clear that  if  $x= r_0e^{i\varphi_0}$ and $x'= r'_0e^{i\varphi'_0}$, $r_0$, $r'_0\neq 0$, $\varphi_0$, $\varphi'_0\!\in [0, 2\pi)$, then
\[
\lim_{t\to +\infty}|X_t^x-X^{x'}_t|= |e^{i(\varphi_0-\varphi'_0)}-1|
\]
which in turn implies that
\[
\lim_{t\to +\infty}\frac{1}{t}\int_0^t |X_s^x-X^{x'}_s|ds= |e^{i(\varphi_0-\varphi'_0)}-1|\qquad \PE\mbox{-}a.s.
\]
This limit being different from $0$ as soon as $\varphi_0\neq\varphi'_0$,   one derives, as  a consequence, that $\nu_\Delta$ cannot be the only invariant distribution.  In fact, a more precise statement can be proved.

\begin{prop}$(a)$ A distribution $\mu$ is  invariant for  the semi-group  $(Q_t)_{t\ge0}$ of the duplicated diffusion if and only if $\mu$ has the following form: 
\begin{equation}\label{IMC}
\mu={\cal L}(e^{i\Theta},e^{i(\Theta+V)})
\end{equation}
where $\Theta$ is uniformly distributed over $[0,2\pi]$ and $V$ is a $[0,2\pi)$-valued random variable independent of $\Theta$ 

\smallskip
\noindent $(b)$ When $V=0$ $a.s.$, we retrieve $\nu_\Delta$ whereas, when $V$ also has uniform distribution on $[0,2\pi]$, we obtain $\nu\otimes\nu$. Finally, $\mu$ is extremal in the convex set of $(Q_t)_{t\ge0}$ invariant distributions  if and only if there exists $\theta_0\in[0,2\pi)$
such that $V={\theta_0}$ $a.s$.
\end{prop}

The proof is postponed to Appendix~B. However, note that the claim about extremal invariant distributions  follows from the fact that  for every $\theta\in[0,2\pi)$, $(Q_t)_{t\ge0}$ leaves the set $\Gamma_{\theta}:=\{(e^{i\varphi},e^{i\varphi'})\in{\cal S}_1\times{\cal S}_1, \varphi'-\varphi\equiv\theta \,{\rm mod.} \,2\pi\}$ stable.

\smallskip \begin{Remarque}
 In the above counterexample, the invariant measure of $(r_t)_{t\ge0}$ is the Dirac measure $\delta_1$. In fact, setting again $x= \varepsilon_0e^{i\varphi_0}$ and $x'= r'_0e^{i\varphi'_0}$ and using that 
$ X_t^x-X_t^{x'}= r_t^x\left(e^{i(\varphi_0+W_t^2)}-e^{i(\varphi'_0+W_t^2)}\right)+(r_t^x-r_t^{x'})e^{i(\varphi'_0+W_t^2)},$
an easy adaptation of the above proof shows that it can be generalized to any ergodic non-negative process $(r_t)_{t\ge0}$ {solution to an autonomous SDE and} satisfying the following properties: 
\begin{itemize}
\item{} Its unique  invariant distribution $\pi$ satisfies $\pi(\ER_+^*)=1$. 
\item{} For every $x,y\!\in (0,+\infty)$,  $r_t^x-r_t^y\longrightarrow0$ $a.s.$ as $t\rightarrow+\infty$.
\end{itemize} 
For instance, let $(X^x_t)_{t\ge0}$ be an Ornstein-Uhlenbeck process satisfying the SDE  $dX_t=-X_t dt+\sigma dW_t, X_0=x$. Set $r^x_t=(X^x_t)^2$ (this is a special case of the  Cox-Ingersoll-Ross process). The process
$(r^x_t)$ clearly satisfies the first two properties. Furthermore, $(X^x_t)_{t\ge 0}$ satisfies $a.s.$ for every $x,y\!\in\ER$ and every $t\ge0$, $|X_t^x-X_t^y| = |x-y|e^{- t}$. Then, since for every 
$x\!\in\ER$, 
$$
\frac{X_t^x}{t} = -\frac{1}{t} \int_0^t X^x_sds +\sigma \frac{W_t}{t}\rightarrow0\quad a.s.\quad \mbox{ as } \quad t\rightarrow+\infty,
$$ 
it follows that $(r^x_t)_{t\ge0}$ also satisfies for all positive $x,\, y$,  $r_t^x-r_t^y\longrightarrow0$ $a.s.$ as $t\rightarrow+\infty$ (Many other examples can be built using Corollary~\ref{corhasm}).\\
Finally, note that if $\mu={\cal L}(Re^{i\Theta},R e^{i(\Theta+V)})$ where $R$, $\Theta$ and $V$ are independent
random variables such that the distributions of $R$ and $\Theta$ are respectively $\pi$ and the uniform distribution on $[0,2\pi]$
and $V$ takes values in $[0,2\pi)$, then $\mu$ is an invariant distribution of the associated duplication system.
\end{Remarque}

 In connection with  this counterexample we can mention a general result on the Brownian flows of Harris (see \cite{Harris},\cite{KUN} Theorem 4.3.2). The theorem gives conditions on $b$ and $\sigma$ under which $\nu$ is an invariant measure of the one point motion $(X^x_t)_{t \ge 0}$ and $\nu\otimes\nu$ is an invariant measure of the two point motion $(X^{x_1}_t,X^{x_2}_t)_{t \ge 0}$.

\subsection{Uniqueness of the invariant measure: $(S,\theta)$-confluence}
In the sequel of this section, we propose criterions for the uniqueness of the invariant distribution of the duplicated system
in the multidimensional case. The underlying idea of the criterions discussed below is to analyze  the coupled diffusion process $(X^{x_1},X^{x_2})$ through  the squared distance process $r_t = |X_t^{x_1} - X_t^{x_2}|_{_S}^2$ (where we recall that for a given positive definite matrix $S$, $|\,.\,|_{_S}$ is the Euclidean norm on $\R^d$ induced by the scalar product $(x|y)_{_S}=(x|Sy)$). It is somewhat  similar to that of  Has'minskii's test for explosion of diffusions in $\ER^d$   or  to the one proposed in Chen and Li's work devoted to the coupling of diffusions (see~\cite{CHLI}).
We begin by a general abstract result under an assumption depending  on a continuous   function $\theta: (0,+\infty)\to \ER_+$ to be specified further on. Then,   more explicit \textit{pointwise} or \textit{integrated} criterions are derived in the next subsections. In particular, one  involves a kind of bi-variate non-infinitesimal Lyapunov exponent. 

\smallskip
Let us introduce some notations. {For some probability measures $\nu$ and ${\nu'}$ on $\ER^{d}$, we set
$$
{\cal P}_{\nu,{\nu'}}^\star=\Big\{m\in {\cal P}(\ER^{d}\times\ER^{d} ), m(dx\times \ER^d)=\nu, \,m(\ER^d\times dy)={\nu'}, m(\Delta_{\ER^{d}\times\ER^{d} })=0\Big\}
$$
and  ${\cal P}^\star=\{{\cal P}_{\nu,{\nu'}}^{\star},\nu,{\nu'}\in{\cal I}_{SDE}\}$  where ${\cal I}_{SDE}$ denotes the set of invariant distributions of $(SDE)$. In particular,
${\cal P}^\star={\cal P}_{\nu,\nu}^\star$ when ${\cal I}_{SDE}=\{\nu\}$ (which is the case of main interest).\smallskip}
\noindent For $S\in {\cal S}^{++}(d,\ER)$, we also set 
$$
[b]_{S,+}=\sup_{x\neq y} \frac{(b(x)-b(y)|x-y)_{_S}}{|x-y|_{_S}^2}.
$$
Note that if $[b]_{S,+}<+\infty$ and if $\sigma$ is Lipschitz continuous, strong existence, pathwise uniqueness and the Feller Markov property hold for $(SDE)$.

For a continuous function $\theta:(0,+\infty)\rightarrow\ER_+$, we define the \textit{pseudo-scale} ${\cal C}^2$-function $f_\theta$  and its companion $g_{\theta}$ by
\begin{equation}\label{def:pseu-sc}
\forall u \in(0,+\infty),\qquad f_\theta(u)=\int_1^{u}e^{\int_\xi^1 \frac{\theta(w)}{w}dw}d\xi\;\mbox{ and }\; g_{\theta}(u)=uf'_{\theta}(u). 
\end{equation}
Finally, for $S$ and $\theta$ defined as above, we define the \textit{$(S,\theta)$-confluence} function $\Psi_{\theta,S}$ on $^c \Delta_{\ER^2d}$ by
$$
\Psi_{\theta,S}(x,y)=(b(x)-b(y)|x-y)_{_S} +\frac 12\|\sigma(x)-\sigma(y)\|^2_{_S}-\theta(|x-y|^2_{_S})\Big|(\sigma^*(x)-\sigma^*(y))\frac{S(x-y)}{|x-y|_{_S}}\Big|^2.
$$
Let us now state the result. 
\begin{theorem}\label{thm:echellemultidim}  Let $S\in{\cal S}^{++}(d,\ER)$. Assume that $b$ is a continuous function such that $[b]_{S,+}\!<\!+\infty$ and $\sigma$ is Lipschitz continuous. 
Assume that the set ${\cal I}_{SDE}$ of invariant distributions of $SDE$ is (nonempty, convex and) weakly compact.			
Furthermore, assume that 
for every $m\in {{{\cal P}^\star}}$, the following $(S,\theta)$-confluence condition is satisfied: there exists a continuous function $\theta: (0,+\infty)\to \ER_+$ such that
\begin{equation}\label{eq:Seminal Cond}
\hskip -0.25cm \left\{\begin{array}{ll}
  (i) \displaystyle \limsup_{u\rightarrow 0^+}\int_u^1\frac{\theta(w)-1}{w} dw<+\infty.\\
\\
(ii) \;
\displaystyle \int_{\R^{d}\times \ER^{d}} f'_\theta(|x-y|_{_S}^2)\Psi_{\theta,S}(x,y)m(dx,dy)<0.\\

\end{array}\right.
\end{equation}  
$(a)$ {\em Weak confluence (Uniqueness of both invariant distributions)}: 
Then, if  ${\cal I}_{DSDS}$ denotes the set of  invariant distributions of the duplicated system $(DSDS)$, one has  
\[
{\cal I}_{SDE}=\big\{\nu\big\}\qquad \mbox{ and }\qquad {\cal I}_{DSDS}=\big\{\nu_{\Delta}\big\}
\] 
keeping in mind that $\nu_{\Delta}= \nu\circ (x\mapsto(x,x))^{-1}$.

\smallskip
\noindent $(b)$ {\em Pathwise confluence:} Let $\theta: (0,+\infty)\to \ER_+$ be a continuous function such that $\displaystyle \int_{0}^1e^{\int_v^1 \frac{\theta(w)}{w}dw}dv<+\infty$ and such that 
$$
\forall \, x,\,y\!\in \ER^d, \, x\neq y, \quad\Psi_{\theta,S}(x,y)<0.
$$
If furthermore, for every $x\!\in\ER^d$, $\displaystyle \Big(\frac{1}{t}\int_0^t P_s(x,dy)ds\Big)_{t\ge 1}$ is tight, we  have $a.s.$   pathwise asymptotic confluence:
\begin{equation}\label{confluence++}
\forall\, x_1,\, x_2\!\in \ER^d,\quad X^{x_1}_t-X^{x_2}_t \longrightarrow 0\; \mbox{ as }\; t\to +\infty \;\; \PE \mbox{-}a.s. 
\end{equation}
\end{theorem}

\smallskip \begin{Remarque}
\label{rem3.3}
 $\rhd$ Owing to Assumption $(i)$ and to $[b]_{S,+}<+\infty$, $(x,y)\mapsto f'_\theta({|x-y|^2})\Psi_{\theta,S}(x,y)$ is always bounded from above on $^c\Delta_{\ER^{d}\times\ER^{d} }$ so that 
the integrals with respect to $m\in {{{\cal P}^\star}}$ are well-defined. Also note that since $f'_\theta$ is positive, Assumption $(ii)$ holds in particular if there exists $\theta$ and $S$ such that {the $(S,\theta)$-confluence function $\Psi_{\theta,S}$ is negative on}  $^c{\Delta}_{\ER^{d}\times\ER^{d} }$.

\smallskip 
\noindent  $\rhd$ If we also assume in $(a)$, that $\displaystyle \Big(\frac{1}{t}\int_0^t \!P_s(x,dy)ds\Big)_{t\ge 1}$ is tight then, so is   $\displaystyle \Big(\frac{1}{t}\int_0^t\! Q_s(x,x',dy,dy')ds\Big)_{t\ge 1}$.  Since, by construction, the weak limiting distributions  of this sequence as $t\to +\infty$ are invariant distributions, it follows that  $\displaystyle  \frac{1}{t}\int_0^t \!Q_s(x,x',dy,dy')ds $ weakly converges to $\nu_\Delta$ as $t\to +\infty$. This motivates the ``weak confluence'' terminology.

\smallskip 
\noindent $\rhd$  It is natural to wonder if the assumptions for pathwise asymptotic confluence (claim $(b)$) are more stringent than Assumptions $(i)$ and $(ii)$. The fact that  $\Psi_{\theta,S}<0$ on $\ER^{d}\times\ER^{d} \setminus \Delta_{\ER^{d}\times\ER^{d} }$ implies Assumption $(ii)$ has  already been mentioned. One can also checks that  $\int_{0}^1e^{\int_v^1 \frac{\theta(w)}{w}dw}dv<+\infty$ implies Assumption $(i)$: one first derives from the Cauchy criterion that $\int_{0}^1e^{\int_v^1 \frac{\theta(w)}{w}dw}dv~<+\infty$ implies that 
$\int_{\frac{u}{2}}^u e^{\int_v^1 \frac{\theta(w)}{w}dw}dv\rightarrow0$ as $u\rightarrow 0^+$. Using that 
$v\mapsto e^{\int_v^1 \frac{\theta(w)}{w}dw}$ is non-increasing on $(0,1]$, it follows that $ue^{\int_u^1\frac{\theta(w)}{w}dw}\rightarrow0^+$ as $u\rightarrow0$. Taking the logarithm yields $\int_u^1\frac{\theta(w)-1}{w}dw\rightarrow -\infty$ and thus, Assumption $(i)$.

\medskip
\noindent$\rhd$ If $b$ and $\sigma$ are Lipschitz continuous, Kunita's Stochastic flow theorem (see \cite{KUN}, Section 4.5) ensures in particular that, if $x_1\neq x_2$, the solutions $X^{x_1}_t $ and $X^{x_2}_t$ $a.s.$ never get stuck. Taking advantage of this remark slightly shortens    the proof below.

\smallskip\noindent $\rhd$ Tightness  criterions of  $\displaystyle \Big(\frac{1}{t}\int_0^t P_s(x,dy)ds\Big)_{t\ge1}$ for every $x\!\in\ER^d$ usually rely on the mean-reversion property of the solutions of $(SDE)$ usually established under various assumptions involving the existence of a so-called {\em Lyapunov function} $V$ going to infinity at infinity and such that $\mathcal{A} V$ is upper-bounded and $\limsup_{|x|\to +\infty}\mathcal{A} V(x)<0 $ where $\mathcal{A}$ denotes  the infinitesimal generator of $X^x$ (so-called {\em Has'minskii's criterion}). Keep in mind that
\[
\mathcal{A} V(x) = (b|\nabla V)(x) +\frac 12 {\rm Tr} \Big(\sigma\sigma^*(x)D^2V(x)\Big)
\]
where ${\rm Tr} (A)$ stands for the trace of the matrix $A$.  

On the other hand, a classical criterion for pathwise  asymptotic confluence ($a.s.$ at exponential rate, see $e.g.$~\cite{basak}, \cite{lemairethese}  and often referred to as {\em asymptotic flatness}) is
\begin{equation}\label{explem}
\forall\, x,\, y\!\in \ER^d, \quad (b(x)-b(y)|x-y) +\frac 12 \|\sigma(x)-\sigma(y)\|^2 < -c |x-y|^2, \qquad c>0, 
\end{equation}
and, as a straightforward consequence, uniqueness of the invariant distribution $\nu$ of $(SDE)$ (and of $(DSDS)$ as well). Moreover, putting $y=0$ in the above inequality straightforwardly yields real coefficients $\alpha>0$, $\beta\ge 0$ such that $\mathcal{A}V\le \beta -\alpha V$ with $V(x) \!=\!|x|^2$. Hence Has'minskii criterion is fulfilled, so it is also an existence criterion for the invariant distribution.
{In fact, both weak and pathwise assumptions in Theorem \ref{thm:echellemultidim}  are much weaker than \eqref{explem} but some of the properties which hold under \eqref{explem} are still preserved. For instance, since the left-hand side of \eqref{explem} corresponds  to the   $(S,0)$-confluence function, we deduce from the criterions that if the $(S,0)$-confluence function is (only) negative on $^c\Delta_{\ER^{d}\times\ER^{d} }$, uniqueness of the invariant distribution $\nu$ of $(SDE)$ (and of $\nu_{\Delta}$ for ($DSDS)$) holds and,  combined with the tightness of the occupation measure of the semi-group, it becomes a criterion for $a.s.$ pathwise asymptotic confluence.}
\end{Remarque}

\noindent {\em Proof of Theorem~\ref{thm:echellemultidim}.} {\sc Step~1:}  Exactly like  in the beginning of the proof of Theorem~\ref{theoremhasm}$(a)$, one checks that the set ${\cal I}_{DSDS}$ of  invariant distributions of $(Q_t)_{t\ge0}$ is a nonempty,  convex and weakly compact subset of ${\cal P}(\ER^{d}\times\ER^{d} )$. 
As a  a consequence of the Krein-Millman theorem, ${\cal I}_{DSDS}$ has extremal distributions (and is their closed convex hull).

\smallskip On the other hand, 
it follows from strong uniqueness theorem  for SDE's that the semi-group $(Q_t)_{t\ge 0}$ leaves stable  the diagonal $\Delta_{\ER^{d}\times\ER^{d} }=\{(x,x),\, x\!\in \R^d\}$.  

Let $x_1,x_2\!\in \R^d$, $x_1\neq x_2$. We define the stopping time
\[
\tau_{x_1,x_2}:=\inf\big\{t\ge 0\,|\, X^{x_1}_t=X^{x_2}_t\big\}.
\] 
Still by a strong uniqueness argument it is clear that $\{\tau_{x_1,x_2}>t\}= \{X_t^{x_1}\neq X^{x_2}_t\}$ so that 
\[
Q_t((x_1,x_2), ^c\!\Delta_{\ER^{d}\times\ER^{d} })= \mbox{\bf 1}_{^c\!\Delta_{\ER^{d}\times\ER^{d} }}(x_1,x_2)\PE(\tau_{x_1,x_2}>t)
\]
and $Q_t((x_1,x_1), ^c\!\Delta_{\ER^{d}\times\ER^{d} })= 0$.

Let $\mu\!\in {\cal I}_{DSDS}$ be an extremal invariant measure. We have, for every $t\ge 0$,
\[
\mu (^c\!\Delta_{\ER^{d}\times\ER^{d} }) = \int_{^c\!\Delta_{\ER^{d}\times\ER^{d} }}\mu(dx_1,dx_2)\PE(\tau_{x_1,x_2}>t).
\]
Letting $t$ go to $+\infty$ yields
\[
\mu (^c\!\Delta_{\ER^{d}\times\ER^{d} }) = \int_{^c\!\Delta_{\ER^{d}\times\ER^{d} }}\mu(dx_1,dx_2)\PE(\tau_{x_1,x_2}=+\infty)
\]
so that, on $^c\!\Delta_{\ER^{d}\times\ER^{d} }$, $\mu(dx_1,dx_2)$-$a.s.$, $\PE(\tau_{x_1,x_2}=+\infty)
=1$ or equivalently the process $(X^{x_1},X^{x_2})$ lives in $^c\!\Delta_{\ER^{d}\times\ER^{d} }$. Consequently, if $\mu(^c\!\Delta_{\ER^{d}\times\ER^{d} })\!\in (0,1)$, both conditional measures $\mu^{^c\!\Delta_{\ER^{d}\times\ER^{d} }}$ and $\mu^{\Delta_{\ER^{d}\times\ER^{d} }}$ are invariant distributions for $(SDSD)$ as well. If so,
\[
\mu = \mu(^c\!\Delta_{\ER^{d}\times\ER^{d} })\mu^{^c\!\Delta_{\ER^{d}\times\ER^{d} }} +\mu(\Delta_{\ER^{d}\times\ER^{d} })\mu^{\Delta_{\ER^{d}\times\ER^{d} }}
\]
cannot be extremal. Consequently $\mu(\Delta_{\ER^{d}\times\ER^{d} })= 0$ or $1$.

\smallskip
\noindent {\sc Step~2:} Let $\mu$ be an extremal distribution  in ${\cal I}_{DSDS}$   and assume that  $\mu(^c\Delta_{\ER^{d}\times\ER^{d} })=1$ so that $\mu\in {{{\cal P}^\star}}$. We will prove that this yields a contradiction under Assumptions $(i)$ and $(ii)$.\\ 
Note that $f'_{\theta}$ and  $g_{\theta}$ defined in \eqref{def:pseu-sc} are positive on $(0,+\infty)$, that Assumption $(i)$ reads $\limsup_{u\rightarrow0^+} g_\theta(u)<+\infty$ and that  $g'_{\theta}(u) = f'_{\theta}(u)(1-\theta(u))$. Moreover, if Assumption $(ii)$ is fulfilled, so is the case for any continuous function $\widetilde \theta$ satisfying $\widetilde \theta\ge \theta$. As a consequence,  we may modify $\theta$ on $[1,+\infty)$ so that   $\theta$ still satisfies $(ii)$ and   $\theta\ge 1$ over $[2\varepsilon,+\infty)$. Then the function $g_{\theta}$ is   non-increasing   on $[2,+\infty)$. Consequently, without loss of generality, we may assume in the sequel of the proof that   
\begin{equation}\label{condf'}
\sup_{u>0} g_\theta(u)<+\infty.
\end{equation}

We now define a (Lyapunov) function $\varphi:\,^c\!\Delta_{\ER^{d}\times\ER^{d} }\rightarrow \ER$ by 
$$\varphi(y_1,y_2):=f_\theta(|y_1-y_2|^2_{_S}).$$
We know from Step~1  that $\mu(dx_1,dx_2)$-$a.s.$,  ($X^{x_1}_t\neq X^{x_2}_t$ for every $t\ge 0$) $a.s.$. Then, $f_\theta$ being a ${\cal C}^2$-function, we derive from  It\^o's formula applied to  $\varphi(X^{x_1}_t,X^{x_2}_t)$ that $\mu(dx_1,dx_2)$-$a.s.$,
\begin{multline*}
\varphi(X^{x_1}_t,X^{x_2}_t) = \varphi(x_1,x_2)+\int_0^t {\cal A}^{(2)}\varphi(X^{x_1}_s,X^{x_2}_s)ds \\
 +\underbrace{ \int_0^t f'_\theta(|X^{x_1}_s-X^{x_2}_s|^2_{_S})\big( (\sigma^*(X^{x_1}_s)-\sigma^*(X^{x_2}_s))S(X^{x_1}_s-X^{x_2}_s)|dW_s\big)}_{=:M_t \;\hbox{\footnotesize local martingale}}
\end{multline*}
where, for every $(x_1,x_2)\!\in (\R^d)^2$,  
\begin{multline}\label{eq:A2}
 {\cal A}^{(2)}\varphi(x_1,x_2) = 2\Big(  (b(x_1)-b(x_2)|x_1-x_2)_{_S} +\frac 12\|\sigma(x_1)-\sigma(x_2)\|_{_S}^2\Big)f'_{\theta}(|x_1-x_2|^2_{_S}) \\
	+ 2 f''_{\theta}(|x_1-x_2|^2_{_S}) \big|(\sigma^*(x_1)-\sigma^*(x_2))S(x_1-x_2)\big|^2.
\end{multline}
\noindent  Using that $f_{\theta}$ is increasing and  satisfies the $ODE\equiv \theta(\xi) f'_{\theta}(\xi)+\xi  f''_{\theta}(\xi)=0$, $\xi\!\in(0,+\infty)$, we deduce
that  
$$
{\cal A}^{(2)}\varphi(x_1,x_2)=2f'_\theta(|x_1-x_2|_{_S}^2)\Psi_{\theta,S}(x_1,x_2)
$$
so that $\displaystyle \int_{\R^d\times \R^d} {\cal A}^{(2)}\varphi(x_1,x_2)\mu(dx_1,dx_2)<0$ by Assumption $(ii)$.

On the one hand, since $\mu$ is extremal and  since  ${\cal A}^{(2)}\varphi$ is bounded from above (see Remark \ref{rem3.3}), we can apply  Birkhoff's theorem and obtain: 
\begin{equation}\label{eq:eq:calA2mu}
\mu(dx_1,dx_2)\mbox{-}a.s.,\quad \frac{1}{t}\int_0^t{\cal A}^{(2)}\varphi(X^{x_1}_s,X^{x_2}_s)ds\xrn{t\nrn}\int_{^c\Delta_{\ER^{d}\times\ER^{d} }} \hskip -0.5 cm {\cal A}^{(2)}\varphi \,d\mu\in[-\infty,0)\quad  a.s.
\end{equation}
On the other hand, using that $g_{\theta}$ is bounded and $\sigma$ is Lipschitz continuous, it follows that $(M_t)_{t\ge0}$ is an $L^2$-martingale such that
\begin{equation}\label{crochetmt2}
\left\langle  M\right\rangle_t=\int_0^t  g_\theta(|X^{x_1}_s-X^{x_2}_s|^2_{_S})^2\Big|\frac{(\sigma^*(X^{x_1}_s)-\sigma^*(X^{x_2}_s))}{|X^{x_1}_s-X^{x_2}_s|_{_S}}\frac{S(X^{x_1}_s-X^{x_2}_s)}{|X^{x_1}_s-X^{x_2}_s|_{_S}}\Big|^2ds\le C t\quad \PE_{\mu}\mbox{-}a.s.
\end{equation}
where $C$ is a deterministic positive constant so that $\frac{M_t}{t}\rightarrow0$  $\PE_{\mu}\mbox{-}a.s.$.

As a consequence, $\mu(dx_1,dx_2)$-$a.s.$,
$$
\lim_{t\rightarrow+\infty} \frac{\varphi(X^{x_1}_t,X^{x_2}_t)}{t}=\int_{^c\Delta_{\ER^{d}\times\ER^{d} }} {\cal A}^{(2)}\varphi d\mu<0\quad a.s..
$$
Hence,  $a.s.$,  $f_\theta(|X^{x_1}_t-X^{x_2}_t|^2_{_S})=\varphi(X^{x_1}_t,X^{x_2}_t)\xrn{t\nrn}-\infty\quad a.s.$

\noindent If  $\lim_{u\rightarrow 0^+} f_\theta(u)>-\infty$, this yields a contradiction since $f_\theta$ is increasing on $\ER_+^*$. Otherwise  
$|X^{x_1}_t-X^{x_2}_t|^2_{_S}\xrn{t\nrn}0$. But applying again Birkhoff's theorem, 
we obtain $\mu(dx_1,dx_2)$-$a.s.$,
$$
\int |y_1-y_2|^2_{_S}\mu(dy_1,dy_2)=\lim_{t\rightarrow+\infty}\frac{1}{t}\int_0^t|X^{x_1}_t-X^{x_2}_t|^2_{_S} ds=0\quad a.s.,
$$
which contradicts the assumption $\mu(^c\Delta_{\ER^{d}\times\ER^{d} })=1$. Consequently, for any extremal invariant  distribution $\mu$, we have $\mu(\Delta_{\ER^{d}\times\ER^{d} })=1$.

\smallskip  We can now prove Claim $(a)$:  by Krein-Millman's Theorem ${\cal I}_{SDS}$ is the weak closure of the convex hull of its extremal distributions. Consequently, the diagonal $\Delta_{\ER^{d}\times\ER^{d} }$ being a closed subset of $\ER^{d}\times\ER^{d} $, all invariant distributions of the duplicated system 
are supported by this diagonal. For any such invariant distribution $\mu$, both its marginals   are invariant distributions for $(SDE)$. If  $(SDE)$ had two distinct invariant distributions $\nu$ and $\nu'$, we know  from the introduction  that ${\cal I}_{DSDS}$ would contain at least a distribution  $\mu$ for which the two marginals distributions are $\mu( .\times \R^d)= \nu$ and $\mu( \R^d\times .)=\nu'$ respectively. As a consequence, such a distribution $\mu$ could not by supported by the diagonal $\Delta_{\ER^{d}\times\ER^{d} }$.
Finally, ${\cal I}_{SDE}$ is reduced to a singleton $\{\nu\}$ and ${\cal I}_{DSDS} =  \{\nu_{\Delta} \}$.

\smallskip
\noindent  {\sc Step~3 (Claim $(b)$: Proof of \eqref{confluence++}):} Under the additional assumption on $\theta$ of  $(b)$, we have
$\lim_{u\rightarrow0^+} f_\theta(u)>-\infty$ and thus, $\inf_{u>0} f_\theta(u)>-\infty$ since $f_\theta$ is increasing. Let $x_1$, $x_2\!\in \ER^d$. Using again that ${\cal A}^{(2)}\varphi<0$ on $^c\!\Delta_{\ER^{d}\times\ER^{d} }$ where ${\cal A}^{(2)}\varphi$ is given  by~(\ref{eq:A2}).
It follows  that $\big(f_\theta(|X^{x_1}_t-X^{x_2}_t|^2_{_S})\big)_{t\ge 0}$ is a lower-bounded $\PE$-supermartingale. Thus, it $a.s.$ converges toward $L^{x_1,x_2}_{\infty}\!\in L^1(\PE)$. Using again that $f_\theta$ is increasing, it follows that $|X^{x_1}_t-X^{x_2}_t|^2_{_S}$ $a.s.$ converges toward a finite random variable $\ell^{x_1,x_2}_{\infty}=f^{-1}_\theta(L^{x_1,x_2}_{\infty})$.

Now, using that for every $x\!\in\ER^d$, $\displaystyle\Big(\frac{1}{t}\int_0^t P_s(x,dy)ds\Big)_{t\ge 1}$ is tight, we derive that $\displaystyle\Big(\frac{1}{t}\int_0^t Q_s((x_1,x_2),(dy_1,dy_2))ds\Big)_{t\ge 1}$ is tight as well. Then the uniqueness of  $\nu_\Delta$ as an invariant distribution of $Q$  implies that 
\[
\frac 1t \int_0^t Q_s\big((x_1,x_2), (dy_1,dy_2)\big)ds \stackrel{(\ER^d)}{\Longrightarrow} \nu_{\Delta}.
\]
Now for every bounded continuous function $g:\ER^d\to \ER$,
\[
\frac 1t \int_0^t Q_s(g(|y_1-y_2|^2_{_S}))(x_1,x_2)ds = \frac 1t \int_0^t \ES\,g(|X^{x_1}_s-X^{x_2}_s|^2_{_S})ds\longrightarrow \ES\,g(\ell^{x_1,x_2}_\infty)
\]
so that 
\[
\ES\,g(\ell^{x_1,x_2}_\infty) = \int g(|y_1-y_2|^2_{_S})\nu_{\Delta}(dy_1,dy_2) = g(0).\qquad \Box
\]

In Assumption $(ii)$ of the previous theorem, we see that  the function $(x,y)\mapsto |(\sigma^*(x)-\sigma^*(y)) S (x-y)|$  plays an important role. In the sequel, we will obtain  specific results  when this function is not degenerated 
away from the diagonal. Such type of assumption will be called \textit{strong} or \textit{regular directional $S$-ellipticity assumption}.
 
\smallskip
In  the following proposition, we first show that when such an assumption is satisfied,  claim $(b)$ of the previous theorem  still holds without  the tightness assumption on $\big(\frac 1t\int_0^tP_s(x,dy)ds\big)_{t\ge 1}$ (although it is not really restrictive in our framework (see the fourth item of Remark \ref{rem3.3})).

\begin{prop} If the function $\theta$ is $(0,1]$-valued and $\sigma$ satisfies the following strong directional $S$-ellipticity assumption away from the diagonal
\begin{equation}\label{eq:ellip}
\exists\,\, \alpha_0>0,\; \forall\, x,\, y\!\in \R^d,\;  \big|(\sigma^*(x)-\sigma^*(y))S(x-y)\big|\ge \alpha_0|x-y|^2,
\end{equation}
then the conclusion of Claim $(b)$ in the above proposition remains true without the tightness assumption on $(\frac 1t\int_0^tP_s(x,dy)ds)_{t\ge 1}$.
\end{prop}
\begin{proof} First, we recall that under the assumptions of $(b)$, we recall that $(f_\theta(|X^{x_1}_t-X^{x_2}_t|^2_{_S}))_{t\ge0}$
 is a lower-bounded $\PE$-supermartingale thus convergent to an integrable random variable and that this implies that
 $(|X^{x_1}_t-X^{x_2}_t|^2_{_S})_{t\ge0}$ is $a.s.$ convergent to a finite random variable $\ell_\infty^{x_1,x_2}$ (since $f_\theta$ is increasing). On the other hand,  since $-{\cal A}^{(2)}\varphi$ is positive and $f_\theta$ is lower-bounded,  we also have that
\[
f_\theta(|X^{x_1}_t-X^{x_2}_t|^2_{_S})-\int_0^t{\cal A}^{(2)}\varphi(X^{x_1}_s,X^{x_2}_s)ds = \varphi(x_1,x_2)+M_t
\]
 is a lower bounded $\PE$-(local) martingale starting at a  deterministic starting value, hence converging toward an integrable random variable. Owing to the computations of \eqref{crochetmt2}  (which hold for every starting points $x_1$, $x_2$),   $(M_t)_{t\ge0}$ is in fact an $L^2$- convergent martingale. Thus, $\langle M\rangle_\infty<+\infty$ and taking advantage of the expression  of this bracket (see \eqref{crochetmt2}) and to Assumption \eqref{eq:ellip}, we derive that for every $\varepsilon>0$
 \[
 \int_0^{+\infty} g_{\theta}\big(|X^{x_1}_s-X^{x_2}_s|^2_{_S}\big)^2\mbox{\bf 1}_{\{|X^{x_1}_s-X^{x_2}_s|^2_{_S}\ge \varepsilon\}}ds<+\infty\quad a.s.
 \]
The function $g_\theta$ is positive on $(0,+\infty)$ and non-decreasing since $g'_{\theta}(u)=f'_{\theta}(u)(1-\theta(u))\ge 0$. This implies that, for every $\varepsilon>0$,
\[
\liminf_{t\to +\infty}g_{\theta}\big(|X^{x_1}_t-X^{x_2}_t|^2_{_S}\big)\mbox{\bf 1}_{\{|X^{x_1}-X^{x_2}|^2_{_S}\ge \varepsilon\}}=0\quad  a.s.
\]
Combined with the  convergence of the squared norm this yields
\[
\forall\, \varepsilon>0,\quad g_{\theta}\big(\ell_{\infty}^{x_1,x_2}\big)\mbox{\bf 1}_{\{\ell_{\infty}^{x_1,x_2}\ge \varepsilon\}}=0 \quad a.s.
\]
which finally implies $\ell_{\infty}^{x_1,x_2}=0$ $a.s.$
\end{proof}
\subsection{Global criterions, NILS exponent}\label{subsec:NILS}  

In this section and the following, we derive several corollaries of Theorem~\ref{thm:echellemultidim} illustrated by  different examples.

\begin{prop}\label{cor:GC} Let $S\in {\cal S}^{++}(d,\ER)$. Assume $[b]_{S,+}<+\infty$,  $\sigma$ is Lipschitz continuous and ${\cal I}_{SDS}$ is non empty and weakly compact.  Then,

\smallskip
\noindent $(a)$ If  Assumption $(ii)$ of Theorem \ref{thm:echellemultidim} holds with some continuous functions $\theta:(0,+\infty)\rightarrow\ER_+$ satisfying: there exists $\varepsilon_0>0$ such that $\theta(u) \le 1$, $u\!\in  (0, \varepsilon_0]$, then   $(SDE)$~(\ref{EDS}) and its duplicated system $(DSDS)$ have $\nu$ and $\nu_{\Delta}$ as unique invariant distributions respectively. 

\smallskip
\noindent $(b)$ If for every $x\!\in\ER^d$, $(\frac{1}{t}\int_0^t P_s(x,dy)ds)_{t\ge1}$ is tight and if $\Psi_{\theta,S}<0$ on $\Delta_{\ER^{d}\times\ER^{d} }^c$ with  a continuous function $\theta:(0,+\infty)\rightarrow\ER_+$ satisfying: there exists $\kappa>1$ and $\varepsilon_0\!\in (0, e^{-\frac{\kappa}{2}})$ such that 
\begin{equation}\label{eq:kappa>1}
\forall\, u\!\in(0,\varepsilon_0], \quad \theta (u)\le \Big(1+\frac{\kappa}{\log u}\Big),
\end{equation}
then the duplicated system of $(SDE)$ is pathwise confluent in the sense of Theorem~\ref{thm:echellemultidim}$(b)$. This condition is in particular satisfied if there exists $\varepsilon_0>0$ and $\theta_0\!\in (0,1)$ such that 
\[
\forall\, u\!\in(0,\varepsilon_0], \quad \theta (u)\le \theta_0.
\]
\end{prop}

\begin{proof} Claim~$(a)$ is obvious. As for~$(b)$, one checks that $\int_0^1e^{\int_{v}^1 \frac{\theta(w)}{w}dw}dv<+\infty$ as soon as $\liminf_{u\to 0^+} \log (u) \big(\theta(u) - 1 \big)>1$ and the result follows.
\end{proof}
\begin{Remarque}\label{remtheta=0}
 The simplest case where the preceding result holds is obtained when $\theta\equiv 0$. In this case, the weak confluence condition, referred to as $(S,0$)-confluence in what follows, reads:
\begin{equation}\label{asymp-conf1}
\forall m\!\in {{{\cal P}^\star}},\; \int_{\R^d\times\R^d} (b(x)-b(y)|x-y)_{_S} +\frac 12\|\sigma(x)-\sigma(y)\|^2_{_S}m(dx,dy)<0
\end{equation}
and claim $(b)$ holds as soon as the integrated function is negative on $^c \Delta_{\ER^{d}\times\ER^{d} }$.
\end{Remarque}
 
At this stage, it is important for practical applications to note that {\em the constant function  $\theta \equiv 1$  satisfies the assumption in $(a)$ of the above Proposition}. This leads us to introduce  an important quantity of interest for our purpose.
\begin{definition}  The {\em non-infinitesimal $S$-Lyapunov (NILS)  exponent} is a function on $\ER^{d}\times\ER^{d} \setminus\Delta_{\ER^{d}\times\ER^{d} }$ defined    for every 
$x,\, y\!\in \R^d, \; x\neq y$, by 
\begin{equation*}
	\hskip -0.5 cm\Lambda_{_S}(x, y) = \frac{(b(x)-b(y)|x-y)_{_S}}{|x-y|^2_{_S}}	+  \frac{1}{2}\frac{\|\sigma(x)-\sigma(y)\|^2_{_S}}{|x-y|^2_{_S}} - \Biggl(
	\Big| \frac{(\sigma^*(x)-\sigma^*(y))S(x-y)}{|x-y|^2_{_S}}\Big|^2\Biggr).
\end{equation*}
\end{definition}
\begin{Corollaire}\label{Cor:lyap} Assume  $b$ and $\sigma$ are like in Proposition~\ref{cor:GC} and ${\cal I}_{SDE}$ is non empty and weakly compact.

\smallskip
\noindent $(a)$  {\em Negative {Integrated} NILS exponent}: if
\begin{equation}\label{eq:NILSmneg}
\forall m\in {{{\cal P}^\star}},\quad \quad \int _{\R^d}\Lambda_{_S}(x,y) m(dx,dy)<0,
\end{equation}
then $(SDE)$ and its duplicated system $(DSDS)$ have $\nu$ and $\nu_{\Delta}$ as unique invariant distributions respectively.

\smallskip 
\noindent $(b)$  {\em Negative  NILS exponent  bounded away from $0$}: If furthermore $\big(\frac{1}{t}\int_0^t P_s(x,dy)ds\big)_{t\ge1}$ is tight for every $x\!\in\ER^d$  or $\sigma $ satisfies~(\ref{eq:ellip}) and if there exists $c_0>0$ such that
\begin{equation}\label{negNILS}
\forall\, x,\, y\!\in\R^d, \; x\neq y,\; |x-y|^2_{_S}\le \varepsilon_0 \Longrightarrow  \Lambda_{_S}(x,y)\le  -c_0
\end{equation}
then the duplicated diffusion is pathwise confluent $i.e.$
\[
\forall\, x_1,x_2\!\in \R^d,\quad X^{x_1}_t-X^{x_2}_t \longrightarrow 0 \quad a.s.\quad \mbox{ as }t\to+\infty.
\]
\end{Corollaire}
\begin{proof}$(a)$  follows from claim~$(a)$ in the above proposition with $\theta\equiv 1$ since $u f'_{\theta}(u) \equiv 1$ on $(0,+\infty)$ so that ${\cal A}^{(2)}\varphi(x,y) = 2 \Lambda_{_S}(x,y)$ in the proof of Theorem~\ref{thm:echellemultidim}. $(b)$ follows from claim~$(b)$ in the same proposition.
\end{proof}

\begin{Remarque}\label{rem3.5} $\rhd$ It is obvious that~\eqref{eq:NILSmneg} is satisfied, $i.e.$ the integrated NILS (INILS) exponent is negative for every $m \! \in {{{\cal P}^\star}}$,  as soon as the NILS exponent  itself is negative on $^c \Delta_{\ER^{d}\times\ER^{d}}$.  
This pointwise negativity  may appear as the only checkable condition for practical applications, but so is not the case and we will see in the next subsections that we can devise  criterions when $\Lambda_{_S}$ is not negative everywhere.

\smallskip
\noindent {$\rhd$ Let us assume that {$\nu$ is unique, that} $\nu\otimes\nu \!\in{\cal P}_{\nu,\nu}^\star$ (for instance, so is the case if $\nu$ is atomless) and that $b$ and $\sigma$ are such that for all $x\neq y$, for al $\nu\otimes\nu \!\in{\cal P}_{\nu,\nu}^\star$ (for instance, so is the case if $\nu$ is atomless) and that $b$ and $\sigma$ are such that for all $x\neq y$, for all $t\ge0$,  $\PE(X_t^x\neq X_t^y)=1$ (see the fourth item of  Remark~\ref{rem3.3} for comments on this topic). Then, for each $t > 0$, one easily checks that the probability  measure $\mu_t^{(0)}$ defined in \eqref{loiind1} belongs to ${\cal P}_{\nu,\nu}^\star$. {Furthermore, if \eqref{eq:NILSmneg} holds, $(\mu_t^{(0)})_{t\ge1}$ converges weakly to $\nu_\Delta$ since this family is weakly compact (see the first item of Remark~\ref{rem3.3}) and $\nu_{\Delta}$ is its only possible limiting distribution. On the other hand, the function $\Lambda_{_S}$ being continuous and upper-bounded on $^c\Delta_{\ER^d\times\ER^d}$,  can be extended  on the diagonal  into a lower semi-continuous (l.s.c.) function on $\ER^d\times\ER^d$ (with values in $\ER\cup\{-\infty\}$). Let   $\underline{\Lambda}_{_S}$ denote its  l.s.c. envelope. { Temporarily assume that $b$ and $\sigma$ are Lipschitz continuous so that $\underline \Lambda_{_S}$ is bounded. Then applying Fatou's  Lemma in distribution to  $\underline{\Lambda}_{_S}$,  it follows from~\eqref{eq:NILSmneg} that
\begin{equation}\label{prescriterediag}
\int \underline{\Lambda}_{_S}(x,x)\nu(dx)=\int \underline{\Lambda}_{_S}(x,y)\nu_{\Delta}(dx,dy)\le \liminf_{t\rightarrow+\infty} \int\underline{\Lambda}_{_S}(x,y)\mu_t^{(0)}(dx,dy)\le 0.
\end{equation}
}
The interesting point is that the left-hand side of~\eqref{prescriterediag} is an integral with respect to $\nu$ and can be seen as a necessary condition for the criterion~\eqref{eq:NILSmneg}.} {In fact the result still holds if    $[b]_{S,+}\!<\!+\infty$ {\em mutatis mutandis} and there exists  a continuous  function  $\ell \!\in L^{1}(\nu)$ such that  
\begin{equation}\label{eq:Lambdalsc}
  \forall\, x,\, y\!\in \R^d\times \R^d,\; \Lambda^-_{_S}(x,y) \le \ell(x)+\ell(y)
\end{equation}
(where, for a function $f$, $f^{\pm}= \max(\pm f,0)$).
}

{Furthermore, when $b$ and $\sigma$ are continuously differentiable, one derives from a Laplace-Taylor expansion  (integral remainder) that  $\underline{\Lambda}_{_S}(x,x)$ reads for every $x\!\in \R^d$:
\begin{equation}\label{formellsc}
\underline{\Lambda}_{_S}(x,x)= \frac 12\inf_{|u|_{_S}=1}     \left(u^*\left( SJ_b(x)+J_b^*(x)S\right)u +  \Big\|(\nabla\sigma(x) |u)\Big\|^2_{_S} - 2 \Big|(\nabla \sigma^*(x)Su|u)\Big|^2\right).
\end{equation}
Thus,~\eqref{prescriterediag} can be read as a checkable necessary condition  for  the criterion~\eqref{eq:NILSmneg}}. We will come back
on this condition in  Subsection~\ref{OptiTrans}.}

\smallskip  
\noindent $\rhd$ In~$(b)$, Condition \eqref{negNILS} can be replaced by the sharper: for all $x,y\in\ER^d$, $\Lambda_{_S}(x,y)<0$ and there exists $\kappa>1$ and 
$\varepsilon_0\!\in (0,e^{-\frac{\kappa}{2}})$ such that for all $x,y\in\ER^d$ such that
$|x-y|_S\le\varepsilon_0$, $\displaystyle \Lambda_{_S}(x,y)\le \frac{\kappa}{\log(|x-y|_{_S}^2)}$.
\end{Remarque}

\medskip
\noindent {When the coefficients are smooth enough, the negativity of $\Lambda_{_S}$ can be ensured by the following criterion:
\begin{Corollaire}[Smooth coefficients]  \label{Cor:lypadiff} Assume  the functions $b$ and $\sigma$ are 
 {\em continuously} differentiable.  Let  $J_b (x)=\left[\frac{\partial b_i}{\partial x_j}(x)\right]_{1\le i,j\le d}$ denote the Jacobian of $b$ at $x$ and let
 $\nabla \sigma (x)=\Big[\frac{\partial \sigma_{ij}}{\partial x^k}(x)\Big]_{i,j,k}$ denote the gradient of $\sigma$ at $x$. If both  $SJ_b+J^*_bS$ and $\nabla \sigma$ are {\em Lipschitz continuous} and if  $\nabla \sigma$ is bounded then $\Lambda_{_S}(x,y)\le  -c_0$ on $^c \Delta_{\ER^d\times\ER^d}$ if 
 \[
\sup_{x\in \ER^d}\sup_{|u|_{_S}=1}     \left(u^*\left( SJ_b(x)+J_b^*(x)S\right)u +  \Big\|(\nabla\sigma(x) |u)\Big\|^2_{_S} - 2 \Big|(\nabla \sigma^*(x)Su|u)\Big|^2\right)<0
\]
where, for every $v=(v^1, \ldots,v^d)\!\in \R^d$, $\displaystyle (\nabla\sigma(x) |v)= \Big[(\nabla \sigma_{ij}(x)|v)\Big]_{1\le i,j\le d}$ and   $\displaystyle \nabla \sigma(x)  v=\left[\sum_{k=1}^d \frac{\partial \sigma_{ij}}{\partial x_k}(x)v^k\right]_{1\le i,j\le d}$.
When  $S=I_d$, this may also be written 
\[
\sup_{x\in \ER^d}\sup_{|u|=1} \left(u^*\Big((J_b +J^*_b)(x) +\sum_{ i,j}(\nabla \sigma_{ij}(x))^{\otimes 2} -2 \Big[(\nabla\sigma_{ij}(x)|u)\Big]\Big[(\nabla\sigma^*_{ij}(x)|u)\Big]\Big)u\right)<0.
\]
\end{Corollaire}

The proof is again an easy consequence of the Laplace-Taylor formula. Computational   details  are left  to the reader.}

\subsection{Applications and extensions}\label{App&Ext}

\subsubsection{Localization around the diagonal} 
By local we mean that the confluence condition will be effective only   in the neighbourhood of the diagonal $\Delta_{\ER^{d}\times\ER^{d} }$. The price to pay is a regular  directional ellipticity assumption on $\sigma(x)-\sigma(y)$ in the direction $S(x-y)$ away from the diagonal.
\begin{prop}  \label{prop_principale} Assume $[b]_{S,+}<+\infty$,  $\sigma$ is Lipschitz continuous and $(SDE)$ admits at least one invariant distribution $\nu$. If there exists  $\varepsilon_0>0$ such that
\[
   \left\{\begin{array}{l}
   (i) \;\hbox{\it Directional $S$-ellipticity: }\;\eta_0:= \displaystyle \inf\Big\{\big|(\sigma^*(x)-\sigma^*(y))S(x-y)\big|, \; |x-y|_{_S}\ge \varepsilon_0\Big\} >0,\\
   (ii)\;\;\hbox{\it Locally negative INILS exponent: } \forall\,m\!\in {{{\cal P}^\star}}, \;\displaystyle \int_{|x-y|_{_S}\le \varepsilon_0}\Lambda_{_S}(x,y)m(dx,dy)  <0,

\end{array}\right.
\]
then $(SDE)$~(\ref{EDS}) and its duplicated system still have $\nu$ and $\nu_{\Delta}$ as unique invariant distributions respectively. 

\end{prop}

\noindent {\em Proof.}  Owing to $(i)$, we have for every $u\!\in (\varepsilon_0,+\infty)$:
\[
\sup_{|x-y|_{_S}= u}\frac{(b(x)-b(y)|x-y)_{_S} +\frac 12\|\sigma(x)-\sigma(y)\|^2_{_S}}{ \Big|(\sigma^*(x)-\sigma^*(y))S\frac{x-y}{|x-y|_{_S}}\Big|^2.}\le \big([b]_{+} +\frac 12 [\sigma]^2_{\rm Lip}\big)\frac{u^4}{\eta^2_0}. 
\]

For every $\varepsilon'_0>\varepsilon_0$, let ${\theta}_{\varepsilon'_0}:(0,+\infty)\rightarrow\ER_+$ denote the continuous function defined by 
\begin{equation*}
\theta_{\varepsilon'_0} (u) =
\begin{cases}
	\big([b]_{S,+} +\frac 12 [\sigma]^2_{\rm Lip}\big)\frac{u^4}{\eta^2_0} &\text{if } u\!\in (\varepsilon'_0,\infty),\\
 1 &\text{if }u\!\in (0,\varepsilon_0],\\
 1+{\theta_{\varepsilon'_0}({\varepsilon'_0})}\frac{t-\varepsilon_0}{\varepsilon'_0-\varepsilon_0}  &\text{if }u\!\in (\varepsilon_0,\varepsilon'_0).
\end{cases}
\end{equation*}
Since $\theta_{\varepsilon'_0}(u)=1$ in the neighbourhood of  $0$, Assumption $(i)$ of Theorem~\ref{thm:echellemultidim} is satisfied. For Assumption $(ii)$,
one first deduces from the construction and to the first assumption that, 
$$\forall \, \varepsilon'_0>\varepsilon_0,\;\forall x,y\; \text{such that }|x-y|_{_S}\ge \varepsilon'_0,\quad \Psi_{\theta_{\varepsilon'_0},S}(x,y)\le 0.$$
Using that $\theta_{\varepsilon'_0}(u)=1$ on $(0,\varepsilon_0]$, it follows that  for all $m\in {{{\cal P}^\star}}$,
$$\int_{\ER^{d}\times\ER^{d} } f'_\theta({|x-y|^2_{_S}})\Psi_{\theta_{\varepsilon'_0},S}(x,y)m(dx,dy)\le \int_{|x-y|_{_S}\le \varepsilon_0}\Lambda_{_S}(x,y)m(dx,dy)+I_2(\varepsilon'_0)$$
where 
$$I_2(\varepsilon'_0)=\int_{\varepsilon_0<|x-y|_{_S}< \varepsilon'_0}f'_\theta({|x-y|^2_{_S}})\Psi_{\theta_{\varepsilon'_0},S}(x,y)m(dx,dy).$$
Since the integrated function is bounded from above on $\{(x,y),\varepsilon_0<|x-y|< \varepsilon'_0\}$, we deduce that 
$$I_2(\varepsilon'_0)\le C \,m(\varepsilon_0<|x-y|< \varepsilon'_0)\xrn{\varepsilon'_0\rightarrow \varepsilon_0}0.$$
By the second assumption of this proposition, it follows that there exists $\varepsilon'_0>\varepsilon_0$ such that
$$\int_{\ER^{d}\times\ER^{d} }f'_\theta({|x-y|^2_{_S}})\Psi_{\theta_{\varepsilon'_0},S}(x,y)m(dx,dy)<0$$
so that Assumption $(ii)$ of Theorem~\ref{thm:echellemultidim} holds. This completes the proof.

$\hfill \Box$

\subsubsection{Local criterion on compact sets}  As mentioned in Remark \ref{remtheta=0}, Theorem \ref{thm:echellemultidim} can be applied under the $(S,0)$-confluence condition~\eqref{asymp-conf1}. This condition is in particular satisfied when 
\begin{equation*}
\forall\, x,\, y\!\in \ER^d,\; x\neq y, \quad (b(x)-b(y)|x-y)_{_S} +\frac 12\|\sigma(x)-\sigma(y)\|_{_S}^2<0.
\end{equation*}

One asset of this more stringent assumption is that it  can be localized in two ways: first in the neighbourhood of  the diagonal  like in the above local criterions, but also on compacts sets of $\ER^{d}\times\ER^{d} $. This naturally leads to a criterion based on the differentials of $b$ and $\sigma$ when they exist.
 
\begin{prop}[Criterion on compact sets]\label{quadralocal} $(a)$ Let $S\in{\cal S}^{++}(d,\ER)$ such that for every $R>0$, there exists $\delta_R>0$ such that $\forall\, x,y\!\in B _{|.|_{_S}}(0;R)$,
\begin{equation}\label{NotsoWAC}
  0<|x-y|_{_S}\le \delta_R \;\Longrightarrow\; (b(x)-b(y)|x-y)_{_S} +\frac 12\|\sigma(x)-\sigma(y)\|_{_S}^2<0.
\end{equation}
Then the diffusion is asymptotically $(S,0)$-confluent.

\smallskip
\noindent $(b)$ If $b$ and $\sigma$ are continuously differentiable, then~(\ref{NotsoWAC}) holds as soon as 
\[
	(AC)^{\hbox{diff}}\equiv \forall\, x\!\in \ER^d,\quad SJ_b(x)+J^*_b (x)S+ \sqrt{S}\sum_{i,j}(\nabla \sigma_{ij}(x))^{\otimes 2} \sqrt{S}<0\;\mbox{ in } \;{\cal S}(d,\ER).
\]
\end{prop}

\noindent {\bf Proof.} $(a)$  
Let $x,y\!\in \ER^d$ such that $x\neq y$. Set $R= \max(|x|_{_S},|y|_{_S})$ and 
\[
x_0=x,\; x_i = x+\frac iN (y-x),\, i=1,\ldots,N-1,\,x_N = y
\]
where  $|y-x|_{_S}<N\delta_R$.  Then for every $i\!\in \{1,\ldots,N\}$, $|x_i|_{_S}\le R$ and $|x_i-x_{i-1}|_{_S}\le \delta_R$. Then 
\begin{eqnarray*}
\|\sigma(x)-\sigma(y)\|_{_S}^2 &= &\Big \|\sum_{i=1}^N\sigma(x_i)-\sigma(x_{i-1})\Big\|_{_S}^2
\le N\sum_{i=1}^N \|\sigma(x_i)-\sigma(x_{i-1})\|_{_S}^2\\
&<& -2 N \sum_{i=1}^N  (b(x_i)-b(x_{i-1})|x_i-x_{i-1})_{_S} 
 =       -2  \sum_{i=1}^N  \big(b(x_i)-b(x_{i-1})| y-x\big)_{_S}\\
 &   < &      -2 (b(y)-b(x)|y-x)_{_S}.      
\end{eqnarray*}

\noindent $(b)$ First, we prove the result when $S=I_d$. We note  that, for every continuously differentiable function $g:\ER^d\to \ER$, $g(y)-g(x)= \int_0^1\big(\nabla g(x+t(y-x)) |y-x\big)dt= \int_0^1(y-x)^*\nabla g(x+t(y-x)) dt$ so that
\[
\big(b(y)-b(x)|y-x\big) =\int_0^1 (y-x)^*J_b(x+t(y-x)\big)(y-x)dt=\int_0^1 (y-x)^*J_b^*(x+t(y-x)\big)(y-x)dt
\]
and
\[
\|\sigma(y)-\sigma(x)\|^2  =\sum_{i,j=1}^d \left(\int_0^1\big(\nabla\sigma_{ij}(x+t(y-x))|y-x\big) dt\right)^2.
\]
By Schwarz's Inequality and the fact that $(u|v)^2= u^*v^{\otimes2}u$, we deduce
\begin{align*}
\big(b(y)-b(x)|y-x\big)+\frac 12 \|\sigma(y)-&\sigma(x)\|^2\le \frac{1}{2}\int_0^1 \Big((y-x)^*(J_b+J_b^*)(x+t(y-x))(y-x) \\
&+\frac 12 \sum_{ij}(y-x)^*\big(\nabla \sigma_{ij}(x+t(y-x))\big)^{\otimes 2}(y-x)\Big)dt.
\end{align*}
This completes the proof  when $S=I_d$.
This extends to  general matrix $S\!\in {\cal S}^{++}(d,\R)$ using that  $\|\sigma(y)-\sigma(x)\|_{_S}^2=\|(\sqrt{S}\sigma)(y)-(\sqrt{S}\sigma)(x)\| ^2$ and the fact that  $(Au)^{\otimes 2}= Au^{\otimes 2}A^*$ with $A= \sqrt{S}$.\hfill $\Box$

\noindent \subsubsection{The case $\Lambda_{_S}\le 0$} 
 As mentioned before, the main field of applications of Corollary \ref{Cor:lyap} seems to be the case $\Lambda_{_S}<0$ out of the diagonal $\Delta_{\ER^{d}\times\ER^{d} }$.
In the two next sections, our objective is to state some results when this condition is not fulfilled. We begin by a simple application of Corollary~\ref{Cor:lyap} where the NILS exponent is only non-positive and  negative outside of a compact set.

\begin{prop}  \label{prop_principale2}  Assume $[b]_{S,+}<+\infty$,  $\sigma$ is Lipschitz continuous and $(SDE)$ has a unique  invariant distribution $\nu$ whose support is not compact. Then, uniqueness for $\nu_{\Delta}$ holds true as soon as 
\begin{equation}\label{locflat}
\forall x,y\,\in\ER^d, \, \Lambda_{_S}(x,y)  \le 0\;\textnormal{ and }\; \exists \, R> 0\mbox{ s.t. } \max(|x|_{_S}, |y|_{_S}) > R \Longrightarrow \Lambda_{_S}(x,y)  < 0.
\end{equation}
\end{prop}
 
\begin{proof}  Since the support of $\nu$ is not compact, we have for every $m\in {{\cal P}_{\nu,\nu}^\star}$: $m\big(\{\max(|x|_{_S}, |y|_{_S}) > R\}\big)\ge \nu(\{|x|_{_S}>R\})>0$. It follows from the assumption that
$$\forall m\in {{\cal P}_{\nu,\nu}^\star},\quad\int \Lambda_{_S}(x,y) m(dx,dy)\le \int \Lambda_{_S}(x,y)1_{\{\max(|x|_{_S}, |y|_{_S}) > R\}} m(dx,dy)<0$$
and we deduce the result from  Corollary~\ref{Cor:lyap}.

\end{proof}
\smallskip \begin{Remarque}
$\rhd$  In the particular case where $\sigma$ is constant, Condition \eqref{locflat} becomes a monotony condition on $b$ (decrease with respect to $(.|.)_{_S}$ at infinity), namely: 
	\begin{align*}
		&\forall x, y\in\ER^d,\;x\neq y\quad(b(x)-b(y)|x-y)_S\le 0,\\
		\textnormal{ and }\quad & \exists \, R> 0\mbox{ s.t. } \max(|x|_{_S}, |y|_{_S}) > R\Longrightarrow (b(x)-b(y)|x-y)_S<0.
\end{align*}
  This means that $b$ is $S$-non-increasing on $\R^d$, $S$-decreasing outside $B _{|.|_{_S}}(0;R)^2$. For instance, if $b=-\nabla U$, the above assumption holds if  $U$ is convex and (only) strictly convex outside of a compact set.

\smallskip
\noindent $\rhd$ {Note that when $\nabla U$ is  only increasing outside $B _{|.|_{_S}}(0;R)$ but possibly with no specific monotony on $B _{|.|_{_S}}(0;R)$, it is still possible to find some diffusion coefficients  $\sigma$ such that the SDE   $dX_t=-\nabla U(X_t) dt+\sigma(X_t) dW_t$ remains weakly or pathwise confluent. We refer to the next subsection for models with such \textit{stochastically stabilizing} diffusive components.   }

\smallskip
\noindent $\rhd$ Finally, note  that the above condition~(\ref{locflat})  can be also localized around the diagonal under {the directional $S$-ellipticity assumption}. To be more precise, when $\nu$ is unique and its support is not compact, Proposition~\ref{prop_principale2} still holds  if Assumption $(ii)$ is ``localized" into:
\begin{equation*}
	(ii)_{loc}\equiv \text{ for every }x,y\in\ER^d\text{ such that }0<|x-y|\le \varepsilon_0,\; \Lambda_{_S}(x,y)\le 0.
\end{equation*}
\end{Remarque}
\subsubsection{$\Lambda_{_S}$ possibly positive on some areas of $\ER^{d}\times\ER^d $}\label{nonpointwise}

In the continuity of the previous section, we try to explore some multidimensional settings where $\Lambda_{_S}$ can be positive in some parts of the space. More precisely, we focus here on gradient systems with constant noise whose potential $U$ is not convex in all the space ({see \cite{tearne} for other confluence results on this type of model with the ``random attractor'' viewpoint}). For such dynamical systems, we obtain a criterion below that we next apply to super-quadratic non-convex potentials. 
Then, we will come back to this problem in   section~\ref{explnilsnonneg} where we focus on the particular example $U(x)=(|x|^2-1)^2$, case for which we are able to obtain a sharper result.

\begin{prop} [Gradient system] \label{pro:gradsyst}Let $U:\R^d\to \R_+$ be a locally Lipschitz, differentiable function satisfying 

 $\displaystyle  0 < \liminf_{|x|\to +\infty} \frac{U(x)}{|x|^{\gamma}}<\displaystyle \limsup_{|x|\to +\infty} \frac{U(x)}{|x|^{\gamma}}<+\infty$ for a positive $\gamma$. Then,  the Brownian diffusion 
\[
dX^x_t =-\nabla U(X^x_t)dt +\sigma dW_t, \;X^x_0=x,
\]
where $\sigma>0$ and $W$  is a standard Brownian motion on $\R^d$,  satisfies a  strong existence-uniqueness property with unique invariant distribution   $\displaystyle \nu_{\sigma}(dx)= C_{\sigma} e^{-\frac{2U(x)}{\sigma^2}}dx$. 

\smallskip
Furthermore assume that its NILS exponent satisfies
\begin{equation} \label{eq: preKantoro}
\forall\, x,\, y\!\in\R^d,\quad \Lambda_{_{Id}}(x,y)\le\beta-\frac{\alpha}{2}\big(|x|^a+|y|^a\big) \quad \mbox{ where } \; \beta\in\ER,\,\alpha,\, a>0
\end{equation}

then there exists $\sigma_c>0$ such that, for every $\sigma>\sigma_c$, the related $(DSDS)$ system ($2$-point motion)  is weakly confluent.

\end{prop}
\begin{proof} The strong existence-uniqueness is classical background.  The form of the invariant distribution $\nu_{\sigma}$ as well. Then by Fatou's Lemma  and the asymptotic upper-bound, there exists $A>0$ such that
\[
\liminf_{\sigma\to +\infty} \int_{\R^d} |u|^a  e^{-\frac{2U(\sigma^{2/\gamma}u)}{\sigma^2}}du\ge \int_{\R^d} |u|^a e^{-A  |u|^{\gamma}}  du>0.
\]
On the other hand, note that
\[
\int_{\R^d} |x|^a e^{-\frac{2U(x)}{\sigma^2}}dx = \sigma^{\frac d2 +\frac{2a}{\gamma}}\int_{\R^d} |u|^a  e^{-\frac{2U(\sigma^{2/\gamma}u)}{\sigma^2}}du. 
\]
Owing to the asymptotic lower bound for $U$ at infinity {and the (reverse) Fatou's Lemma, there exists a real number $B>0$ such that}  
\[
\limsup_{\sigma\to +\infty} \int_{\R^d} e^{-\frac{2U(\sigma^{2/\gamma}u)}{\sigma^2}}du\le \int_{\R^d}   e^{-B|u|^{\gamma}}du<+\infty.
\]
As a consequence $\displaystyle \liminf_{\sigma\to +\infty} \nu_{\sigma}\big(|x|^a\big)=+\infty$. For any distribution $m\!\in {\cal P}(\R^d\times \R^d)$ with marginal $\nu_{\sigma}$ and assigning no weight to the diagonal, one has
\[
\int_{\R^d\times \R^d\setminus \Delta_{\ER^{d}\times\ER^{d} }}\Lambda_{Id}(x,y) m(dx,dy) \le \beta-\alpha \nu_{\sigma} \big(|x|^a\big)<0
\]
as soon as $\sigma$ is large enough to ensure that $\nu_{\sigma} \big(|x|^a\big)\ge \frac{ \beta}{\alpha}$. 
\end{proof}

\begin{Remarque} We may assume without loss of generality that argmin$ _{\R^d} U = \{U=0\}\subset \{\nabla U= 0\}$ so that $\nu_{\sigma}\stackrel{\R^d}{\Longrightarrow} \nu_0 = \mbox{Unif}(\{U=0\})$ as $\sigma \to 0$. Hence, from a practical point of view, the fact that the critical $\sigma_c$ can be taken as $0$ seems a reasonable conjecture if $ \beta-\alpha \nu_0\big(|x|^a\mbox{\bf 1}_{\{U(x)=0\}}\big)\le 0$. Thus, in Section~\ref{explnilsnonneg}, we prove that it holds true for the potential  fonction $U(x)=(|x|^2-1)^2$.
\end{Remarque}
\begin{Corollaire}\label{corgrad} Assume that $U:\ER^d\rightarrow\ER_+$ is defined by $U(x)=C |x|^{2p}+\varepsilon(x)$ where $p>1$, $C>0$ and $\varepsilon$ is a ${\cal C}^1$-function such that $\nabla \varepsilon$ is Lipschitz continuous.
 Then, there exists $\sigma_c>0$ such that, for every $\sigma>\sigma_c$, the $(DSDS)$  related to the gradient system $dX_t=-\nabla U(X_t)dt+\sigma dW_t$  is weakly confluent. 
\end{Corollaire}
\begin{proof} Using that for every $x\in\ER^d$ (even if $x=0$ with an obvious extension by continuity), 
$$
D^2 (|x|^{2p})= 2p |x|^{2(p-1)}\left( 2(p-1)  \frac{x^{\otimes 2}}{|x|^2}  +  I_d\right)\ge 2p |x|^{2(p-1)}I_d\; \mbox{ in } {\cal S}^+(d,\ER),
$$
 we deduce that for every $x\neq y$,
$$\frac{\big(\nabla (|x|^{2p})-\nabla (|y|^{2p})|x-y\big)}{|x-y|^2}\ge 2 p\int_0^1 |y+t(x-y)|^{2(p-1)} dt.
$$
 If $p\ge 2$, we deduce from Jensen's inequality that 
\begin{align*}
\int_0^1 |y+t(x-y)|^{2(p-1)} dt\ge \left(\int _0^1 |y+t(x-y)|^{2} dt\right)^{p-1}&\ge \left(\frac{1}{6}(|x|^2+|y|^2)\right)^{p-1}\\&
\ge \left(\frac{1}{6}\right)^{p-1}\big(|x|^{2(p-1)}+|y|^{2(p-1)}\big)
\end{align*}
where in the last inequality, we used again that $p-1\ge 1$.
It follows that 
$$
\Lambda_{_{Id}}(x,y)\le [\varepsilon]_1-\alpha_p(|x|^{2(p-1)}+|y|^{2(p-1)})
$$
where $[\varepsilon]_1$ denotes the Lipschitz constant of $\varepsilon$ and $\alpha_p>0$. The previous result then applies in this case.

\smallskip
When $p\in(1,2)$, we deduce from the elementary inequality $||u|^\rho-|v|^\rho|\le |u-v|^\rho$ for $0<\rho<1$ 
 that 
$$\int_0^1 |y+t(x-y)|^{2(p-1)} dt\ge \alpha_p \big(|x|^{2(p-1)}+|y|^{2(p-1)}\big)$$
with $\alpha_p>0$ and the result   follows likewise.
\end{proof}

\subsection{Examples}

\subsubsection{An example of confluent diffusion with increasing drift}

Assume that $\sigma:\ER^d\rightarrow{\cal M}(d,d,\R)$ is defined by  $\sigma(x) = x  \otimes \lambda + \sigma^0$ where $\sigma^0\! \in{\cal M}(d,d,\R)$ and $\lambda :\R^d\to \R^d$ is a bounded Lipschitz function (such that $\sigma$ is Lipschitz too). If there exists $\rho\!\in (0,\frac 12)$  such that 
\begin{equation}\label{hypcd}
\limsup_{|x|\to +\infty} \frac{(b(x)|x)-\rho |x|^2|\lambda(x) |^2}{(1+|x|^2)^{\rho+\frac 12}}=-\infty
\end{equation}
then the diffusion~(\ref{EDS}) has at least one invariant distribution $\nu$. Thus, if  
$$[b]^{0}_{+} =\sup_{x\neq 0}\frac{(b(x)-b(0)|x)}{|x|^2}\neq +\infty,$$  the above condition is satisfied as soon as 
\[
\liminf_{|x|\to +\infty} |\lambda(x)|^2 >2[b]^{0}_{+}.
\]
The key is to introduce the Lyapunov function $V(x)= (a+|x|^2)^{\rho+\frac 12}$. Using that $\|(\sigma-\sigma^0)(x)\|^2=| (\sigma-\sigma^0)^*(x)\frac{x}{|x|}|^2= |\lambda(x)|^2|x|^2$, we deduce that
$$
\frac{1}{2}\|\sigma(x)\|^2-(\rho+\frac{1}{2})\Big| \sigma^*(x)\frac{x}{|x|}\Big|^2=-\rho|\lambda(x)|^2|x|^2+O(1)
$$
and it  follows that $\limsup_{|x|\to +\infty} {\cal A}V(x)=-\infty$ if \eqref{hypcd} is fulfilled (where ${\cal A}$ denotes the infinitesimal generator of~(\ref{EDS})).

If the function $\lambda$ {\em  is constant}, the diffusion is asymptotically pathwise confluent (so that $\nu$ is unique  for (\ref{EDS}) and the duplicated system  has $\nu_{\Delta}$ as unique  invariant distribution) as soon as there exists $\varepsilon_0>0$ satisfying

\begin{equation}\label{verifii}
|x-y|\le \varepsilon_0\Longrightarrow   (b(x)-b(y)|x-y)-\frac 12 |\lambda |^2  |x-y|^2< 0.
\end{equation}
This is a consequence of Proposition~\ref{prop_principale} applied with $S=I_d$ (the directional ellipticity assumption~$(i)$ is clearly true since $| (\sigma^*(x)-\sigma^*(y))(x-y)|=|\lambda|.|x-y|^2$). 
 If $b$ is smooth this condition is satisfied as soon as, for every $x\!\in \R^d$,  $\frac12 (J_b+J^*_b)(x) < \frac 12 |\lambda|^2 I_d$ in ${\cal S}(d,\R)$ ($J_b(x)$ denotes the Jacobian matrix of $b$).

\subsubsection{Baxendale's model}
Let  $\Xi_t= (X_t,Y_t)$ be the unique strong solution to the $2$-dimensional SDE
\begin{eqnarray*}
dX_t&=& \big(a-\frac{\sigma^2}{2}\big)X_tdt-(\sigma Y_t-\theta_{_X})dW_t\\
dY_t&=& \big(b-\frac{\sigma^2}{2}\big)Y_t dt +(\sigma X_t+\theta_{_Y})dW_t
\end{eqnarray*}
where $W$ is scalar standard Brownian motion, $a$, $b$, $\sigma$ are real numbers satisfying
\[
ab<0, \; a+b<0,\; \sigma >\sqrt{\frac{2ab}{a+b}}.
\]
and  $\theta_{_X}$, $\theta_{_Y}\!\in \R$. When $\theta_{_X}=\theta_{_Y}=0$, this system is known as Baxendale's system (see~$e.g.$~\cite{KLPL}). Its stochastic stability has been extensively investigated in connection with its Lyapunov exponent.  Then set 
\[
\lambda=\lambda(\sigma)= \frac{b-a+\sqrt{(b-a)^2 +\sigma^4}}{\sigma^2}\!\in(0,1)\;\mbox{ and }\; \alpha = \sigma^2-(a+b)-\sqrt{(a-b)^2+\sigma^4}>0.
\]
and $|\,.\, |_{\lambda}=|\,.\,|_{_S}$ with $S = {\rm Diag}(1,\lambda)$. It\^o's Lemma implies 
\[
d|\Xi_t|_{\lambda}^2= \Big(-\alpha |\Xi_t|_{\lambda}^2 + \theta_{_X}(\theta_{_X}-2\sigma Y_t) +\lambda\theta_{_Y}(\theta_{_Y}+2\sigma X_t)\Big)dt +\Theta(\Xi_t)dW_t
\]
where $ \Theta(x,y)=  2\big((\lambda-1) \sigma xy +\lambda \theta_{_Y} x+   \theta_{_X} y\big)$. It is clear that there exists $\beta\!\in \R_+$ such that 
\[
| \theta_{_X}(\theta_{_X}-2\sigma y )+\lambda\theta_{_Y}(\theta_{_Y}+2\sigma x)|\le \beta (|(x,y)|_{\lambda}+1).
\]
Then using that $|\xi|_{\lambda}\le \frac{1}{2\alpha} +\frac{\alpha}{2}|\xi|_{\lambda}^2$ and setting $\beta' =\beta+\frac{1}{2\alpha}$, we derive that 
\[
d|\Xi_t|_{\lambda}^2\le \beta'-\frac{\alpha}{2}|\Xi_t|_{\lambda}^2 dt +\Theta(\Xi_t)dW_t
\]
where $\theta(\xi)\le C|\xi|_{\lambda}$. Hence, the function $V(\xi)= |\xi|_{\lambda}^2$ is a Lyapunov function for the system since $AV \le \beta'-\frac{\alpha}{2} V$.
 As a consequence there exists at least one invariant distribution $\nu$ for the system and any such distribution satisfies $\nu(V)\le 2\frac{\beta'}{\alpha}$ .

\smallskip At this stage we can compute the non-infinitesimal $S$-Lyapunov exponent of the duplicated system. Tedious although elementary computations show that, for every $\xi= (x,y), \, \xi'=(x',y')\!\in \R^2$, 
\[
\Lambda_{_S}(\xi,\xi') = -\frac{\alpha}{2}- (\lambda-1) ^2\sigma^2 \frac{(x-x')^2(y-y')^2}{|\xi-\xi'|_{\lambda}^4}<0.
\]

\smallskip \begin{Remarque}

 Adapting results from~\cite{BAXBAX} obtained for diffusions on compact manifolds, one easily derive another type of criterion for weak confluence. Namely, if the diffusion $(X^x_t)_{t\ge 0}$ ($1$-point motion) has a unique invariant distribution $\nu$ with support $\R^d$ and if $X^x_t\stackrel{\cal L}{\longrightarrow} \nu$ as $t\to +\infty$ for every $x\!\in \R^d$ and if {\em no nonempty closed connected subset $C$ of $\R^d\times \R^d\setminus \Delta_{\ER^{d}\times\ER^{d} }$ is left stable by the $2$-point motion $(X^{x_1}_t,X^{x_2}_t)_{t\ge 0}$ $((x_1,x_2)\!\in C)$}, then the $2$-point motion is weakly confluent with invariant distribution $\nu_{\Delta}= \nu \circ (x\mapsto (x,x))^{-1}$. However, although more intuitive this criterion seems not to be tractable compared to the above criterions based on the $NILS$ exponent.

\end{Remarque}

\subsubsection{An example of gradient system with a non- convex potential}\label{explnilsnonneg}
Let $U:\R^d\to \R_+$ be defined by $U(x) = \frac{(|x|^2-1)^2}{4}$, $x\!\in \R^d$. Applying Corollary \ref{corgrad} with 
$p=2$ and $\varepsilon(x)=\frac{1}{4}(1-2x^2)$, one deduces that there exists $\sigma_c>0$ such that for every $\sigma>\sigma_c$, 
the $2$-point motion related to  
$dX^x_t =-\nabla U(X^x_t)dt +\sigma dW_t$ is weakly confluent. In fact, for this function, we obtain the weak confluence for every $\sigma>0$.
\begin{prop}\label{resultcaspart} 
Let $U:\R^d\to \R_+$ be defined by $U(x) = \frac{(|x|^2-1)^2}{4}$. Then, for every $\sigma>0$, the (DSDS) related to the Brownian diffusion 
$dX^x_t =-\nabla U(X^x_t)dt +\sigma dW_t$ is weakly confluent.
\end{prop}
\begin{proof}
Elementary computations show that, for every $x,y\!\in \R^d$,  
\[
\Lambda_{_{Id}}(x,y)= 1-\frac 12 \Big(\big(|x|^2+|y|^2\big)+ \frac{(x+y|x-y)^2}{|x-y|^2}\Big)\le 1-\frac 12 \big(|x|^2+|y|^2\big)
\]
so that  for every $m\in {{\cal P}_{\nu,\nu}^\star}$,
\begin{equation}\label{duality}
\int \Lambda_{_{Id}}(x,y)m(dx,dy)<1-\frac{1}{2}\left(\int|x|^2\nu_\sigma(dx)+\int|y|^2\nu_\sigma(dy)\right)=\frac{1}{Z_\sigma}\int (1-|x|^2)e^{- \frac{2U(x)}{\sigma^2}}dx
\end{equation}
with $Z_\sigma= \displaystyle \int _{\R^d} e^{-\frac{2 U(x)}{\sigma^2}}dx$. By Corollary \ref{Cor:lyap}$(a)$, it is now enough to prove that 
$$
\int_{\R^d} (1-|x|^2)e^{- \frac{2U(x)}{\sigma^2}}(dx)<0.
$$
Thanks to a change of variable,
$$
\int_{\R^d} (1-|x|^2)e^{- \frac{2U(x)}{\sigma^2}}dx={\rm Vol}({\cal S}_{d-1})\int_0^{+\infty} (1-r^2) r^{d-1} e^{-\frac{(r^2-1)^2}{2\sigma^2}} dr
$$
where ${\rm Vol}({\cal S}_{d-1})$ denotes the  hyper-volume of the $d-1$-dimensional Euclidean ball. When $d=2$, it follows that
$$
\int_0^{+\infty} (1-r^2) r e^{-\frac{(r^2-1)^2}{2\sigma^2}} dr=\frac{\sigma^2}{2}\left[e^{-\frac{(r^2-1)^2}{2\sigma^2}}\right]_0^{+\infty}=\frac{\sigma^2e^{-\frac{1}{\sigma^2}}}{2}<0.
$$
When $d>2$, note  that $(1-r^2) r^{d-1} \le (1-r^2) r$ for every $r\in[0,+\infty)$ so that
$$
\int_0^{+\infty} (1-r^2) r^{d-1} e^{-\frac{(r^2-1)^2}{2\sigma^2}} dr=\int_0^{+\infty} (1-r^2) r e^{-\frac{(r^2-1)^2}{2\sigma^2}} dr<0.
$$
This completes the proof.
\end{proof}
\subsection{Weak confluence: toward an optimal transport viewpoint}\label{OptiTrans} As a conclusion of this first part of the paper, let us note  that {when $\nu$ is unique,} the question of the negativity of the Integrated NILS exponent on the 
set of probabilities $m\! \in {\cal P}_{\nu,\nu}^\star$ is connected with an optimal transport problem (see~$e.g.$~\cite{villani} for a background on this topic). 

Let us be more precise. Assume that $\Lambda_{_S}$ satisfies~\eqref{eq:Lambdalsc} and let  $\bar{\Lambda}_{_S}:\ER^d\times\ER^d\rightarrow\ER$  denote its upper semi-continuous (u.s.c.) envelope.
If $[b]_{S,+}\!<\!+\infty$ and  $\sigma$ is Lipschitz continuous, $\bar \Lambda_{_S}$ is $[-\infty, C_{b,\sigma}]$-valued where $C_{b,\sigma}$ is a real constant (note that when $b$ and $\sigma$ are continuously differentiable,  the extension on the diagonal has an explicit form obtained by replacing the infimum by a supremum  in~\eqref{formellsc}).
If we slightly  strengthen our criterion~\eqref{eq:NILSmneg} --~negativity of the the INILS exponent on ${\cal P}_{\nu,\nu}^\star$~-- by   also asking that  $\int_{\R^d}\bar{\Lambda}_{_S}(x,x)\nu (dx) <0$~(\footnote{to be compared to the necessary condition~\eqref{prescriterediag}.}) and if we denote  by ${\cal P}_{\nu,\nu}(\ER^d\!\times\!\ER^d)$ the (convex) set of distributions on  $ \ER^d\!\times\!\ER^d$ with marginals $\nu$ on $\R^d$, one   checks   that the more stringent resulting criterion reads
$$
\forall m\!\in {\cal P}_{\nu,\nu}(\ER^d\times\ER^d),\quad \int_{\R^d\times \R^d}\hskip -0.25 cm  \bar{\Lambda}_{_S}(x,y) m(dx,dy)<0.
$$
Owing to the weak compactness of ${\cal P}_{\nu,\nu}(\ER^d\times\ER^d)$  and to the (weak) u.s.c. of  the mapping $m\mapsto \int_{\R^d\times \R^d} \bar\Lambda_{_S}(x,y) m(dx,dy)$, the above criterion is equivalent to
$$
\max\left\{ \int_{\R^d\times \R^d}\hskip -0.25cm  \bar{\Lambda}_{_S}(x,y) m(dx,dy), \;m\!\in {\cal P}_{\nu,\nu}(\ER^d\times\ER^d)\right\}<0.
$$
Thanks to the Kantorovich duality Theorem and the symmetry of $\bar \Lambda_{_S}$, this criterion is in turn  equivalent to 
$$
{\inf}\left\{\int_{\R^d}\! \varphi  \,d \nu,\; \varphi \! \in L^1(\nu),\; \varphi(x)\!+\!\varphi(y)\ge \bar{\Lambda}_{_S} (x,y),\; (x,y)\!\in\ER^d\times\ER^d\right\}<0.
$$
Note that this last formulation of the problem is well-posed   since it only involves the marginal invariant distribution $\nu$. For instance, it could be the starting point to devising numerical methods for testing the weak confluence of the diffusion.
   
Note that the argument derived from~\eqref{duality} can be   viewed as a duality-type argument applied with $\varphi(x)= \frac{1}{2}(1-|x|^2)$ and, more generally, so is the case for the criterion~\eqref{eq: preKantoro} in~Proposition~\ref{pro:gradsyst}.

\section{Application to the Richardson-Romberg extrapolation for the approximation of invariant distributions}\label{sec:RR}
 As an application, we investigate in this section the Richardson-Romberg ($RR$) extrapolation for the approximation of invariant measures.
Roughly speaking,  the aim of a $RR$ method is generally to improve the order of  convergence of an algorithm based on an discretization scheme by cancelling the first order error term induced by the time discretization of the underlying process. However, to be efficient, such a method must be implemented with  a control of its variance. We will see that in this context, this control is strongly linked to the uniqueness of the invariant distribution of the duplicated diffusion.
\subsection{Setting and Background}
\subsubsection{Recursive computation of the invariant distribution of a diffusion: the original procedure }
Following \cite{LP1} and a series of papers cited in the introduction, we consider here a sequence of empirical measures $(\nu_n^{\eta}(\omega,dx))_{n\ge1}$ built as follows: let $(\gamma_n)_{n\ge1}$ denote a non-increasing sequence of positive {\em step parameters}  satisfying
$$
\gamma_n\xrn{n\nrn}0\quad\textnormal{and}\quad\Gamma_n=\sum_{k=1}^{n}\gamma_k\xrn{n\nrn}+\infty.
$$
We denote by $(\bar{X}_n)_{n\ge0}$ the Euler scheme with step sequence $(\gamma_n)_{n\ge1}$ defined by $\bar{X}_0=x\in\ER^d$ and 
$$
\bar{X}_{n+1}=\bar{X}_n+\gamma_{n+1} b(\bar{X}_n)+\sqrt{\gamma_{n+1}}\sigma(\bar{X}_n)U_{n+1}
$$
where 
$(U_n)_{n\ge1}$ is a sequence of i.i.d. centered $\R^q$-valued random vectors such that $\Sigma_{U_1}=I_q$ defined on a probability space $(\Omega,{\cal A}, \PE)$.
The sequence of {\em weighted empirical measures}   $(\nu_n^{\eta}(\omega,dx))_{n\ge1}$ is then defined  
for every  $n\ge1$, by 
$$
\nu_n^{\eta}(\omega,f)=\frac{1}{H_n}\sum_{k=1}^n\eta_k \delta_{\bar{X}_{k-1}(\omega)}
$$
where $\delta_a$ denotes the Dirac mass at $a\!\in \ER^d$ and $(\eta_k)_{k\ge1}$ is a sequence of positive weights such that $H_n=\sum_{k=1}^n\eta_k\xrn{n\nrn}+\infty$. When $\eta_k=\gamma_k$ which corresponds to the genuine case, we will only write $\nu_n(\omega,dx)$ instead of $\nu_n^{\gamma}(\omega,dx)$. 
For this sequence, we recall in Proposition \ref{proprappel} below in a synthesized form the main convergence results   (including rates) of the sequence $ (\nu_n^{\eta}(\omega,dx))$ to the invariant distribution $\nu$ of $(X_t)$. In this way, we introduce two assumptions: \\
\noindent $\mathbf{(S_a)}:(a>0)$ There exists a positive ${\cal C}^2$-function $V:\ER^d\rightarrow\ER$ with 
$$\lim_{|x|\rightarrow+\infty} V(x)=+\infty, \quad |\nabla V|^2\le C V,\quad \textnormal{and }\quad \sup_{x\in\ER^d}\|D^2 V(x)\|<+\infty$$
such that there exist some positive constants $C_{b}$, $\beta$ and $\alpha$ such that:
\begin{align*} 
&\textit{(i)}\quad |b|^2\le C_b V^a, \quad {\rm Tr}(\sigma\sigma^*)(x)=o(V^a(x))\quad\textnormal{as $|x|\rightarrow+\infty$}
&\textit{(ii)}\quad \psg\nabla V| b\psd \le \beta-\alpha V^a.
\end{align*}
This Lyapunov-type assumption is sufficient to ensure the long-time stability of the Euler sheme (in a sense made precise below) as soon as $a\in(0,1]$. Note that the convergence can be obtained under a less restrictive mean-reverting assumption including the case $a=0$ (see \cite{panloup-these}).
The second assumption below is fundamental to establish  the rate of convergence of $(\nu_n^{\eta}(\omega,f))$ to $\nu(f)$ for a fixed smooth enough function $f:\ER^d\rightarrow\ER$: we assume that $f$ has a smooth solution to the Poisson equation (see \cite{Parver1} for results on this topic).

\medskip
\noindent $\mathbf{(C(f,k))}$: There
exists a ${\cal C}^k$-function $g:\ER^{d}\rightarrow\ER$ solution to $f-\nu(f)={\cal A}g$ such that $f$, $g$ and its partial derivatives up to $k$ are dominated by $V^r$ ($r\ge0$):   $|f|\le C V^r$ and for  every $\alpha=(\alpha_1,\ldots,\alpha_d)\!\in \EN^d$
with $|\alpha|:=\alpha_1+\cdots+\alpha_d\in\{0,\ldots,k\}$,
$|\partial^{|\alpha|}_{x^{\alpha_1}_{i_1},\ldots,x^{\alpha_d}_{i_d}} g|\le CV^r$.
\medskip 

\noindent Before recalling the results on $(\nu_n(\omega,dx))$, let us  introduce further notations. We set
$$
\forall\, r\!\in \EN,\qquad \Gamma_n^{(r)}=\sum_{k=1}^n \gamma^{r}_{k}
$$
and for a {smooth enough} function $h:\ER^d\rightarrow\ER$ and an integer $r\ge 2$, we write: 
$$
D^{(r)} h(x)\,y_1\otimes \cdots \otimes y_r=\sum_{(i_1,\ldots,i_r)\in\{1,\ldots,d\}^r}\partial^{r}_{x_{i_1},\ldots,x_{i_r}} h(x) y_{1}^{i_1}\ldots y^{i_r}_r.
$$
\begin{prop}\label{proprappel} Assume $\mathbf{(S_a)}$ holds for an  $a\!\in(0,1]$  {and $U_1\!\in \cap_{p>0}L^p(\PE)$}. Assume that $(\eta_k/\gamma_k)$ is a non-increasing sequence. Then,  

\smallskip
\noindent  $(i)$ For every non-increasing sequence $(\theta_n)_{n\ge1}$ such that $\sum_{n\ge1}\theta_n\gamma_n<+\infty$ and for every $r\!>\!0$, $\sum_{n\ge1}\theta_n\gamma_n \ES[V^r(\bar{X}_n)]<+\infty$.

\smallskip
\noindent  $(ii)$ For every $r\!>\!0$, $\sup_{n\ge1}\nu_n^{\eta}(\omega, V^r)<+\infty$ $a.s.$ In particular, $(\nu_n^{\eta}(\omega, dx))_{n\ge1}$ is $a.s.$ tight.

\smallskip
\noindent  $(iii)$ Every weak limit of $(\nu_n^{\eta}(\omega, dx))_{n\ge1}$ is an invariant distribution for $(X_t)_{t\ge 0}$. Furthermore, if $(SDE)$ has a unique invariant distribution, say $\nu$, then $\nu_n^{\eta}(\omega, f)\xrn{n\nrn}\nu(f)$ $a.s.$ for every $\nu$-$a.s$ continuous function $f$ such that $|f|\le CV^r$ for an $r>0$.

\smallskip
\noindent  $(iv)$ (Rate of convergence when $\eta_k=\gamma_k$): Assume that $\nu$ is unique and that $\ES[U_1^{\otimes 3}]=0$. Let $k\ge1$ such that $f:\ER^d\rightarrow\ER$ satisfies $\mathbf{(C(f,k))}$. Then,

\smallskip
\noindent $\bullet$ If $k=4$ and $\frac{\Gamma_n^{(2)}}{{\sqrt{\Gamma_n}}}\xrn{n \nrn}0$, 
$$
\sqrt{\Gamma_n}\left(\nu_n(\omega,f)-\nu(f)\right)\overset{(\ER)}{\Longrightarrow}{\cal N}\Big(0;\int_{\ER^d}|\sigma^*\nabla g|^2d\nu\Big)\quad \mbox{as}\;n\rightarrow+\infty.
$$
%
\noindent $\bullet$ If $k=5$ and  $\frac{\Gamma_n^{(2)}}{{\sqrt{\Gamma_n}}}\xrn{n \nrn}\widetilde{\beta}\in(0,+\infty]$,
\begin{align*}
\rhd \;&\sqrt{\Gamma}_n\Big(\nu_n(\omega,f)-\nu(f)\Big)\overset{(\ER)}{\Longrightarrow}{\cal N}\Big(\widetilde{\beta}\,m_g^{(1)};\int_{\ER^d}|\sigma^*\nabla g|^2d\nu\Big)\quad \mbox{as}\;n \rightarrow+\infty &\textnormal{ if $\widetilde{\beta}\in (0,+\infty)$},\\
\rhd  \;&\frac{\Gamma_n}{\Gamma_n^{(2)}}\left(\nu_n(\omega,f)-\nu(f)\right)\overset{a.s.}{\longrightarrow} m_g^{(1)}\quad \mbox{as}\;n\rightarrow+\infty
&  \textnormal{if $\widetilde{\beta}= + \infty$\hskip 0,5 cm }
\end{align*}
where $\displaystyle m_g^{(1)}=\int_{\ER^d} \varphi_1d\nu $ with
\begin{equation}\label{eq:varphi1}
\varphi_1(x)=\frac{1}{2} D^2g(x) b(x)^{\otimes 2}+\frac{1}{2} \ES[D^3g(x) b(x)(\sigma(x)U_1)^{\otimes 2}]+\frac{1}{24} \ES[D^4g(x) (\sigma(x)U_1)^{\otimes 4}].
\end{equation}
\end{prop}

The first three claims part $(i)$, $(ii)$ and $(iii)$ of the theorem  follow from \cite{LP2} whereas the $(iv)$ is derived from \cite{LP1} (see
Theorem 10) and \cite{lemairethese} (see Theorem V.3), in which the rate of convergence is established for a wide family
of weights $(\eta_k)$. 

Applying $(iv)$ to polynomial steps of the following form: $\gamma_n= Cn^{-\mu}$, $\mu \!\in(0,1]$,
we observe that the optimal (weak) rate is  $n^{-1/3}$ and is attained for  $ \mu =1/3$. Then
\[
\widetilde \beta = \sqrt{6} \,C^{\frac 32}\; \mbox{ and }\; \sqrt{\Gamma_n} \sim \sqrt{3C/2}\,n^{\frac 13}.
\]
so that
\[
n^{\frac 13} \Big(\nu_n(\omega,f)-\nu(f)\Big)\stackrel{(\ER)}{\Longrightarrow} {\cal N}\Big(2 C\,m_g^{(1)};\frac{2}{3C}\int_{\ER^d}|\sigma^*\nabla g|^2d\nu\Big).
\]
This corresponds to the case where the rate of convergence of the underlying diffusion toward its steady regime  ($\sqrt{\Gamma_n}$ corresponding to $\sqrt{t}$  in the continuous time setting, see~\cite{bhatta82} for the CLT for the diffusion itself) and the discretization error are of the same order. From a practical point of view it seems clear that a balance should be made between the asymptotic bias and the asymptotic variance to specify the constant $C$. Under slightly more stringent assumptions we prove that the $L^2$--norm of the error $\nu_n(\omega,f)-\nu(f)$ satisfies
\begin{equation*}
	\| \nu_n(\omega, f) - \nu(f) \|_{L^2} \sim n^{-\frac{1}{3}} \sqrt{4 C^2 (m_g^{(1)})^2 + \frac{2}{3 C} \int_{\ER^d}|\sigma^*\nabla g|^2d\nu}.
\end{equation*}
An optimisation with respect to $C$ gives the optimal choice $C = \left( \frac{12\int_{\ER^d}|\sigma^*\nabla {g_{_f}}|^2d\nu}{(m_{g_{_f}}^{(1)})^2} \right)^{\frac 1 3}$.

When $\mu\!\in(0,1/3)$, the step sequence decreases too slowly and the error induced by the time discretization error becomes 
prominent. That is why we propose below to use an RR extrapolation in order to cancel the first-order term in the time  discretization error: in practice this amount to  killing the bias $m^{(1)}_g$ in order to extend the range of application of the rate $\sqrt{\Gamma_n}$ (which corresponds to the standard weak rate $\sqrt{t}$ in Bhattacharia's $CLT$) to ``slower steps". 

\subsubsection{The Richardson-Romberg extrapolated algorithm} 
As mentioned before, the starting idea is to introduce a second Euler scheme with step sequence $(\widetilde{\gamma}_n)_{n\ge1}$
defined by 
$$ 
\forall n\ge1,\quad \widetilde{\gamma}_{2n-1}=\widetilde{\gamma}_{2n}={\frac{\gamma_n}{2}}.
$$
As concerns the white noise of both schemes, our aim is to make them consistent in absolute time and correlated (with correlation matrix $\rho$ satisfying  $I_q -\rho^*\rho \!\in {\cal S}^+(d, \ER)$). To achieve that we proceed as follows. 

Let $(Z_n)_{n\ge1}$ be a sequence of $i.i.d.$  $\R^q$-valued random vectors lying in $\cap_{p>0}L^p(\PE)$ and satisfying 
$$
\ES\,Z_1= 0, \quad \Sigma_{Z_1}=I_q, \quad \ES[Z_1^{\otimes 3}]= \ES[Z_1^{\otimes 5}]=0.
$$
Then we devise from this sequence the white noise sequence $(U_n)_{n\ge 1}$ of the ``original" Euler scheme with step $(\gamma_n)_{n\ge 1}$by setting
\begin{equation}\label{noise1}
\forall n\ge1,\quad U_n=\frac{1}{\sqrt{2}}\left(Z_{2n-1}+Z_{2n}\right).
\end{equation}

The white noise sequence for the second Euler scheme (with step $(\widetilde \gamma_n)_{n\ge 1}$), denoted $Z^{(\rho)}$ is defined as follows: 
\begin{equation}\label{noise2}
Z_n^{(\rho)}=\rho^* Z_n+T(\rho) V_n,\; n\ge 1,
\end{equation}
where $(V_n)_{n\ge 1}$ is also   a sequence of $i.i.d.$ centered random variables in $\R^q$ with moments of any order satisfying $\Sigma_{V_1}=I_q$ and $\ES[V_1^{\otimes 3}]= \ES[V_1^{\otimes 5}]=0$, {\em  independent} of $(Z_n)_{n\ge 1}$ and $T_q(\rho)$ is a solution to the equation
\[
T_q(\rho) T_q(\rho)^* =  I_q-\rho^*\rho\!\in {\cal S}^+(d, \ER).
\]
($T_q(\rho)$ can be chosen either as the commuting symmetric square root of $I_q-\rho^*\rho$ or its Choleski transform). Note that $(Z_n^{(\rho)})_{n\ge 1}$ is built in so that it satisfies
$$
\Sigma_{Z_n^{(\rho)}}=I_q\quad\textnormal{and}\quad {\rm Cov}(Z_n,Z_n^{(\rho)})=\rho.
$$
Then the Euler scheme with step $\widetilde \gamma_n$ and consistent $\rho$-correlated white noise  $(Z_n^{(\rho)})_{n\ge 1}$, denoted  $(\bar{Y}_n^{(\rho)})_{n\ge1}$ from now on, is defined by:  
$$
\bar{Y}_{n+1}^{(\rho)}=\bar{Y}_n^{(\rho)}+\widetilde{\gamma}_n b(\bar{Y}_n^{(\rho)})+\sqrt{\widetilde{\gamma}_n}\sigma(\bar{Y}_n^{(\rho)})Z_{n+1}^{(\rho)},\; n\ge1,\; \bar Y_0=y.
$$
Also note that $(\bar{X}_n,\bar{Y}_{2n}^{(\rho)})$ is an Euler scheme at time $\Gamma_n$ of the duplicated diffusion
$(X_t,X_t^{(\rho)})_{t\ge0}$.

For numerical purpose, one usually specifies the independent i.i.d. sequences   $(Z_n)_{n\ge1}$ and $(V_n)_{n\ge 1}$ as  normally distributed so that they can be considered as the normalized increments of two independent Brownian motions $W$ and $\widetilde W$ $i.e.$
\begin{equation*}
Z_n = \frac{W_{\widetilde \Gamma_{n}}-W_{\widetilde \Gamma_{n-1}}}{\sqrt{\widetilde{\gamma}_n}}
\quad \mbox{ and }\quad V_n = \frac{\widetilde W_{\widetilde \Gamma_{n}}-\widetilde W_{\widetilde \Gamma_{n-1}}}{\sqrt{\widetilde \gamma_n}},\; n\ge 1.
\end{equation*}
{Note that in this case, $(U_n)$ is also a sequence of ${\cal N}(0,I_q)$-random variables. This implies  in particular that
\begin{equation}\label{assump-sup}
\ES[U_1^{\otimes 4}]=\ES[Z_1^{\otimes 4}]\quad\textnormal{and}\quad \ES[U_1^{\otimes 6}]=\ES[Z_1^{\otimes 6}].
\end{equation}
Since these properties simplify the result, we will assume them in the sequel of this section (see Remark \ref{exten-romb} for extensions). }  

\medskip 
We denote    $(\nu_n^{\eta,(\rho)}(\omega,dx))_{n\ge1}$ the sequence of empirical measures related to $(\bar{Y}_n^{(\rho)}(\omega))_{n\ge 1}$ (in which the weights are adapted accordingly: $\eta_1/2, \eta_1/2, \eta_2/2, \eta_2/2, \eta_3/2,  \dots$). The empirical measure   $(\bar{\nu}_n^{\eta,(\rho)}(\omega,dx))_{n\ge1}$ associated to the    Richardson-Romberg extrapolation is defined by
 \begin{eqnarray*}
\nu_n^{\eta,(\rho)}(\omega,f)&=&\frac{1}{H_n}\sum_{k=1}^n\frac{\eta_k}{2} \left(f(\bar{Y}_{2(k-1)}^{(\rho)}(\omega))+f(\bar{Y}_{2k-1}^{(\rho)}(\omega))\right)\\
\bar{\nu}_n^{\eta,(\rho)}(\omega,f)&=&(2\nu_n^{\eta,(\rho)}-\nu_n^{\eta}(\omega,f))\\
&=& \frac{1}{H_n}\sum_{k=1}^n \eta_k  \left(f(\bar{Y}_{2(k-1)}^{(\rho)}(\omega))+f(\bar{Y}_{2k-1}^{(\rho)}(\omega))-f(\bar X_k(\omega))\right) .
\end{eqnarray*}
 
Under the assumptions of Proposition \ref{proprappel}, it is clear that $\bar{\nu}_n^{\eta,(\rho)}(\omega,dx)\xrn{n\nrn}\nu(dx)$ $a.s.$. \\
Thus, in the next section, we propose to evaluate the effects of the Richardson-Romberg extrapolation on the rate of convergence of the procedure
and to explain why the uniqueness of the invariant distribution of the duplicated diffusion plays an important role in this problem. 

\subsection{Rate of convergence of the extrapolated procedure}
Throughout this section we assume  that $\eta_k=\gamma_k$ and so we will write $\nu_n$, $\nu_n^{(\rho)}$ and $\bar{\nu}_n^{(\rho)}$ instead of ${\nu}_n^{\eta}$, ${\nu}_n^{\eta,(\rho)}$ and $\bar{\nu}_n^{\eta,(\rho)}$ respectively.
We also set $(D^{3} g_{i,.,.})_{i=1}^d= D^2(\nabla_{\bf{.}}) g$ in order that the notation ${\rm Tr}(\sigma^* D^2(\nabla_{\bf{.}}) g \sigma)$ stands for the {\em vector} of $\ER^d$ defined by  ${\rm Tr}(\sigma^* D^2(\nabla_{\bf{.}}) g \sigma)=({\rm Tr}(\sigma^* D^2(\partial_{x_i} g) \sigma))_{i=1}^d$.
For a fixed matrix $\rho$, the main result about the $RR$ extrapolation is Theorem \ref{theoromberg1} below. At this stage, we do not discuss the choice of the correlation $\rho$ in this result. This point is tackled in  Proposition \ref{propoptchoice} in which we will see that the optimal choice  to reduce the asymptotic variance is atteined with  $\rho=I_q$ as soon as $\nu_\Delta$ is the unique invariant distribution of the associated duplicated diffusion. {This emphasizes the importance of the question of the uniqueness of the invariant distribution in this pathologic case studied in the previous part of the paper}.

\begin{theorem}\label{theoromberg1} Assume $\mathbf{(S_a)}$ holds for an  $a\in(0,1]$. Assume that 
$(X_t,X_t^{(\rho)})_{t\ge0}$ admits a unique invariant distribution $\mu^{(\rho)}$ (with marginals $\nu$). Let $f:\ER^d\rightarrow\ER$ be a function satisfying $\mathbf{(C(f,7))}$ and such that $\varphi_1$ defined by \eqref{eq:varphi1} satisfies $\mathbf{(C(\varphi_1,5))}$ with a solution to the Poisson equation denoted by  $g_{\varphi_1}$.  Then,

\smallskip
\noindent $\bullet$ If $\frac{\Gamma_n^{(3)}}{{\sqrt{\Gamma_n}}}\xrn{n \nrn}0$, 
$$
\sqrt{\Gamma_n}\Big(\nu_n^{(\rho)}(\omega,f)-\nu(f)\Big)\overset{n\rightarrow+\infty}{\Longrightarrow}{\cal N}\big(0; \hat{\sigma}_\rho^2\big)
$$
where 
\begin{equation}\label{eq:varianceRR}
\hat{\sigma}_\rho^2=5\int_{\ER^d}|\sigma^*\nabla g|^2d\nu-4\int_{\ER^{d}\times\ER^d} \big(  (\sigma^*\nabla g)(x)|  \rho(\sigma^*\nabla g)(y)\big)\mu^{(\rho)}(dx,dy).
\end{equation}

\smallskip
\noindent $\bullet$ If $\frac{\Gamma_n^{(3)}}{\sqrt{\Gamma_n}}\xrn{n \nrn}\widetilde{\beta}\in(0,+\infty]$, then
\begin{align*}
\sqrt{\Gamma}_n\big(\nu_n^{(\rho)}(\omega,f)-\nu(f)\big)& \stackrel{(\ER)}{\Longrightarrow} {\cal N}\big(\widetilde{\beta} \,m_g^{(2)}; \hat{\sigma}_\rho^2\big)  \;\mbox{ as } \;n\to +\infty& \textnormal{if $\widetilde{\beta}\in (0,+\infty)$,}\\
\frac{\Gamma_n}{\Gamma_n^{(3)}}(\nu_n^{(\rho)}(\omega,f)-\nu(f))&\xrn{\PE} m_g^{(2)} \;\mbox{ as } \;n\to +\infty &\textnormal{if $\widetilde{\beta}=+\infty$,\hskip 0,55 cm }
\end{align*}
where $\displaystyle m_g^{(2)}=\frac{1}{2}\left(m_{g_{\varphi_1}}+\int_{\ER^d} \varphi_2 d\nu \right)$ with

\begin{equation}\label{eq:varphi2}
\varphi_2(x)=\sum_{k=3}^6 \frac{C_k^{2(k-3)}}{k!}\ES\big[D^k g(x) b(x)^{\otimes (6-k)}(\sigma(x)U_1)^{\otimes 2(k-3)}\big].
\end{equation}
\end{theorem}
\begin{Remarque}\label{exten-romb}
$\rhd$ We recall that the result is stated under the assumption that the increments are normally distributed or more precisely under Assumption \eqref{assump-sup}. When this additional assumption fails (think for instance to $Z_1\sim\big( \frac 12(\delta_{-1}+\delta_1)\big)^{\otimes q}$), the result is remains true except for  the value of $m_g^{(2)}$ which becomes more complicated since it also depends on $\ES[Z_1^{\otimes \ell}]$, $\ell=4$ and $6$).

\noindent $\rhd$ This result extends readily to general weights sequences $(\eta_n)_{n \ge 1}$.Some technical conditions appear on the choice of weights but these conditions are natural and not restrictive (see \cite{lemairethese}). In particular we can always consider the choice $\eta_n = 1$ for which we obtain the following result: if $\frac{\Gamma_n^{(2)}}{\sqrt{\Gamma_n^{(-1)}}} \xrn{n \nrn}\widetilde{\beta}\in (0,+\infty)$, then
\begin{equation*}
\frac{n}{\sqrt{\Gamma^{(-1)}_n}} \big(\nu_n^{(\rho)}(\omega,f)-\nu(f)\big) \stackrel{(\ER)}{\Longrightarrow} {\cal N}\big(\widetilde{\beta} \,m_g^{(2)}; \hat{\sigma}_\rho^2\big)  \;\mbox{ as } \;n\to +\infty.
\end{equation*}

\noindent $\rhd$ {\sc Polynomial steps.} Let $\gamma_n = Cn^{-\mu}$, $\mu \!\in (0, 1]$. If $\mu> \frac 13$, $\Gamma^{(3)}_n \to \Gamma^{(3)}_{\infty}<+\infty$ so that $\frac{\Gamma^{(3)}_n}{\sqrt{\Gamma_n}}\to 0$ as $n\to +\infty$. If $\mu<\frac 13$, $\frac{\Gamma^{(3)}_n}{\sqrt{\Gamma_n}}\asymp n^{\frac{1-5\mu}{2}}$ (and if $\mu=\frac 13$, $\frac{\Gamma^{(3)}_n}{\sqrt{\Gamma_n}}\asymp \frac{\log n}{\sqrt{n}}$). Consequently
\[
\frac{\Gamma^{(3)}_n}{\sqrt{\Gamma_n}}\to 0\; \Longleftrightarrow \; \mu > \frac 15,\;  \frac{\Gamma^{(3)}_n}{\sqrt{\Gamma_n}}\to +\infty \; \Longleftrightarrow\; \mu < \frac 15 \; \mbox{ and }\;\frac{\Gamma^{(3)}_n}{\sqrt{\Gamma_n}}\to \widetilde \beta \!\in (0,+\infty) \; \Longleftrightarrow\; \mu = \frac 15.
\]
When $\mu =\frac15$, $\widetilde \beta = C^{\frac 52}\sqrt{5}$ and $\sqrt{\Gamma_n} \sim \frac{\sqrt{5C}}{2} n^{\frac 25}$.

As a consequence, if $\gamma_n = \eta_n = Cn^{-\frac 15}$, 
\[
n^{\frac 25} \big(\nu_n^{(\rho)}(\omega,f)-\nu(f)\big)\stackrel{(\ER)}{\Longrightarrow} {\cal N}\Big(2 C^2 \,m^{(2)}_g; \frac 45 \frac{\widehat \sigma^2_{\rho}}{C}\Big).
\]
We switch from a weak rate $n^{\frac 13}$ to $n^{\frac 25}$ $i.e.$ a ``gain" of $n^{\frac{1}{15}}$ (see figure below). The second noticeable fact is that the bias is now significantly more sensitive to the constant $C$ than in the  standard setting.
  If we minimize the $L^2$--norm of the error $\nu_n^{(\rho)}(\omega,f)-\nu(f)$ we obtain the optimal choice of $C$ as a function of both bias and standard deviation, precisely $C = \left( \frac{\widehat \sigma^2_{\rho}}{20(m_q^{(2)})^2}\right)^{\frac{1}{5}}$.
	 \begin{figure}[h!t] \label{fig1}
		 \centering \include{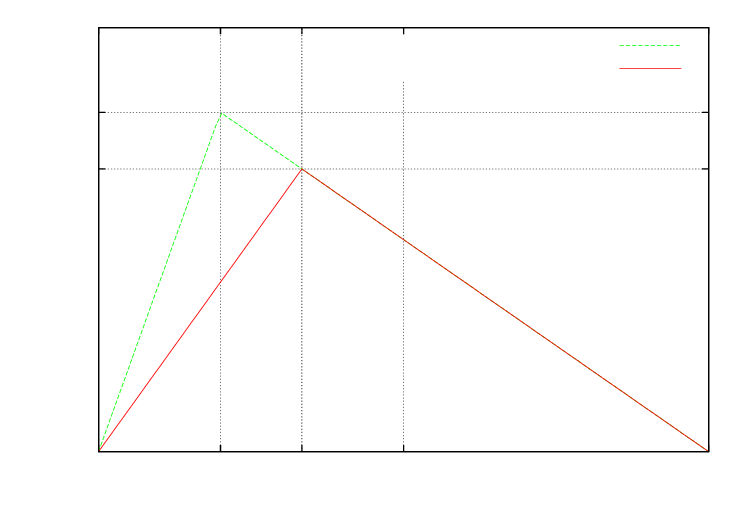}
 \end{figure} 

\end{Remarque}
\subsubsection{Optimal choice of $\rho$ and uniqueness of $\mu^{(I_d)}$}
\begin{prop}\label{propoptchoice} Let $\rho$ be an admissible correlation matrix $i.e.$ such that $\rho^*\rho\le I_q$. Assume that the duplicated diffusion $(X,X^{(\rho)})$ has a unique invariant distribution $\mu^{(\rho)}$ (so that if $\rho=I_q$, $\mu^{(I_q)}= \nu_{\Delta}$).

\smallskip
\noindent $(a)$ $\displaystyle \widehat \sigma^2_{\rho}\ge \int_{\ER^d}|\sigma^*\nabla g|^2 d\nu$.

\smallskip
\noindent $(b)$  If $\rho =0$ then $\displaystyle \widehat \sigma^2_{\rho}= 5 \int_{\ER^d}|\sigma^*\nabla g|^2 d\nu$.

\smallskip
\noindent $(c)$ If $\rho = I_q$, $\displaystyle \widehat \sigma^2_{\rho}=  \int_{\ER^d}|\sigma^*\nabla g|^2 d\nu$.

\end{prop}

\noindent {\em Proof.} Claims $(b)$ and $(c)$ being obvious thanks to  (\ref{eq:varianceRR}), we only prove $(a)$. Keeping in mind that both marginals $\mu^{(\rho)}(\ER^d\times dy)$ and $\mu^{(\rho)}(dx\times \ER^d)$ are equal to $\nu$, one derives  thanks to Schwarz's Inequality  (once on $\ER^d$ and once on $L^2(\mu)$) from the expression~(\ref{eq:varianceRR}) of the asymptotic variance $\widehat \sigma^2_{\rho}$ that
\begin{eqnarray*}
\widehat \sigma^2_{\rho}&\ge & 5 \int_{\ER^d}|\sigma^*\nabla g|^2 d\nu - 4\left[ \int_{\ER^{d}\times\ER^d} |\sigma^*\nabla g|^2(x) \mu^{(\rho)}(dx,dy) \right]^{\frac 12} \left[ \int_{\ER^{d}\times\ER^d} |\rho\sigma^*\nabla g|^2(y) \mu^{(\rho)}(dx,dy)  \right]^{\frac 12}\\
&=&  5 \int_{\ER^d}|\sigma^*\nabla g|^2 d\nu - 4\left[ \int_{\ER^{d}} |\sigma^*\nabla g|^2(x)\nu(dx) \right]^{\frac 12} \left[ \int_{\ER^{d}} |\rho\sigma^*\nabla g|^2(y) \nu(dy)  \right]^{\frac 12}\\
&\ge&  5 \int_{\ER^d}|\sigma^*\nabla g|^2 d\nu - 4\int_{\ER^d}|\sigma^*\nabla g|^2 d\nu = \int_{\ER^d}|\sigma^*\nabla g|^2 d\nu 
\end{eqnarray*}
where we used in the last inequality that $|\rho u|^2\le |u|^2$.  

\noindent {The previous result says that the structural asymptotic variance of the $RR$ estimator is always greater than that of the standard estimator but can be equal if the Brownian motions are equal. This condition is in fact almost necessary. Actually, thanks to the Pythagorean identity,
\begin{align*}
\sigma^2_{\rho}&=5 \int_{\ER^d}|\sigma^*\nabla g|^2 d\nu+2\int_{\ER^{d}\times\ER^d} |\sigma^*\nabla g(x)-\rho\sigma^*\nabla g(y)|^2\mu^{(\rho)}(dx,dy)\\
&-2\int_{\ER^{d}\times\ER^d}|\sigma^*\nabla g(x)|^2\nu(dx)-2\int_{\ER^{d}\times\ER^d}|\rho\sigma^*\nabla g(y)|^2\nu(dy).
\end{align*}
Then, since $\rho^*\rho\le I_q$, a necessary condition to obtain $\sigma^2_{\rho}=\int_{\ER^d}|\sigma^*\nabla g|^2 d\nu$
is 
$$|\rho\sigma^*\nabla g(y)|=|\sigma^*\nabla g(y)|\quad \nu(dy)\mbox{-}a.e.$$
When $\rho^*\rho<I_q$, this equality can not hold except if $\sigma^* \nabla g(y)=0$ $\nu(dy)$-$a.e.$}

\subsection{Proof of Theorem \ref{theoromberg1}}
\subsubsection{Preliminaries}
Without loss of generality, we assume that $f$ satisfies $\nu(f)=0$ so that $f={\cal A}g$ under $\cfk$.
We denote by $\gamma^{(r)}$ the sequence defined by $\gamma^{(r)}_k=\gamma^r_k$.
\begin{lemme}\label{lemmedecomp} Assume that $f$ satisfies $\mathbf{(C(f,7))}$ and denote by $g$ the solution 
to the Poisson equation ${\cal A}g=f$. Then,
\begin{align}
{\Gamma_n}\bar{\nu}_n^{(\rho)}(\omega,f)=&\,2\big(g(\bar{Y}_{2n})\!-\!g(\bar{Y}_0)\big)\!-\!\big(g(\bar{X}_n)\!-\!g(\bar{X}_0)\big)\!-\!\sum_{k=1}^n\!\sqrt{\gamma}_k \big(\sqrt{2}\Delta M_k^{(2)}-\Delta M_k^{(1)}\big)\label{terme1decomp}
\\
&-{\cal E}_n^1-
{\cal E}_n^{2}+N_{n}+R_{n}
\end{align}
where
\begin{align*}
& \Delta M_k^{(1)}=\psg \nabla g(\bar{X}_{k-1})|\sigma(\bar{X}_{k-1}) U_k\psd,\\
&\Delta M_k^{(2)}=\psg \nabla g(\bar{Y}_{2(k-1)})|\sigma(\bar{Y}_{2(k-1)}) Z_{2k-1}\psd+\psg \nabla g(\bar{Y}_{2k-1})|\sigma(\bar{Y}_{2k-1} Z_{2k}\psd,\\
&{\cal E}_n^1=2\sum_{k=1}^n \left(\frac{\gamma_k}{2}\right)^2 \left(\varphi_1(\bar{Y}_{2(k-1)})+\ES[\varphi_1(\bar{Y}_{2k-1})|{\cal F}_{k-1}]\right)-\sum_{k=1}^n \gamma_k^2 \varphi_1(\bar{X}_{k-1}),\\
&{\cal E}_n^2=2\sum_{k=1}^n \left(\frac{\gamma_k}{2}\right)^3 \left(\varphi_2(\bar{Y}_{2(k-1)})+\ES[\varphi_2(\bar{Y}_{2k-1})|{\cal F}_{k-1}]\right)-\sum_{k=1}^n \gamma_k^3 \varphi_2(\bar{X}_{k-1})
\end{align*}
with $\varphi_1$ and $\varphi_2$  defined by \eqref{eq:varphi1} and \eqref{eq:varphi2},\\

%
\noindent $(N_{n})$ is defined by
$$ N_{n}=\sum_{k=1}^n 2\left(\Delta N(\bar{Y}_{2(k-1)},Z_{2k-1},\frac{\gamma_k}{2})+\Delta N(\bar{Y}_{2k-1},Z_{2k},\frac{\gamma_k}{2})\right)-\Delta N(\bar{X}_{k-1},Z_{2k-1},\frac{\gamma_k}{2})$$
where 
$\Delta N(x,U,\gamma)= H(x,U,\gamma)-\ES_x[H(x,U,\gamma)]$ 
and 
\begin{align*}
&H(x,U,\gamma)=\frac{\gamma}{2}D^2 g(x)(\sigma(x)U)^{\otimes 2}+\frac{1}{6}\sum_{\ell=0}^2 C_{3}^{3-\ell}\gamma^{\frac{\ell+3}{2}}D^3 g(x)b(x)^{\otimes \ell}(\sigma(x) U)^{\otimes (3-\ell)}\\
&+\frac{1}{24}\sum_{\ell=0}^1 {\gamma^{\frac{\ell+4}{2}}}C_4^{4-\ell}D^4g(x) b(x)^{\otimes \ell}(\sigma(x)U)^{\otimes (4-\ell)}+\gamma^3\sum_{\ell=4}^6 \frac{C_\ell^{6-\ell}}{\ell!}D^\ell g(x) b(x)^{\otimes (6-\ell)}(\sigma(x)U)^{\otimes \frac{\ell}{2}}.
\end{align*}
Finally, if $\mathbf{(S_a)}$ holds, the sequence $(R_{n})_{n\ge 1}$  satisfies the following property: there exists $r>0$
such that, $a.s.$, for every $n\ge1$,
\begin{equation}\label{eq:contRn}
\ES[|\Delta R_n||{\cal F}_{n-1}]\le C \gamma_n^{\frac{7}{2}} \left(V^r(\bar{X}_{n-1})+V^r(\bar{Y}_{2(n-1)})+V^r(\bar{Y}_{2n-1})\right)
\end{equation}
where $\Delta R_n=R_n-R_{n-1}$.
\end{lemme}

\begin{Remarque} The above decomposition is built as follows: the second term of \eqref{terme1decomp} is the main martingale component of the decomposition whereas ${\cal E}_{n,1}$  contains the first order discretization error. Thanks to the Richardson-Romberg extrapolation, ${\cal E}_{n,1}$ is in fact negligible when $n\rightarrow+\infty$. When the step sequence decreases fast (Theorem\ref{theoromberg1}$(i)$), the rate of convergence is ruled by the main martingale component. In Theorem~\ref{theoromberg1}$(ii)$, the rate is ruled by ${\cal E}_{n,1}$ and ${\cal E}_{n,2}$.  Finally, $N_{n}$ contains all the negligible martingale terms.
\end{Remarque}
\begin{proof}  Owing to $\mathbf{(C(f,7))}$, to the Taylor formula and to the fact that $\ES[D^2(x)(\sigma(x)U_1)^{\otimes 2}]={\rm Tr}(\sigma^*(x)D^2g(x)\sigma(x))$, we have
\begin{align}
g(\bar{X}_k)&=g(\bar{X}_{k-1})+\gamma_k{\cal A} g(\bar{X}_{k-1})+\sqrt{\gamma_k}\Delta M_{k,1}\\
&+\frac{1}{2}\left(D^2(\bar{X}_{k-1})(\sigma(\bar{X}_{k-1})U_k)^{\otimes 2}-\ES[D^2(\bar{X}_{k-1})(\sigma(\bar{X}_{k-1})U_k)^{\otimes 2}|{\cal F}_{k-1}]\right)\\
&+\sum_{l=3}^5 D^l g(\bar{X}_{k-1})\left(\gamma_k b(\bar{X}_{k-1})+\sqrt{\gamma_k}\sigma(\bar{X}_{k-1})U_k\right)^{\otimes l}\label{six}\\
&+D^7 g(\xi_k)\left(\gamma_k b(\bar{X}_{k-1})+\sqrt{\gamma_k}\sigma(\bar{X}_{k-1})U_k\right)^{\otimes 7}\label{sept}
\end{align}
where $\xi_k\in[\bar{X}_{k-1},\bar{X}_{k}]$. The fact that $|\nabla V|^2\le CV$ implies that $\sqrt{V}$ is a Lipschitz continuous function with Lipschitz constant denoted by $[\sqrt{V}]_{\rm Lip}$.
Then, setting $\|D^7 g(x)\|=\sup_{|\alpha|=7}|\partial_\alpha g(x)|$ and using Assumption $\mathbf{(C(f,7))}$, we have
\begin{equation}\label{contevolution}
\|D^7g(\xi_k)\|\le C (\sqrt{V}(\xi_k))^{2r}\le C(\sqrt{V}(\bar{X}_{k-1})+[\sqrt{V}]_{\rm Lip}|\Delta \bar{X}_{k}|)^{2r}
\end{equation}
where $\Delta \bar{X}_{k}=\gamma_k b(\bar{X}_{k-1})+\sqrt{\gamma_k}\sigma(\bar{X}_{k-1}) U_k$. Then, owing to the elementary inequality $|a+b|^p\le c_p(|a|^p+|b|^p)$ and to Assumption $\mathbf{(S_a)}$, it follows that there exists $r>0$ such that
$$
\ES[|\Delta R_n||{\cal F}_{k-1}]\le C\gamma_k^\frac{7}{2} V^r(\bar{X}_{k-1}).
$$
Then we plug this control into the above Taylor expansion and to compensate the terms of~\eqref{six} when necessary.
An appropriate (tedious)  grouping of the terms yields:
\begin{align*}
\gamma_k{\cal A} g(\bar{X}_{k-1})&=g(\bar{X}_k)-g(\bar{X}_{k-1})-\sqrt{\gamma_k}\Delta M_{k,1}\\
&-\gamma_k^2\varphi_1(\bar{X}_{k-1})-\gamma_k^3\varphi_2(\bar{X}_{k-1})-\Delta N(\bar{X}_{k-1},U_k,\gamma_k)-\Delta R_{n}
\end{align*}
where $R_{n,2}$ satisfies \eqref{eq:contRn}. Making the same development for ${\cal A} g(\bar{Y}_{2(k-1)})$ and for ${\cal A} g(\bar{Y}_{2k-1})$ and summing over $n$ yield the announced result.
\end{proof}
\begin{lemme} \label{lemme:MT}Let $a\in(0,1]$ such that $\mathbf{(S_a)}$ holds. Assume that 
$(X_t,X_t^{(\rho)})_{t\ge0}$ admits a unique invariant distribution $\mu^{(\rho)}$. Let $g$ be a ${\cal C}^1$-function such that
$|\nabla g|\le CV^r$ where $r\!\in \ER_+$. Then, 
$$
\frac{1}{\sqrt{\Gamma}_n}\sum_{k=1}^n\sqrt{\gamma}_k(\sqrt{2}\Delta M_k^{(2)}-\Delta M_k^{(1)})\overset{n\nrn}{\Longrightarrow}
\hat{\sigma}_\rho^2.
$$
\end{lemme}
\begin{proof} 
Let $\{\xi_{k,n},k=1,\ldots,n,n\ge1\}$be the triangular array of $({\cal F}_k)$-martingale increments defined by
$$
\xi_{k,n}=\sqrt{\frac{\gamma_k}{\Gamma_n}}(\sqrt{2}\Delta M_k^{(2)}-\Delta M_k^{(1)}).$$ 
Let us show that
$$
\sum_{k=1}^n\ES[|\xi_{k,n}|^2|{\cal F}_{k-1}]\xrn{n\nrn}\hat{\sigma}_\rho^2.
$$ First, using that $\Sigma_{U_1}=I_q$, we obtain that for every $k\ge1$,
$$
\ES[|\Delta M_k^{(1)}|^2|{\cal F}_{k-1}]=|\sigma^*\nabla g(\bar{X}_{k-1})|^2.
$$
Since $x\mapsto|\sigma^*\nabla g|^2(x)$ is a continuous function such that $|\sigma^*\nabla g|^2\le CV^r$ for a positive $r$, it follows from Proposition \eqref{proprappel} that
\begin{equation}\label{mart:conv1}
\frac{1}{\Gamma_n}\sum_{k=1}^n\gamma_k \ES[|\Delta M_k^{(1)}|^2|{\cal F}_{k-1}]\xrn{\nrn}\int_{\ER^d}|\sigma^*\nabla g|^2(x)\nu(dx).
\end{equation}
Similarly,
\begin{equation*}
\ES[|\Delta M_k^{(2)}|^2|{\cal F}_{k-1}]=|\sigma^*\nabla g(\bar{Y}_{2(k-1)})|^2+\ES[|\sigma^*\nabla g(\bar{Y}_{2k-1})|^2|{\cal F}_{k-1}].
\end{equation*}
It follows that 
\begin{equation}\label{mart:conv21}
\frac{1}{\Gamma_n}\sum_{k=1}^n\gamma_k \ES[|\Delta M_k^{(2)}|^2|{\cal F}_{k-1}]=2\nu_n^{(\rho)}(\omega,|\sigma^*\nabla g|^2)-\frac{1}{\Gamma_n}\sum_{k=1}^n\zeta_k
\end{equation}
where $(\zeta_k)$ is a sequence of $({\cal F}_k)$-martingale increments defined by
$$
\zeta_k=\gamma_k\left(|\sigma^*\nabla g(\bar{Y}_{2k-1})|^2-\ES[|\sigma^*\nabla g(\bar{Y}_{2k-1})|^2|{\cal F}_{k-1}]\right).
$$
Using that $|\sigma^*\nabla g|^2\le CV^r$ for a positive real    number $r$, we obtain by similar arguments to those used in~\eqref{contevolution}
that $\ES[|\zeta_k|^2|{\cal F}_{k-1}]\le CV^{2r}(\bar{Y}_{2(k-1)})$. We derive from Proposition~\ref{proprappel}$(i)$ applied with $\theta_k=\frac{1}{\Gamma_k^2}$ that
$$
\sum_{k=1}^{+\infty}\ES[\left|\frac{\zeta_k}{\Gamma_k}\right|^2|{\cal F}_{k-1}]\le C\gamma_1\sum_{k=1}^{+\infty}\frac{\gamma_k}{\Gamma_k^2}V^{2r}(\bar{Y}_{2(k-1)})<+\infty
$$
since 
$$\sum_{k\ge1}\frac{\gamma_k}{\Gamma^2_k}\le 1+\sum_{k=2}^{+\infty}\int_{\Gamma_{k-1}}^{\Gamma_{k}}\frac{ds}{s^2}\le1+\int_{\Gamma_1}^{+\infty}\frac{ds}{s^2}<+\infty.
$$

As a consequence $(\sum_{k=1}^n\frac{\zeta_k}{\Gamma_k})_{n\ge1}$ is a convergent martingale and  the Kronecker Lemma then implies that
$\frac{1}{\Gamma_n}\sum_{k=1}^n \zeta_k\xrn{n\nrn}0$ $a.s$. Thus, we deduce from \eqref{mart:conv21} combined with Proposition
\ref{proprappel} that
\begin{equation}\label{mart:conv2}
\frac{1}{\Gamma_n}\sum_{k=1}^n\gamma_k \ES[|\Delta M_k^{(2)}|^2|{\cal F}_{k-1}]\xrn{n\nrn}2\nu(|\sigma^*\nabla g|^2)\quad a.s.
\end{equation}
Finally, we have to manage the cross-product: keeping in mind the construction of the noises of the Euler schemes (see \eqref{noise1}
and \eqref{noise2}, we have:
\begin{align*}
&\sqrt{2}\ES[\Delta M_k^{(1)}\Delta M_{k,2}|{\cal F}_{k-1}]=\psg (\sigma^*\nabla g)(\bar{X}_{k-1})|\rho(\sigma^*\nabla g)(\bar{Y}_{2(k-1)})\psd\\
&+\psg (\sigma^*\nabla g)(\bar{X}_{k-1})|\rho(\sigma^*\nabla g)(\bar{Y}_{2k-1})\psd-\gamma_k^{-1}\zeta_k^{(2)}
\end{align*}
where 
$$\zeta_k^{(2)}=\gamma_k\left(\psg (\sigma^*\nabla g)(\bar{X}_{k-1})|\rho(\sigma^*\nabla g)(\bar{Y}_{2k-1})\psd-\ES[\psg (\sigma^*\nabla g)(\bar{X}_{k-1})|\rho(\sigma^*\nabla g)(\bar{Y}_{2k-1})\psd|{\cal F}_{k-1}]\right)$$
so that
\begin{equation}\label{martconv31}
\frac{1}{\Gamma_n}\sum_{k=1}^n \sqrt{2}\ES[\Delta M_k^{(1)}\Delta M_{k,2}|{\cal F}_{k-1}]=\mu_n^{(1)}(\psi)+\mu_n^{(2)}(\psi)-\frac{1}{\Gamma_n}\sum_{k=1}^n\zeta_k^{(2)}
\end{equation}
where $\psi:\ER^{2d}\rightarrow\ER$ is defined by $\psi(x,y)=\psg  \sigma^*\nabla g(x)|\rho(\sigma^*\nabla g)(y)\psd $
and for every Borel function $f:\ER^{2d}\rightarrow\ER$,
$$
\mu_n^{(1)}(f)=\frac{1}{\Gamma_n}\sum_{k=1}^n \gamma_k f(\bar{X}_{k-1},\bar{Y}_{2(k-1)})\quad\textnormal{and}\quad
\mu_n^{(2)}(f)=\frac{1}{\Gamma_n}\sum_{k=1}^n \gamma_k f(\bar{X}_{k-1},\bar{Y}_{2k-1}).
$$
By straightforward adaptations of the proof of Proposition \ref{proprappel}, we can show that if $(X_t,X_t^{(\rho)})$
has a unique invariant distribution $\mu^{(\rho)}$ then, for every continuous function $f$ such that 
$f\le CV^r$ with $r>0$,
$$\mu_n^{(i)}(\omega,f)\xrn{n\nrn}\mu^{(\rho)}(f) \quad a.s. \quad\textnormal{with $i=1,2$}.$$
As a consequence, $\mu_n^{(1)}(\psi)+\mu_n^{(2)}(\psi)\xrn{n\nrn}2\mu(\psi)$ $a.s$. Finally, by  martingale arguments similar to those used for $(\zeta_k)$, one checks that $\Gamma_n^{-1}\sum_{k=1}^n \zeta_k^{(2)}\xrn{n\nrn}0$ $a.s.$
Thus, by \eqref{mart:conv1}, \eqref{mart:conv2} and \eqref{martconv31}, we obtain 
that
$$    \sum_{k=1}^n\ES[|\xi_{k,n}|^2|{\cal F}_{k-1}]\xrn{n\nrn}5\nu(|\sigma^*\nabla g|^2)-4\mu(\psi)=\hat{\sigma}_\rho^2.$$
Then, the result follows from the CLT for arrays of martingale increments provided that a Lindeberg-type condition is satisfied (see \cite{hall}, Corollary 3.1).  To be precise, it is enough to prove that there exists $\delta>0$ such that
\begin{equation}\label{eq:lindeberg}
\sum_{k=1}^n\ES[|\xi_{k,n}|^{2+\delta}|{\cal F}_{k-1}]\xrn{n\nrn}0\quad a.s.
\end{equation}
Using Assumption $\mathbf{(S_a)}$ and the fact $|\nabla g|\le C V^r$ ($r>0$), one can check that there exists $r>0$ such that
$$\ES[|\xi_{k,n}|^{2+\delta}|{\cal F}_{k-1}]\le C\frac{\gamma_k^{1+\delta}}{\Gamma_n^{1+\delta}}\left(V^r(\bar{X}_{k-1})+V^r(\bar{Y}_{2k-1})+V^r(\bar{Y}_{2(k-1)})\right).$$
Thus,
$$\sum_{k=1}^n\ES[|\xi_{k,n}|^{2+\delta}|{\cal F}_{k-1}]\le C\frac{\Gamma_n^{(1+\delta)}}{\Gamma_n^{1+\delta}}\left(\nu_n^{\gamma^{(1+\delta)}}(V^r)+\nu_n^{\gamma^{(1+\delta)},(\rho)}(V^r)\right).$$
Checking easily that $\frac{\Gamma_n^{(1+\delta)}}{\Gamma_n^{1+\delta}}\xrn{n\nrn}0$, \eqref{eq:lindeberg} follows from Proposition 
\ref{proprappel}$(ii)$.
\end{proof}
\begin{lemme}\label{lemme:EF} Let $a\in(0,1]$ such that $\mathbf{(S_a)}$ holds. Assume that 
$(X_t)$ admits a unique invariant distribution $\nu$. Assume $\cfk$ and that $\Gamma_n^{(3)}\xrn{n\nrn}+\infty$. Then,\\
(i) If $\varphi_1$ defined by \eqref{eq:varphi1} satisfies $\mathbf{({ C}(\varphi_1,5))}$ then, 
$$\frac{1}{\Gamma_n^{(3)}}{\cal E}_{n,1}\xrn{\nrn}-\frac{1}{2}m_{g_{\varphi_1}}^{(1)}\quad a.s.$$\\
(ii)  If the derivatives of $g$ up to order $6$ are continuous and dominated by $V^r$ (with $r>0$),
$$\frac{1}{\Gamma_n^{(3)}}{\cal E}_{n,2}\xrn{\nrn}-\frac{1}{2}\nu(\varphi_2)\quad a.s.$$
\end{lemme}
\begin{proof} \textit{(i)} Writing
$$2\sum_{k=1}^n \left(\frac{\gamma_k}{2}\right)^2 \left(\varphi_1(\bar{Y}_{2(k-1)})+\ES[\varphi_1(\bar{Y}_{2k-1})|{\cal F}_{k-1}]\right)=
\sum_{k=1}^n \frac{\gamma_k^2}{2}\left(\varphi_1(\bar{Y}_{2(k-1)})+\varphi_1(\bar{Y}_{2k-1})\right)+\sum_{k=1}^n \frac{\gamma_k^2}{2}\Delta T_k $$
with  $\Delta T_k$ being a martingale increment defined by   $\Delta T_k=\ES[\varphi_1(\bar{Y}_{2k-1})|{\cal F}_{k-1}]-\varphi_1(\bar{Y}_{2k-1})$, one obtains that 
$${\cal E}_{n,1} =
\Gamma_n^{(2)}\left[(\nu_n^{\gamma^{(2)},(\rho)}-\nu)(\varphi_1)-(\nu_n^{\gamma^{(2)}}-\nu)(\varphi_1)\right]+\sum_{k=1}^n \frac{\gamma_k^2}{2}\Delta T_k .$$
Applying Theorem V.3 of \cite{lemairethese} (which is an extension of Proposition \ref{proprappel}$(iv)$ to general weights) with $\eta_k=\gamma_k^2$ and $q^*=4$, we obtain that
$$\frac{\Gamma_n^{(2)}}{\Gamma_n^{(3)}} (\nu_n^{\gamma^{(2)}}-\nu)(\varphi_1)\xrn{n\nrn} m_{g_{\varphi_1}}\in\ER\quad\textnormal{in probability.}$$
Similarly, applying this result to the Euler scheme with half-step, we have:
$$\frac{\Gamma_n^{(2)}}{\Gamma_n^{(3)}} [(\nu_n^{\gamma^{(2),(\rho)}}-\nu)(\varphi_1)]= \frac{1}{2}\frac{\Gamma_n^{(2)}}{\sum_{k=1}^n\gamma_k^2.\frac{\gamma_k}{2}} [(\nu_n^{\gamma^{(2),(\rho)}}-\nu)(\varphi_1)]\xrn{n\nrn} \frac{1}{2}m_{g_{\varphi_1}}\in\ER\quad\textnormal{in probability.}$$
Thus, it remains to show that the martingale term is negligible. We set $\theta_k=\frac{\gamma_k^3}{\Gamma_k^{(3)^2}}$. Using that $(\gamma_k)$ is non-increasing, one checks that $(\theta_n)$ is non-increasing and that $\sum\theta_k\gamma_k<+\infty$. Since $|\varphi_1|\le CV^r$ with $r>0$, it follows from Proposition \ref{proprappel} that 
$$\sum_{k\ge1} \frac{\gamma_{k}^4}{(\Gamma_k^{(3)})^2}\ES[|\varphi_1|^2(\bar{Y}_{2k-1})]<+\infty.
$$
This implies that the martingale $\sum \frac{\gamma_k^2}{\Gamma_k^{(3)}}\Delta T_k$ is $a.s.$ convergent so that  the Kronecker lemma yields
$\frac{1}{\Gamma_n^{(3)}}\sum_{k=1}^n\gamma_k^2 \Delta T_k\xrn{n\nrn}0$ $a.s.$. The first assertion follows.

\smallskip
\noindent \textit{(ii)} Remark that
$$
{\cal E}_{n,2}=\frac{1}{2}\nu_n^{\gamma^{(3)},(\rho)}-\nu_n^{\gamma^{(3)}}(\omega,\varphi_2)+\sum_{k=1}^n \frac{\gamma_k^3}{4}T_k.
$$
Under the assumptions, $\varphi_2$ is continuous and dominated by $V^r$ with a positive $r$. Then, since $\Gamma_n^{(3)}\xrn{n\nrn}+\infty$, $(\nu_n^{\gamma^{(3)},(\rho)}(\varphi_2))_{n\ge1}$ and $(\nu_n^{\gamma^{(3)}}(\omega,\varphi_2))_{n\ge1}$ converge to $\nu(\varphi_2)$. With some similar arguments as previously, one checks that the martingale term is negligible and the second assertion follows.
 \end{proof}
\subsubsection{Proof of Theorem \ref{theoromberg1}}
{For the sake of simplicity, we choose to give the proof of Theorem \ref{theoromberg1} only when $\Gamma_n^{(3)}\xrn{n\nrn}+\infty$. Note that if $\gamma_n=Cn^{-\mu}$, this corresponds
to $\mu\le 1/3$, $i.e.$ the case where the Romberg extrapolation really increases the rate of convergence (see Remark \ref{exten-romb}).}
\smallskip

By the decomposition of Lemma  \ref{lemmedecomp} and the convergences established in Lemmas \ref{lemme:MT} and \ref{lemme:EF}, one checks that it is now enough to prove the following points:
\begin{equation}\label{eq:remainder1}
\frac{\Theta_n}{{\Gamma_n}}\left(2\left(g(\bar{Y}_{2n})-g(\bar{Y}_0\right)-(g(\bar{X}_n)-\bar{X}_0)\right)\xrn{\PE}0\quad \textnormal{as $n\nrn$},
\end{equation}
\begin{equation}\label{remainder2}
\frac{\Theta_n}{\Gamma_n} N_n\xrn{\PE}0\quad\textnormal{and}\quad \frac{\Theta_n}{{\Gamma_n}} R_n\xrn{\PE}0\quad \textnormal{as $n\nrn$}
\end{equation}
with $\Theta_n=\sqrt{\Gamma_n}\vee \frac{\Gamma_n}{\Gamma_n^{(3)}}$. 

\medskip
For \eqref{eq:remainder1}, the result is obvious when $g$ is bounded. Otherwise, we use Lemma 3 of \cite{LP2} which implies in particular that for every $p>0$, $\ES[V^p(\bar{X}_n)]\le C_p \Gamma_n$. By Jensen's inequality, this implies that for every $r>0$ and $\alpha\in(0,1]$, there exists a constant $C>0$ such that
$$
\forall \,n\ge1,\quad \ES[V^r(\bar{X}_n)]\le  \big(\ES[V^{\frac{r}{\alpha}}(\bar{X}_n)]\big)^{\alpha} \le C_{\frac{r}{\alpha}} ^{\alpha}\Gamma_n^{\alpha}.
$$
Thus, since the same property holds for the $(\bar{Y}_n)$ and since $|g|\le CV^r$ with $r>0$, \eqref{remainder2} follows taking $\alpha\in(0,1/2)$.\\

\noindent For the first assertion of \eqref{remainder2}, we use a martingale argument. We denote by 
$\{\pi_{k,n},k=1,\ldots,n,n\ge1\}$ the triangular array of $({\cal F}_k)$-martingale increments defined by
$$\pi_{k,n}=\frac{\Delta N_k}{\sqrt{\Gamma_n}}.$$ 
Then, in order to prove the convergence in probability of $(N_n/\sqrt{\Gamma_n})$ to $0$, we use the CLT for martingale increments which says that, since a Lindeberg-type condition holds (we do not prove this point, see Proof of Lemma \ref{lemme:MT} for a similar argument), it is enough to show that
\begin{equation}\label{eq:TCL0}
\sum_{k=1}^n\ES[|\pi_{k,n}|^2|{\cal F}_{k-1}]\xrn{n\nrn}0\quad a.s.
\end{equation}
Under the assumptions on $g$ and on the coefficients, one checks that there exists $r>0$ such that 
$$\sum_{k=1}^n\ES[|\pi_{k,n}|^2|{\cal F}_{k-1}]\le C\frac{1}{\Gamma_n}\sum_{k=1}^n \gamma_k^2 \left(V^r(\bar{X}_{k-1})+V^r(\bar{Y}_{2(k-1)})+V^r(\bar{Y}_{2k-1})\right).$$
By Proposition \ref{proprappel}, $\sup_{n\ge1} \left(\nu_n(\omega,V^r)+\nu_n^{(\rho)}(\omega,V^r)\right)<+\infty$. Assertion \eqref{eq:TCL0} follows.\\

\noindent As concerns $R_n$, it follows from a martingale argument that
$$\frac{1}{{\Gamma}_n^{(3)}}\sum_{k=1}^n\left(\Delta R_k-\ES[\Delta R_k|{\cal F}_{k-1}]\right)\xrn{\PE}0\quad\textnormal{as $n\rightarrow+\infty$}.$$
Now, since $\sup_{n\ge1} \left(\nu_n^{\gamma^3}(\omega,V^r)+\nu_n^{\gamma^3, (\rho)}(\omega,V^r)\right)<+\infty$ $a.s.$ and since $\gamma_n\xrn{n\nrn}0$, we deduce  that 
$$ \frac{1}{\Gamma_n^{(3)}}\sum_{k=1}^n\ES[\Delta R_k|{\cal F}_{k-1}]\xrn{\PE}0.$$
The last assertion follows.

\appendix
\section{Hypo-ellipticity of the correlated duplicated system}
It is a well-known fact that, for a Markov process, the strong Feller property combined with some irreducibility of the transitions
implies uniqueness of the invariant distribution (see $e.g.$ \cite{daprato}, Theorem 4.2.1).  For a diffusion process with smooth coefficients, such properties hold if it  satisfies the hypoelliptic H\"ormander assumption (see \cite{hormander1,hormander2}) and if the deterministic system related to the stochastic differential system (written in the Stratanovich sense) is controllable. In fact, both properties can be transferred from the original $SDE$ to the duplicated system so that its invariant distribution is also unique. The main result of this section is Proposition \ref{prop:hormanderdup}. Before, we need to introduce some H\"ormander-type notations. First, written in a Stratonovich way, $X$ is a solution to 
\begin{eqnarray}\label{sdestrat}
dX_t&=&A_0(X_t)dt+\sum_{j=1}^q A_j(X_t)\circ dW_t^j
\end{eqnarray} 
where $A_0,\ldots A_q $ are  vectors fields on $\ER^d$ defined by\footnote{With a standard abuse of notation,
we identify the vectors fields and the associated differential operators.}:
$$A_0(x)=\sum_{i=1}^d \left[b_i(x)-\frac{1}{2}\sum_{l,j}\sigma_{l,j}(x)\partial_{x_j}\sigma_{i,l}(x)\right]\partial_{x_i}$$
and for every $j\in\{1,\ldots,q\}$:
$$A_j(x)=  \sum_{i=1}^{d}\sigma_{i,j}(x)\partial_{x_i}.$$
For the sake of simplicity, we assume that $b$ and $\sigma$ are ${\cal C}^\infty$ on $\ER^d$ with bounded derivatives. We will also assume the following H\"ormander condition at each point: there exists $N\in\EN^*$ such that $\forall x\!\in\ER^d$,
\begin{equation}\label{hormand_assump}
 {\rm dim}\left({\rm Span}\left\{A_1(x),A_2(x),\ldots,A_q(x),\;\textnormal{L. B. of length $\le N$ of the $A_j(x)$'s  } , 0\le j\le q\right\}\right)=d
\end{equation}
where ``L.B.'' stands for Lie Brackets.
The above assumptions imply that for every $t>0$ and $x\!\in\ER^d$, $P_t(x,.)$ admits a density $p_t(x,.)$ $w.r.t.$ the Lebesgue measure and that $(x,y)\mapsto p_t(x,y)$ is ${\cal C}^\infty$ on $\ER^d\times\ER^d$ (see  $e.g.$ \cite{cattiaux}, Theorem 2.9). In particular, $x\mapsto P_t(x,.)$ is a strong Feller semi-group. 
Assume also that the control system (associated with \eqref{sdestrat})
\begin{equation}\label{deterministic_cs}
\dot{x}^{(u)}=A_0(x^{(u)})+\sum_{j=1}^q A_q (x^{(u)}) u_j, 
\end{equation}
is $approximatively$-controllable:  
\begin{equation}\label{defcontrolable}
\begin{split}
&\textnormal{There exists  $T>0$ such that for every $\varepsilon>0$, $x_1,x_2\in\ER^d$, there exists $u\in L^2([0,T],\ER^d)$}\\
&\textnormal{such that $(x^{(u)}(t))$ solution to \eqref{deterministic_cs} satisfies $x(0)=x_1$ and $|x(T)-x_2|\le\varepsilon$. }
\end{split}
\end{equation}
Under Assumptions \eqref{hormand_assump} and \eqref{defcontrolable},  the diffusion has a unique invariant distribution $\nu$.
Actually, the controllability assumption combined with the Support Theorem implies that for every non-empty open set $O$, for every $x\!\in\ER^d$, $P_T(x,O)>0$. The semi-group $(P_t)$ is then irreducible. Owing to the strong Feller property, it follows classically that 
$(P_t)$ admits a unique invariant distribution (see  $e.g.$ \cite{daprato}, Proposition 4.1.1. and Theorem 4.2.1.).\\
Furthemore, $\nu$ is absolutely continuous with respect to the Lebesgue measure on $\ER^d$ and its topological support is $\ER^d$ (since 
for every open set $O$ of $\ER^d$, $\nu(O)=\int P_T(x,0)\nu(dx)>0$).  
Let us now consider the duplicated diffusion $(X_t,X_t^{(\rho)})$. Setting $Z_t^{(\rho)}=(X_t,X_t^{(\rho)})$ and using the preceding notations, \eqref{eds_duplicated}  can be written:
\begin{equation*}
dZ_t^{(\rho)}=\widetilde{A}_0(Z_t^{(\rho)})dt+\sum_{j=1}^{q}\widetilde{A}_j(Z_t^{(\rho)})dW_t^j+\sum_{j=1}^{q}\widetilde{A}_{d+j}(Z_t^{(\rho)})d\widetilde{W}_t^j
\end{equation*}
where $\widetilde{A}_0(z)=(A_0(x),A_0(y))^T$ (with $A_0(y)=\sum_{i=1}^d  \left[b_i(y)-\frac{1}{2}\sum_{l,j}\sigma_{l,j}(y)\partial_{y_j}\sigma_{i,l}(y)\right]\partial_{y_i}$ and
$z=(x,y)$),   $\widetilde{W}$ is a $d$-dimensional Brownian Motion independent of $W$ such that $W^{(\rho)}=\rho^* W+(I_q-\rho^*\rho)^{\frac{1}{2}} \widetilde{W}$ and for every $j\in\{1,\ldots,q\}$,
\begin{equation}\label{eq:ajs}
\widetilde{A}_j(z)=A_j(x)+ A_j^{(\rho)}(y)\quad\textnormal{and,}\quad
\widetilde{A}_{q+j}(z)= A_j^{((I_q-\rho^*\rho)^{\frac{1}{2}})}(y)
\end{equation}
where for a for a $q\times q$ matrix $B$, $A_j^{(B)}(y)=\sum_{i=1}^d(\sigma(y)B)_{i,j}\partial_{y_i}$. Then, the following property holds.
\begin{prop}\label{prop:hormanderdup} Let $\rho\in \mathbb{M}_{q,q}(\ER)$ such that $\rho^*\rho<I_q$. Assume that $b$ and $\sigma$ are ${\cal C}^\infty$ on $\ER^{d}$ with bounded derivatives. Assume \eqref{hormand_assump} and \eqref{defcontrolable}. Then, uniqueness holds for the invariant distribution $\nu^{(\rho)}$ of the duplicated diffusion $(X_t,X_t^{(\rho)})$. Furthermore, if $\nu^{(\rho)}$ exists, then $\nu^{(\rho)}$ has a density $p^{(\rho)}$ ($w.r.t.$ $\lambda_{2d}$) which is $a.s.$ positive.
\end{prop}
\begin{proof} First, let us check the H\"ormander conditions for $(X_t,X_t^{(\rho)})_{t\ge0}$. Setting $S=(I_q-\rho^*\rho)^{\frac{1}{2}}$, standard computations yield
$$\forall j\in\{1,\ldots,q\},\quad\widetilde{A}_{q+j}(z)=\sum_{l=1}^q S_{l,j} A_l(y).$$
Since $S$ is invertible, we deduce that $\{A_l(y),l=1,\ldots,q\}$ belongs to ${\rm Span }\{\widetilde{A}_{q+j}(z), j=1,\ldots,q\}$.
Similarly, checking that for every $j\in\{1,\ldots,q\}$,
$$[\widetilde{A}_0(z),\widetilde{A}_{q+j}(z)]=[A_0(y),A_j^{(S)}(y)]=\sum_{l=1}^q S_{l,j}[A_0(y),A_l(y)]$$
one deduces from the invertibility of $S$ that $\{[A_0(y),A_l(y)], l=1,\ldots,q\}$ is included in ${\rm Span }\{[\widetilde{A}_0(z),\widetilde{A}_{q+j}(z)], j=1,\ldots,q\}$. Owing to \eqref{hormand_assump}, it follows that
${\rm Span}\{\partial_{y_1},\ldots,\partial_{y_d}\}$ is included in
$$
V={\rm Span}\left\{\widetilde{A}_1(z),\widetilde{A}_2(z),\ldots,\widetilde{A}_q(z),\;\textnormal{Lie Brackets of length$\le N$ of the $\widetilde{A}_j(z)$'s } , 0\le j\le q\right\}.
$$
Now, let us show that ${\rm Span}\{\partial_{x_1},\ldots,\partial_{x_d}\}$ is included in $V$. Since ${\rm Span}\{\partial_{y_1},\ldots,\partial_{y_d}\}$ is included in $V$, it is clear that for every $x\!\in\ER^d$, $A_j(x)=A_j^{(\rho)}(y)-\widetilde{A}_j(z)$ also belongs to $V$. Since 
$$
[\widetilde{A}_0(z),\widetilde{A}_j(z)]=[A_0(x),A_j(x)]+[A_0(y),A_j^{(\rho)}(y)],
$$
$[A_0(x),A_j(x)]$ has the same property. Using again \eqref{hormand_assump}, we deduce that ${\rm Span}\{\partial_{x_1},\ldots,\partial_{x_d}\}$ is included in $V$  and thus that $\dim (V)=2d$.
As a consequence, for every $z\in \ER^{d}\times\ER^{d} $ and $t>0$, $Q_t^{(\rho)}(z,.)$ admits a density $q_t(z,.)$ $w.r.t.$ $\lambda_{2d}$ such that 
$(z,z')\mapsto q_t(z,z')$ is ${\cal C}^\infty$ on $\ER^{d}\times\ER^{d} \times\ER^{d}\times\ER^{d} $.\\
In order to obtain uniqueness for the invariant distribution, it remains to show that there exists $T>0$ such that for every $z\in\ER^{d}\times\ER^{d} $, for every non-empty open set $O$ of $\ER^{d}\times\ER^{d} $, $Q_T(z,O)>0$. Owing to \eqref{defcontrolable}, it is clear that 
for every $z_1=(x_1,y_1)$ and $z_2=(x_2,y_2)$, for every $\varepsilon>0$, there exist $u$ and $\widetilde{u} \in L^2([0,T],\ER^d)$ such that $z(t)=(x^{(u)}(t),x^{(\widetilde{u})}(t))$, where $x^{(u)}$ and $x^{(\widetilde{u})}$ are solutions to \eqref{deterministic_cs} starting from $x_1$ and $y_1$, satisfies $|z(T)-z_2|\le \varepsilon$. Furthermore, since $S$ is invertible, we can assume that $\widetilde{u}=\rho u+ S \omega$ with $\omega\in L^2([0,T],\ER^d)$. Then, the support Theorem can be applied to obtain that for every $z_1, z_2, \varepsilon$ $Q_T(z_1,B(z_2,\frac{\varepsilon}{2})>0$  and thus to conclude that for every $z\in\ER^{d}\times\ER^{d} $ and every non-empty open set $O$, $Q_T(z,O)>0$. 
\end{proof}

\section{Additional proofs about the two-dimensional counterexample}
\textbf{Proof of \eqref{rayontend}:} For the sake of completeness, we show that $r_t\rightarrow 1$ $a.s.$
as soon as $r_0>0$. First, note that uniqueness holds for the solution of the SDE~(\ref{eq:countex1}) since the coefficients are Lipschitz continuous.
In particular, $(r_t^1)$ defined $a.s.$ by $r_t^1=1$ for every $t\ge0$ is the unique solution starting from $r_0=1$.  Owing to the strong Markov property, this implies that if  $\tau^1:=\inf\{t\ge0, r_t=1\}$, then $r_t=1$  on  $\{\tau \le t\}$. The same property   holds
at $0$. We deduce that $(r_t)_{t\ge0}$ lives in $[1,+\infty)$ if $r_0>1$ and  in $[0,1]$ if $r_0\in[0,1]$.
Moreover, if $r_0>1$, we have  $d(r_t-1)=-(r_t-1)(dt+\vartheta dW_t)$ so that 
$$
r_t-1=e^{-(1+\frac{\vartheta^2}{2})t+\vartheta W_t}.
$$
It follows that $\lim_{t\rightarrow+\infty} r_t=1$ since $\lim_{t\rightarrow+\infty} \frac{W_t}{t}=0$ $a.s.$.
Now, if $r_0\in[0,1]$, we have 
$$
d r_t= r_t(1-r_t)(dt+\vartheta dW_t).
$$ 
Thus, $(r_t)$ is a $[0,1]$-valued submartingale. In particular, $r_t$ converges $a.s.$ to a $[0,1]$-valued random variable $r_\infty$. Since 
$$
\forall t\ge 0,\qquad \ES[r_t]=r_0+\ES\Big(\int_0^t r_s(1-r_s)ds\Big),
$$ 
it follows that $ \ES[\int_0^{+\infty} r_s(1-r_s)ds]$ which in turn implies that $\int_0^{+\infty} r_s(1-r_s)ds<+\infty$ $a.s.$ 
As a consequence $\liminf_{t\rightarrow+\infty} r_t(1-r_t)=0$ $a.s.$. The process $(r_t)$ being $a.s.$ convergent to $r_\infty$, it follows that $r_\infty\in\{0,1\}$ $a.s.$. It remains to prove that $\PE(r_\infty=0)=0$.  Denote by $p$ the scale function of $(r_t)$ null at $r=1/2$. For every 
$r\in(0,1)$,
$$p(r)=\int_{\frac{1}{2}}^re^{-\int_{\frac{1}{2}}^{\xi}\frac{2}{\vartheta^2u(1-u)}du}d\xi=\int_{\frac{1}{2}}^r\left(\frac{1-\xi}{\xi}\right)^{\frac{2}{\vartheta^2}}d\xi.$$
As a consequence, if $\vartheta\in(0,\sqrt{2}]$, $\lim_{r\rightarrow+\infty} p(r)=+\infty$. This means that $0$ is a repulsive point 
and that, as a consequence (see $e.g.$ \cite{KATA}, Lemma 6.1 p. 228), 
$$
\forall b\in(0,1)\qquad \PE(\lim_{a\to 0^+}\tau_a<\tau_b):=\lim_{a\rightarrow0^{+}}\PE(\tau_a<\tau_b)=0
$$
where $\tau_a= \inf\{t\ge 0\,|\, r_t = a\}$, $y\!\in [0,1]$. We deduce that $\PE(r_\infty=0)=0$. This completes the proof.$\qquad \Box$

\bigskip
\noindent \textbf{Proof of \eqref{IMC}:} We want to prove that 
$\mu$ is invariant for $(X_t^x,X_t^{x'})$ if and only if $\mu$ can be represented by \eqref{IMC}. First, since the unique invariant distribution of $(X_t^x)$ is ${\lambda}_{S_1}$, it is clear that $\mu={\cal L}(e^{i\Theta_0},e^{i(\Theta_0+V_0)})$
where $\Theta_0$ has uniform distribution on $[0,2\pi]$ and $V_0$ is a random variable with values in $[0,2\pi)$.
One can check that if $V_0$ is independent of $\Theta_0$, $\mu$ is invariant.
Thus, it remains to prove that it is a necessary condition or equivalently that 
$K(\theta,dv):={\cal L}(e^{iV_0}|e^{i\Theta_0}=e^{i\theta})$ does not depend on $\theta$. Denote by
 $(e^{i\Theta_t},e^{i(\Theta_t+V_t)})$ the (stationary) duplicated diffusion starting from $(e^{i\Theta_0},e^{i(\Theta_0+V_0)})$.
Since $\mu$ is invariant, we have for every $t\ge0$ 
$$
{\cal L}(e^{iV_t}|e^{i\Theta_t}=e^{i\theta})=K(\theta,dv)$$
but thanks to the construction, for every $t\ge0$, $\Theta_t=\Theta_0+W_t$ and $V_t=V_0$ (the angular difference between the two coordinates does not change) so that
$$
{\cal L}(e^{iV_t}|e^{i\Theta_t}=e^{i\theta})=\int K(\theta',dv)\rho_t(\theta,d\theta')$$
where $\rho_t(\theta,d\theta')={\cal L}(e^{i(\theta+W_t)})$. But $\rho_t(\theta,d\theta')$ converges weakly to ${\lambda}_{S_1}$
when $t\rightarrow+\infty$. From the two previous equations it follows that $K(\theta,dv)$ does not depend on $\theta$ since
$\forall \theta\ge0$, $K(\theta,dv)=\int K(\theta',dv){\lambda}_{S_1}(d\theta')$.\hfill$\cqfd$

\small

\medskip
\noindent {\sc Acknowledgement:} We thank an anonymous referee and an associate editor for their suggestions which  helped improving the paper.

\bibliographystyle{alpha}

\bibliography{bibli}
\end{document}